\documentclass[10pt]{amsart}

\usepackage[left=0.87 in,right=0.87
in,top= 1.1 in,bottom=1.08 in]{geometry} 
\usepackage{amsmath,amsthm,amssymb,amsfonts}
\usepackage{setspace}
\usepackage{tikz}
\usepackage{mathtools}
\usepackage{nccmath}
\usepackage[mathscr]{euscript}
\usepackage{hyperref}
\usepackage{dsfont}
\usepackage{yhmath}
\usepackage{csquotes}
\usepackage{bbm}
\usepackage{comment}
\usepackage{mathrsfs}
\usepackage{upgreek}
\usepackage{indentfirst}
\usepackage{pgf,tikz}
\usepackage{caption}
\usepackage{enumitem}

\usetikzlibrary{calc}

\newcommand{\indicatrix}{\mathbbm{1}}

\newcommand{\R}{\mathbb{R}}
\newcommand{\C}{\mathbb{C}}

\newcommand{\F}{\mathscr{F}}

\newcommand{\Nl}{\mathscr{N}}
\newcommand{\M}{\mathscr{M}}
\newcommand{\B}{\mathscr{B}}
\newcommand{\A}{\mathscr{A}}

\DeclareMathOperator{\real}{Re}

\newtheorem{theorem}{Theorem}[section]

\newtheorem{lemma}{Lemma}[section]
\newtheorem{proposition}{Proposition}
\theoremstyle{remark}
\newtheorem{remark}{Remark}[section]

\theoremstyle{definition}
\newtheorem{definition}{Definition}[section]

\numberwithin{equation}{section}

\begin{document}
 
\title{Sharp local well-posedness for the Hirota-Satsuma system}
\author{Rafael Deiga}
\address{\textsc{Rafael Deiga}, Center for Mathematical Analysis, Geometry and Dynamical Systems,
Department of Mathematics,
Instituto Superior T\'ecnico, Universidade de Lisboa, 
Av. Rovisco Pais, 1049-001 Lisboa, Portugal}
\email{rafael.deiga.ferreira@tecnico.ulisboa.pt}
\email{rafaeldeiga@hotmail.com}
\thanks{Acknowledgments: 
The author was supported by FCT/Portugal via projects CAMGSD and IST-ID, the FCT PhD grant UI/BD/153714/2022, and the FAPERN public funding call No. 13/2022 for international doctoral studies. Special thanks are due to Simão Correia and Jorge Drumond Silva for their guidance and insightful suggestions.
}
\date{\today}

\begin{abstract}
    We establish sharp local existence results for the Hirota--Satsuma system in $H^k(\R)\times H^s(\R)$, depending on the ratio between the dispersion of the components. These theorems significantly generalize previous works, which were restricted to the diagonal case of equal regularity $s=k$. Furthermore, we extend the known global well-posedness theory to the off-diagonal regime. The argument relies on the Fourier restriction norm method coupled with the concept of integrated-by-parts strong solution -- a framework that generalizes the classical notion of strong solution.

\end{abstract}

\maketitle

\section{Introduction}

    \subsection{The problem and literature review}    
        In this work, we investigate the initial value problem (IVP) for the Hirota--Satsuma system,
        \begin{equation}\label{equation hirota system}\tag{H-S}
            \left\{ 
            \begin{aligned} 
                u_t+au_{xxx} &= \beta \partial_x(u^2) + \gamma \partial_x(v^2),\\
                v_t +v_{xxx}  &= \theta uv_x,
            \end{aligned} \right. \quad (t,x) \in \R \times \R,
        \end{equation}
        subject to initial data $(u_0,v_0) \in H^k(\R) \times H^s(\R)$, where $a \in \R \setminus \{0,1\}$ and $\beta, \gamma, \theta \in \C \setminus \{0\}$. Our objective is to characterize the precise region in the $(k,s)$-plane for which local well-posedness (LWP) holds and to establish a global well-posedness (GWP) result.

        The Hirota--Satsuma system is a coupled model of KdV-type equations. Physically, it describes nonlinear interactions of long waves with distinct dispersion relations, appearing in fluid mechanics (e.g., shallow water and stratified fluids), solid-state physics, plasma physics, and thermodynamics (see \cite{hirota-satsuma, martinez2024feedbackboundarystabilizationhirotasatsuma, sadeeq, Zedan2009NewAF}). It was originally proposed by Hirota and Satsuma in \cite{hirota-satsuma} in the following form:
        \begin{equation}\label{system proposed by hirota-satsuma}
            \left\{ 
            \begin{aligned} 
                u_t+au_{xxx} &= -6a uu_x + \gamma vv_x\\
                v_t +v_{xxx}  &= -3 uv_x
            \end{aligned} \right.; (t,x) \in \R \times \R.
        \end{equation}
        
        The IVP for the Hirota--Satsuma system has been extensively studied for data $(u_0,v_0) \in H^s(\R) \times H^s(\R)$. This scenario, where both components possess identical regularity, is hereafter referred to as the \emph{diagonal} case. The first available LWP result was established by Feng \cite{feng}, who proved LWP for system \eqref{system proposed by hirota-satsuma} for $s\geq 1$ in the diagonal case, assuming $a\neq 1$ and $\gamma >0$. This was subsequently improved by Alvarez and Carvajal \cite{alvarez} for $s>3/4$. Related results in weighted spaces were also obtained by Muñoz and Pastor \cite{Munoz}.
        
        The current state-of-the-art is due to Yang and Zhang \cite{Yang}, who employed the Fourier restriction norm method introduced by Bourgain \cite{Bourgain1993, Bourgain1993-part2} to obtain the results summarized in Table \ref{table LWP Yang}. Notably, the regularity threshold depends crucially on the dispersion ratio $a$. The authors also proved that the associated bilinear estimates are sharp with respect to $s$.
        
        \begin{table}[ht]
            \begin{minipage}{0.47\textwidth}
                \centering
                \renewcommand{\arraystretch}{1.5}
                \begin{tabular}{|c|c|c|}
                \hline
                Coefficient $a$ & $s$ \\
                \hline
                $a \in (-\infty,1/4) \setminus \{0\}$ & $s > -3/4$ \\
                \hline
                $a \in (1/4,\infty) \setminus \{1\}$ & $s \geq 0$ \\
                \hline
                $a = 1$ & $s > 0$ \\
                \hline
                $a = 1/4$ & $s \geq 3/4$ \\
                \hline
                \end{tabular}
                \captionsetup{hypcap=false}
                \captionof{table}{LWP results from \cite{Yang}.}\label{table LWP Yang}
        \end{minipage}
        \begin{minipage}{0.47\textwidth}
            \centering
            \renewcommand{\arraystretch}{1.7}
            \begin{tabular}{|c|c|c|}
            \hline
            Coefficient $a$ and $\gamma$ & $s$ \\
            \hline
            $a \not \in \{1/4,1\}$, $\gamma >0$ & $s \geq 0$ \\
            \hline
            $a =1/4$, $ \gamma>0$ & $s \geq 1$ \\
            \hline
            \end{tabular}
            \captionsetup{hypcap=false}
            \captionof{table}{GWP results from \cite{Yang}.}\label{table GWP Yang}
            \end{minipage}
        \end{table}

        Regarding ill-posedness, Alvarez and Carvajal \cite{alvarez} established that the IVP for \eqref{system proposed by hirota-satsuma} is ill-posed in $H^k(\R) \times H^s(\R)$ for $k \in [-1,-3/4)$ and $s \in \R$, in the sense that the flow map fails to be uniformly continuous for $a\neq 0$. This result follows from the observation that, for initial data of the form $(u_0,0)$, the system reduces to the KdV equation, which exhibits the same type of ill-posedness in this regime (cf. \cite{Christ2002AsymptoticsFM}). Similarly, the diagonal case is known to be $C^3$-ill-posed for $s<-3/4$ according to \cite[Section 6]{Bourgain1997PeriodicKD}. Consequently, in the diagonal case, the results in Table \ref{table LWP Yang} are sharp for $a \in (-\infty,1/4) \setminus \{0\}$ given the analyticity of the flow for solutions obtained through fixed point methods. However, a gap remains for the case $a \geq 1/4$.

        As for the global theory, we observe that \eqref{system proposed by hirota-satsuma} conserves the mass 
        \begin{equation}\label{equation mass}
            M = \int (\theta u^2- 2\gamma v^2)\,dx,
        \end{equation}
        while \eqref{equation hirota system} additionally preserves the energy
        \begin{equation}\label{equation energy}
            E = \int (1-a)u^2_x + \gamma v^2_x-2(1-a)u^3- \gamma uv^2\, dx.
        \end{equation}
        
        By employing \eqref{equation mass} alongside energy-type estimates, Feng \cite{feng} concluded that the $H^s$-norm of the solution remains bounded and, consequently, established GWP for \eqref{system proposed by hirota-satsuma} in the diagonal case for $s\geq 1$, assuming $\gamma >0$ and $0<a<1$. Subsequently, Yang and Zhang \cite{Yang} improved this result by leveraging refined LWP estimates in conjunction with the conservation laws \eqref{equation mass} and \eqref{equation energy} (with $\theta = -3$). Their conclusions are summarized in Table \ref{table GWP Yang}.

        To date, LWP results for the off-diagonal case remain unavailable. In recent years, new techniques— specifically frequency-restricted estimates (FREs) and normal form reductions in the context of well-posedness—have been developed (cf. \cite{ correia2025sharplocalwellposednessschrodingerkortewegde, correia2023sharp}). The primary objective of the present work is to leverage this framework to address the off-diagonal regime.

    \subsection{Main results} 
    
        We now state our main results of this paper. Let $a  \in \R \setminus \{0,1\}. $ We define the following admissible regularity regions in the $(k,s)$-plane (cf. Figures \ref{figure LWP a>=1/4} and \ref{figure LWP a<1/4}):
        \begin{equation}\label{equation definition A_a}
            \A_{a} \coloneqq 
            \begin{cases}
                \displaystyle \left \{(k,s)\in \R^2; \text{ }  k>-\frac{3}{4}, \text{ } \max \left\{-\frac{3}{4},\frac{k}{2}-\frac{3}{4}, k-2\right\}< s< k+3  \right \},  &\text{if } a < 1/4, \\
                \displaystyle \left \{(k,s)\in \R^2; \text{ }  k\geq\frac{3}{4}, \text{ } s \geq \frac{k}{2}+\frac{3}{8}, \text{ }  k-2< s< k+3  \right \},  &\text{if } a=1/4, \\
                \displaystyle \left \{(k,s)\in \R^2; \text{ }  k\geq0, \text{ } s \geq \frac{k}{2}, \text{ }  k-2< s< k+3  \right \},   &\text{if } a >1/4.
            \end{cases}
        \end{equation}

        The following theorem presents a summarized version of our main LWP result. A precise statement involving the notion of integrated-by-parts strong solutions is provided in Theorem \ref{theorem improved LWP} (cf. Definition \ref{definition IBPS}).
        
        \begin{theorem}\label{theorem general LWP}
            Let $a \in \R \setminus \{0,1\}$ and $(k,s) \in \A_a$. Then the system \eqref{equation hirota system} is analytically locally well-posed in $H^k(\R) \times H^s(\R)$. 
        \end{theorem}

        The proof of Theorem \ref{theorem general LWP} proceeds as follows. First, we establish LWP for $(k,s)$ within the region
        \begin{equation}\label{equation definition A^0}
            \A^0_{a} \coloneqq 
            \begin{cases}
                \displaystyle  \left \{(k,s)\in \R^2; \text{ } k>-\frac{3}{4}, \text{ } \max \left\{\frac{k}{2}-\frac{3}{8}, k-\frac{3}{2}\right\}< s< k+\frac{5}{2}  \right \},  &\text{if } a < 1/4, \\
                \displaystyle \left\{(k,s)\in \R^2; \text{ } k\geq \frac{3}{4}, \quad s\geq \frac{k}{2}+\frac{3}{8},\quad k-\frac{3}{2}< s< k+\frac{5}{2}  \right \},  &\text{if } a=1/4, \\
                \displaystyle \left \{(k,s)\in \R^2; \text{ } k\geq 0, \quad s\geq \frac{k}{2},\quad k-\frac{3}{2}< s< k+\frac{5}{2}  \right \},   &\text{if } a>1/4.
            \end{cases}
        \end{equation}
        which corresponds to the blue regions in Figures \ref{figure LWP a>=1/4} and \ref{figure LWP a<1/4}. Once the bilinear estimates associated with \eqref{equation hirota system} are obtained via frequency-restricted estimates, the proof follows from a standard application of the contraction mapping principle. To extend LWP to the gray regions, we adopt the approach in \cite{correia2025sharplocalwellposednessschrodingerkortewegde}: we first construct a solution at a lower regularity within $\A^0_a$ using a fixed-point argument. Subsequently, by employing the nonlinear smoothing effect provided by integration by parts and FREs, we show that this solution attains the higher regularity of the initial data.

        \begin{figure}[ht]
        \centering
        \begin{minipage}{0.47\textwidth}
            \centering
            \begin{tikzpicture}[scale= 0.6]       
            
                \draw[->] (-1,0) -- (10,0) node[right] {$k$};
                \draw[->] (0,-0.5) -- (0,7) node[above] {$s$};
            
                \begin{scope}
                    \fill[gray, opacity=0.5] (0,0) -- (0,3) --(4,4+3) -- (9,9-2) --(4,4 -2)  --cycle;
                \end{scope}
            
                \begin{scope}
                    \fill[blue!60, opacity=0.6] (0,0) -- (0,5/2) --(9/2,9/2+5/2) --(17/2, 17/2-3/2) -- (5,5-3/2) --(3,3 -3/2)  --cycle;
                \end{scope}
            
                \begin{scope}
                    \fill[red, opacity=0.4] (0,0) -- (0,3) --(4,4+3) -- (-1,7) -- (-1,-1)-- (10,-1) -- (10,7)  -- (9,7) -- (4,4-2)-- cycle;
                \end{scope}
            
                \draw[ domain=0:7, yellow] plot (\x, {\x}) 
                    node[above right] {};
            
                \path let \p1 = (0,0), \p2 = (7, 7) in node[above, rotate={atan2((\y2-\y1),\x2-\x1)}] at (7/2,7/2) { \color{yellow}{$s = k$}};
            
                \draw[dashed,domain=0.1:4, black,line width=1pt] plot (\x, {\x + 3});
            
                \path let \p1 = (0,3), \p2 = (4, 7) in node[above=4pt, rotate={atan2((\y2-\y1),\x2-\x1)}] at (2,5) {$s = k + 3$};
            
                \path let \p1 = (0,5/2), \p2 = (9/2, 7) in node[below=-3pt, rotate={atan2((\y2-\y1),\x2-\x1)}] at (9/4,19/4) { $s = k + 5/2$};
            
                \draw[dashed, domain=4.03:9, black,line width=1pt] plot (\x, {\x - 2});
            
                \path let \p1 = (3,3/2), \p2 = (17/2, 7) in
                node[above=-4pt, rotate={atan2((\y2-\y1),\x2-\x1)}] at (23/4,17/4) { $s = k -3/2 $};
            
                \path let \p1 = (4,2), \p2 = (9, 7) in
                node[below, rotate={atan2((\y2-\y1),\x2-\x1)}] at (13/2,9/2) { $s = k -2 $};
            
                \path let \p1 = (0,0), \p2 = (4,2) in
                node[below, rotate={atan2((\y2-\y1),\x2-\x1)}] at (2,1) {$s = k/2 $};
    
                \draw[line width=1.5 pt] (0,0) -- (0,2.98);
    
                \draw[line width=1pt] (0,3) circle (2.5pt);
    
                \draw (0,0)[line width=1.5 pt] -- (3.98,3.98/2);
    
                \draw[line width=1pt] (4,2) circle (2.5pt);
                
            \end{tikzpicture}
            \end{minipage}
            \begin{minipage}{0.47\textwidth}
            \centering
                \begin{tikzpicture}[scale= 0.6] 
                    \draw[->] (-1,0) -- (10,0) node[right] {$k$};
                    \draw[->] (0,-0.5) -- (0,7) node[above] {$s$};
            
                    
                    \draw[dashed,domain=3/4:4, black, line width=1pt] plot (\x, {\x + 3}) 
                        node[above] {};
                    
                    \draw[dashed, domain=19/4:9, black, line width=1pt] plot (\x, {\x - 2}) 
                        node[above right] {};
            
                    
            
                    \begin{scope}
                        \fill[gray, opacity=0.5] (3/4,3/4) -- (3/4,3/4+3) --(4,7)-- (9,7) --(19/4,19/4 -2)  --cycle;
                    \end{scope} 
            
                    \begin{scope}
                        \fill[blue!60, opacity=0.6] (3/4,3/4) -- (3/4,3/4+5/2) --(9/2,9/2+5/2) --(17/2, 17/2-3/2)  --(15/4,15/4-3/2)  --cycle;
                    \end{scope}
            
                    \begin{scope}
                        \fill[red, opacity=0.4] (3/4,3/4) -- (3/4,3/4+3)  --(4,7) --(-1,7) -- (-1,-1) --(10,-1) --(10,7) --(9,7) --(19/4,19/4 -2)  --cycle;
                    \end{scope} 
            
                    \draw[domain=3/4:7, yellow] plot (\x, {\x});
            
                    \draw[dashed,black] (3/4,3/4) -- (3/4,0)          node[below, yshift= -1pt] {$\frac{3}{4}$};
            
                    \path let \p1 = (3/4,3/4+3), \p2 = (4,7) in
                    node[above=2pt, rotate={atan2((\y2-\y1),\x2-\x1)}] at (\x1/2+\x2/2,\y1/2+\y2/2) {$s = k +3$};
                    
                    \path let \p1 = (3/4,3/4+5/2), \p2 = (9/2,7) in
                    node[below=-4pt, rotate={atan2((\y2-\y1),\x2-\x1)}] at (\x1/2+\x2/2,\y1/2+\y2/2) {$s = k+5/2 $};
            
                    \path let \p1 = (3/4,3/4), \p2 = (7,7) in
                    node[above, rotate={atan2((\y2-\y1),\x2-\x1)}] at (\x1/2+\x2/2,\y1/2+\y2/2) {\color{yellow}{$s = k$}};
            
                    \path let \p1 = (9/4,9/4-3/2), \p2 = (17/2,17/2-3/2) in
                    node[above=-6pt, rotate={atan2((\y2-\y1),\x2-\x1)}] at (\x1/2+\x2/2,\y1/2+\y2/2) { $s = k-3/2 $};
            
                    \path let \p1 = (3/4,3/4), \p2 = (19/4,19/4-2) in
                    node[below, rotate={atan2((\y2-\y1),\x2-\x1)}] at (\x1/2+\x2/2,\y1/2+\y2/2) { $s = k/2+3/8 $};
            
                    \path let \p1 = (19/4,19/4-2), \p2 = (9, 7) in
                    node[below, rotate={atan2((\y2-\y1),\x2-\x1)}] at (\x1/2+\x2/2,\y1/2+\y2/2) { $s = k -2 $};
        
                    \draw[line width=1.5 pt] (3/4,3/4) -- (19/4-0.02,19/4-0.02-2);
        
                    \draw[line width=1pt] (19/4,19/4-2) circle (2.5pt);
        
                    \draw (3/4,3/4)[line width=1.5 pt] -- (3/4,3/4+3-0.02);
        
                    \draw[line width=1pt] (3/4,3/4+3) circle (2.5pt);
                    
                \end{tikzpicture}
            
            \end{minipage}
    
            \caption{Regularity regions for LWP with $a \in (1/4,\infty)\setminus\{1\}$ (left) and $a=1/4$ (right) for initial data $(u_0,v_0) \in H^k\times H^s$. The yellow line represents the results by Yang and Zhang \cite{Yang}. The set $\A_a^0$ corresponds to the blue area, including a portion of the continuous black line, while $\A_a$ comprises both the blue and gray regions, inclusive of the entire continuous black line. In the red zone, the problem is $C^2$-ill-posed (Theorem \ref{theorem ill-posedness}).} 
            \label{figure LWP a>=1/4}
        \end{figure}
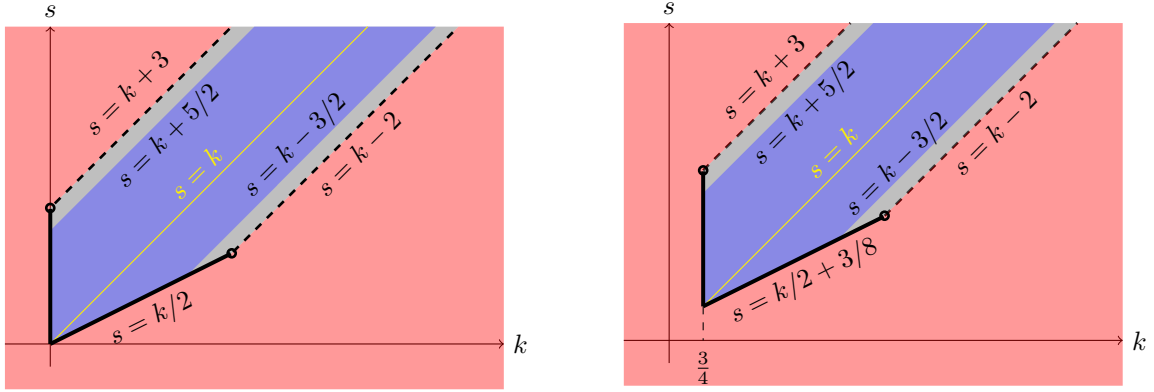
    
        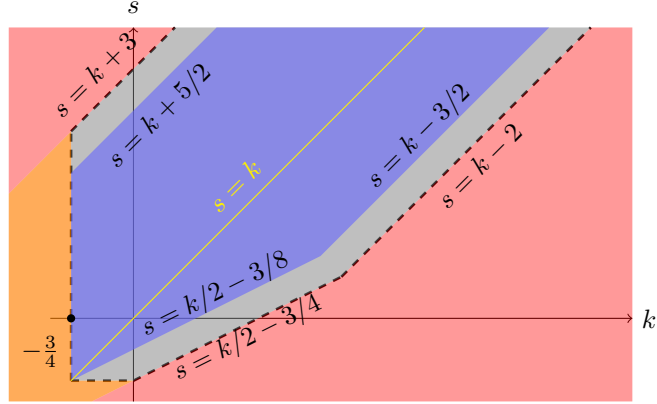
\begin{figure}[ht]
            \centering
            \begin{tikzpicture}[scale=1.1]       
                \draw[->] (-1,0) -- (6,0) node[right] {$k$};
                \draw[->] (0,-1) -- (0,3.5) node[above, yshift= 2pt] {$s$};
                
                \draw[dashed,black, line width=1pt] (-3/4,-3/4) -- (-3/4,-3/4+3)          node[left] {}; 
    
                \draw[dashed,black, line width=1pt] (-3/4,-3/4) -- (0,-3/4)          node[left] {}; 
                \draw[dashed,domain=-3/4:0.5, black, line width=1pt] plot (\x, {\x + 3}) 
                    node[above] {};
                
                \draw[dashed, domain=5/2:5.5, black, line width=1pt] plot (\x, {\x - 2}) 
                    node[above right] {};
                \draw[dashed,domain=0:5/2, black, line width=1pt ] plot (\x, {\x/2 - 3/4}) 
                    node[below right] {};
                
                \begin{scope}
                    \fill[gray, opacity=0.5] (-3/4,-3/4) -- (-3/4,-3/4+3) --(0.5,0.5+3)  -- (5.5,5.5-2) --(5/2,5/2 -2) -- (0,-3/4) --cycle;
                    
    
                    \fill[blue!60, opacity=0.6] (-3/4,-3/4) -- (-3/4,-3/4+5/2) --(1,1+5/2)  -- (5,3.5) --(9/4,9/4 -3/2) --cycle;
    
                    \fill[orange, opacity=0.65](-1/2,-1/4 -3/4) --(0,-3/4) -- (-3/4,-3/4) -- (-3/4,-3/4+3) --(-1.5,-1.5+3)  --  (-1.5,-3/4)  -- (-1.5,-1)--cycle;
    
                    \fill[red, opacity=0.4]  (-1.5,-1.5+3) --(0.5,0.5+3)  -- (-1.5, 3.5) -- cycle;
    
                    \fill[red, opacity=0.4] (6,-1/4-3/4) -- (6,3.5) --(5.5,5.5-2) --(5/2,5/2 -2) -- (-1/2,-1/4 -3/4) --cycle;
                     
                \end{scope}
    
                \fill (-3/4,0) circle (1.4 pt);
                \node[below left, yshift=-3pt, xshift= -1pt] at (-3/4,0) {$-\frac{3}{4}$};
    
                \draw[domain=-3/4:3.5, yellow] plot (\x, {\x});
    
                \path let \p1 = (-3/4,-3/4+3), \p2 = (0.5,0.5+3) in
                node[above=2pt, rotate={atan2((\y2-\y1),\x2-\x1)}] at (-3/4+0.45,-3/4 +0.45 +3) {$s = k+3 $};
    
                \path let \p1 = (-3/4,-3/4+5/2), \p2 = (1,1+5/2) in
                node[below, rotate={atan2((\y2-\y1),\x2-\x1)}] at ((\x1/2+\x2/2,\y1/2+\y2/2) {$s = k+5/2 $};
    
                \path let \p1 = (-3/4,-3/4), \p2 = (1+5/2,1+5/2) in
                node[above=2pt, rotate={atan2((\y2-\y1),\x2-\x1)}] at ((\x1/2+\x2/2,\y1/2+\y2/2) {\color{yellow}{$s = k$}};
    
                \path let \p1 = (-3/4,-3/4), \p2 = (9/4,9/4-3/2) in
                node[above= -4pt, rotate={atan2((\y2-\y1),\x2-\x1)}] at (10+\x1/2+\x2/2 ,5+\y1/2+\y2/2) {$s = k/2-3/8 $};
    
                \path let \p1 = (9/4,9/4-3/2), \p2 = (5,5-3/2) in
                node[above=-4pt, rotate={atan2((\y2-\y1),\x2-\x1)}] at ((\x1/2+\x2/2,\y1/2+\y2/2) { $s = k-3/2 $};
    
                \path let \p1 = (0,-3/4), \p2 = (5/2,5/2-2) in
                node[below= -4pt, rotate={atan2((\y2-\y1),\x2-\x1)}] at ((\x1/2+\x2/2,\y1/2+\y2/2) {$s = k/2-3/4 $};
    
                \path let \p1 = (5/2,5/2-2), \p2 = (5.5,5.5-2) in
                node[below , rotate={atan2((\y2-\y1),\x2-\x1)}] at ((\x1/2+\x2/2,\y1/2+\y2/2) {$s = k-2 $};
            \end{tikzpicture}
            
            \caption{For $a \in (-\infty, 1/4)\setminus\{0\}$, the yellow line represents the LWP results established by Yang and Zhang \cite{Yang}. The set $\A_a^0$ corresponds to the blue region, while $\A_a$ comprises both the blue and gray areas. In the red and orange regions, the problem is $C^2$-ill-posed and $C^3$-ill-posed (for $a \neq -1/8$), respectively; see Theorem \ref{theorem ill-posedness} and Remark \ref{remark ill-posedness for a=-1/8}.} 
            \label{figure LWP a<1/4}
            
        \end{figure}

        By analyzing the multilinear estimates required to establish Theorem \ref{theorem general LWP}, we can identify the specific frequency interactions responsible for the constraints defining $\A_a$. These interactions enable the construction of counterexamples that yield the following:

        \begin{theorem}[Ill-posedness]\label{theorem ill-posedness}
            Let $(k,s) \in \R^2 \setminus \overline{\A_a}$. The system \eqref{equation hirota system} is $C^2$-ill-posed in $H^k(\R) \times H^s(\R)$ for $a \in [1/4, \infty) \setminus \{1\}$ and $C^3$-ill-posed in $H^k(\R) \times H^s(\R)$ for $a \in (-\infty, 1/4) \setminus \{-\frac{1}{8},0 \}$.
        \end{theorem}
        
        \begin{remark}\label{remark ill-posedness for a=-1/8}
            More precisely, for $a \in (-\infty, 1/4) \setminus \{0\}$, system \eqref{equation hirota system} is $C^2$-ill-posed in $H^k(\R) \times H^s(\R)$ whenever $s > k+3$ or $s < \min \{k/2 - 3/4, k-2, -1\}$. Consequently, for $a = -1/8$, the behavior of the system remains an open question in the interior of the region 
            \begin{equation}\label{missing region for a=-1/8}
                \left\{(k,s) \in \R^2 : k > -\frac{3}{4}, \, \min \left\{\frac{k}{2}-\frac{3}{4}, -1\right\} < s < -\frac{3}{4} \right\}.
            \end{equation}
        \end{remark}
        
        \begin{remark}
            In the case $a=-1/8$, we conjecture that $C^3$-ill-posedness also holds for $(k,s) \notin \overline{\A_{a}}$. The restriction $a \neq -1/8$ arises from the specific choice of counterexample and is likely a technical artifact. For further details, see the proof of Lemma \ref{lemma ill-posedness s<-3/4}.
        \end{remark}
        
        By exploiting the GWP framework established in \cite{correia2025sharplocalwellposednessschrodingerkortewegde}, we obtain:
        
        \begin{theorem}[GWP]\label{theorem GWP}
            Let $(k,s) \in \A_a$. If $a \in \R \setminus \{0,1,1/4\}$, $\gamma\theta < 0$, and $k,s \geq 0$, then system \eqref{equation hirota system} is globally well-posed in $H^k(\R) \times H^s(\R)$. Furthermore, system \eqref{system proposed by hirota-satsuma} with $a=1/4$ is globally well-posed in $H^k(\R) \times H^s(\R)$ provided that $k,s \geq 1$, $(k,s) \in \A_{1/4}$, $\gamma > 0$, and $\theta < 0$.
        \end{theorem}
        
        The proof is based on the multilinear estimates established in Subsections \ref{section bilinear estimates} and \ref{section multilinear estimates} to deduce persistence of regularity in $\A_a$. Moreover, the conservation law \eqref{equation mass} ensures GWP in $L^2\times L^2$ (for $a\neq 1/4$), allowing us to conclude GWP for $k,s \geq 0$ and $(k,s) \in \A_a$ via the persistence property. For $a=1/4$, global well-posedness in $H^1 \times H^1$ follows from the conservation laws \eqref{equation mass}--\eqref{equation energy}, and the result is obtained analogously.
        
        \subsection{Overview of the method}

            As previously stated, we first establish LWP for $(k,s) \in \A^0_a$. We apply a Picard iteration scheme to the Duhamel formulation of system \eqref{equation hirota system} via a contraction mapping argument in suitable Fourier restriction spaces. To derive the required multilinear estimates, we use the method of frequency-restricted estimates (FREs) introduced in \cite{correia2023sharp}. For initial data $(u_0,v_0) \in H^k(\R) \times H^s(\R)$, the linear solutions to \eqref{equation hirota system} are denoted by
            \[e^{-at\partial_x^3}u_0 = \F_x^{-1}(e^{iat\xi^3}\F_xu_0) \hspace{2mm}\text{ and } \hspace{2mm} e^{-t\partial_x^3}v_0 = \F_x^{-1}(e^{it\xi^3}\F_xv_0), \]
            where $\F_*$ denotes the Fourier transform with respect to the variable $*$. For $k,s, b \in \R$, the Fourier restriction spaces (or Bourgain spaces, cf. \cite{Bourgain1993, Bourgain1993-part2}) are defined as the closure of the Schwartz space $\mathscr{S} (\R\times\R )$ equipped with the norms
            $$\|u\|_{X^{k,b}} \coloneqq \left( \int_{\R^2} \langle \tau - a\xi^3 \rangle^{2b} \langle \xi \rangle^{2k}|\F_{(x,t)} u(\tau, \xi)|^2 \, d\tau \, d\xi \right)^{1/2},$$
            $$\|v\|_{Y^{s,b}} \coloneqq \left( \int_{\R^2} \langle \tau - \xi^3 \rangle^{2b} \langle \xi\rangle^{2s}|\F_{(x,t)} v(\tau, \xi)|^2 \, d\tau \, d\xi \right)^{1/2}. $$ 
            
            Within this framework, we develop the LWP theory in $\A_0$ adopting the classical notion of strong solutions, proving that the fixed point \begin{equation}\label{equation classical u}
                u(t) = \eta(t)\left(e^{-at\partial_x^3}u_0 +  \int_0^t \eta_T(t') e^{-a(t-t')\partial_x^3} \partial_x(u^2+ v^2) \,dt'\right),
            \end{equation}
            \begin{equation}\label{equation classical v}
                v(t) = \eta(t)\left(e^{-t\partial_x^3}v_0+  \int_0^t \eta_T(t') e^{-(t-t')\partial_x^3} (u v_x) \,dt'\right),
            \end{equation}
            holds for $(u,v) \in X^{k,b} \times Y^{s,b}$. Here, $\eta \in C_c^\infty(\R)$ is a cutoff function satisfying $\eta \equiv 1$ on $[-1,1]$, and we set $\eta_T(t) =\eta(t/T)$.
                    
            As detailed in Subsection \ref{preliminaries}, proving the estimates requires analyzing functions of the form:
            \begin{equation}\label{equation phi^u}
                \Phi^u_1 (\xi,\xi_1,\xi_2)\coloneqq -a\xi^3 + \xi_1^3 + \xi_2^3,
            \end{equation}
            \begin{equation}\label{equation phi^v}
                 \Phi^v(\xi,\xi_1,\xi_2) \coloneqq -\xi^3 + a\xi_1^3 + \xi_2^3,
            \end{equation} 
            where $\xi = \xi_1 + \xi_2$. The functions \eqref{equation phi^u} and \eqref{equation phi^v} are referred to as \emph{total phases} and are associated with the coupling terms $\partial_x(v^2)$ and $uv_x$ in \eqref{equation hirota system}, respectively. We now examine the properties of $\Phi_1^u$; analogous results hold for $\Phi^v$ due to the relation
            \begin{equation}\label{equation relation phi^v and phi^u}
                \Phi^v (\xi,\xi_1,\xi_2) = \Phi_1^u (-\xi_1,-\xi,\xi_2).
            \end{equation}
            The phase $\Phi_1^u$ admits the factorization
            \begin{equation}\label{equation factorization phi^u}
                \Phi_1^u = \xi(\xi_1^2 - \xi_1\xi_2 + \xi_2^2 - a\xi^2) = \xi(3\xi_1^2 - 3\xi\xi_1 + (1-a)\xi^2).
            \end{equation}
            The quadratic factor $3\xi_1^2 - 3\xi\xi_1 + (1-a)\xi^2$ does not vanish for $a < 1/4$. Specifically, in this regime, we have the lower bound\footnote{Given $f,g \ge 0$, the notation $f \lesssim g$ denotes the existence of a universal constant $C>0$ such that $f \le C g$, whereas $f \sim g$ signifies that $f \lesssim g \lesssim f$. We write $f \ll g$ if $f \le \eta g$ for a sufficiently small $\eta > 0$. Finally, the relation $f \simeq g$ means that $|f - g| \ll \max\{|f|, |g|\}$.} 
            $$|3\xi_1^2 - 3\xi\xi_1 + (1-a)\xi^2| \gtrsim |\xi|^2,$$ 
            which implies
            \begin{equation}\label{inequality phase a<1/4}
                |\Phi^u_1| \gtrsim |\xi|^3.
            \end{equation}
            If $a \geq 1/4$, the total phase factorizes as
            \begin{equation}\label{equation factorization phi^u a >=1/4}
                \Phi^u_1 = 3\xi(\xi_1 - \mu_a\xi)(\xi_1 - (1-\mu_a)\xi),
            \end{equation}
            where $\mu_a = \frac{1}{2} + \frac{\sqrt{3(4a-1)}}{6}$. Note that $\Phi_1^u$ vanishes within the region $|\xi| \sim |\xi_1|\sim|\xi_2|$, since $\xi_1 \simeq \mu_a\xi$ implies $|\xi_2| \simeq |(1-\mu_a)\xi| \gtrsim |\xi|$ for $a \neq 1$. Moreover, the threshold case $a=1/4$ yields a double root:
            \[ \Phi_1^u = \xi\left(\xi_1 - \frac{\xi}{2}\right)^2. \]
            Since the existence of the analytic flow is sensitive to the structure of the resonance sets (i.e., where $\Phi_1^u$ or $\Phi^v$ vanish), the regularity regions for local well-posedness and ill-posedness depend significantly on the value of $a$. This heuristic underlies the three distinct parameter regimes observed in Theorems \ref{theorem general LWP} and \ref{theorem ill-posedness}.

        \vspace{2mm}

        In Section \ref{section LWP in A_a^0}, we establish sharp bilinear estimates for the coupling terms. Combined with the known estimate for the KdV equation, LWP for $(k,s)\in \A^0_a$ follows via a standard fixed-point argument.   

        While these bilinear estimates are sharp, they do not necessarily determine the optimal LWP region; such estimates may fail for certain $k$ and $s$ even if the continuous flow exists at that regularity level. Indeed, such a failure may merely reflect a limitation of the chosen functional space for the solution. In fact, in Sections \ref{section normal form} and \ref{section well posedness in A_a}, we extend LWP to the larger region $\A_a$. Although we obtain a solution $(u,v) \in C_tH^k_x \times C_tH^s_x$, we do not construct it directly in $X^{k,b} \times Y^{s,b}$. The underlying strategy is as follows: if an obstruction to the bilinear estimate for $\partial_x(v^2)$ arises where $|\Phi_1^u|$ is large, we integrate the coupling term in \eqref{equation classical u} by parts in time to mitigate this issue. This approach yields sufficient decay to conclude LWP for $(k,s) \in \A_a$, with an analogous procedure applied to $uv_x$. 
        
        This technique of employing integration by parts in time to enhance the properties of nonlinear dispersive equations traces back to the work of Shatah \cite{shatah}, who derived energy estimates to prove global existence for a class of nonlinear Klein-Gordon equations with sufficiently small initial data. Since then, this approach---often referred to as a normal form reduction---has been applied in various contexts, including local well-posedness \cite{correia2025sharplocalwellposednessschrodingerkortewegde,oh}, global well-posedness \cite{babin}, global existence \cite{germain}, unconditional uniqueness \cite{CHUNG20171273,Guo2013,kishimoto2021unconditionaluniquenesssolutionsnonlinear, kwon}, and blow-up stability \cite{correia2024sharpblowupstabilityselfsimilar}.

        \vspace{2mm}
        
        The paper is organized as follows: Section \ref{preliminaries} lists the necessary frequency-restricted estimates and auxiliary results. Section \ref{section LWP in A_a^0} proves the sharp bilinear estimates for the coupling terms in \eqref{equation hirota system} and establishes LWP in $\A_a^0$. In Section \ref{section normal form}, we derive and define the notion of an integrated-by-parts strong solution. Section \ref{section well posedness in A_a} is devoted to showing the required multilinear estimates and extending the LWP region to $\A_a$.  Finally, Section \ref{section ill-posedness} provides the ill-posedness results, proving the sharpness of the LWP thresholds.

        \vspace{2mm}
        
        \noindent \textit{Further notation}. Let $c \in \mathbb{R}$ be a constant. The notation $c^+$ (resp. $c^-$) in the hypotheses of a result denotes the existence of some $\epsilon_0 > 0$ such that the statement holds for every $c' = c + \epsilon$ (resp. $c' = c- \epsilon$) with $\epsilon \in (0, \epsilon_0)$. When considering frequencies $\xi, \xi_1, \xi_2, \dots$, we consistently assume the convolution relation $\xi = \xi_1 + \dots + \xi_n$.
        \section{Preliminaries}\label{preliminaries}
        
        In this section, we reduce the proof of multilinear estimates in Fourier restriction norm spaces to establishing frequency-restricted estimates (FREs). We conclude by providing several auxiliary tools for proving these FREs.
        
        In system \eqref{equation hirota system}, the nonlinearities possess a multilinear structure well-suited for frequency space analysis. Given a multiplier $m: \R^{n+1} \rightarrow \C$, we consider an $n$-linear form $\Nl[g_1, \dots, g_n]$ satisfying

        \begin{equation}\label{equation form of the nonlinearity}
            \F_x \Nl(\xi) = \int_{\xi = \xi_1 + \dots + \xi_n} m(\xi, \xi_1, \dots, \xi_n) \F_x g_1(\xi_1) \dots \F_x g_n(\xi_n) \, d\xi_1 \dots d\xi_{n-1}.
        \end{equation}

       Let $J_u \subset\{ \emptyset,1, \cdots, n \}$ be the set of indices associated with the $u$-factors and $J_v =\{ \emptyset,1, \cdots, n \} \setminus J_u$ the set of indices associated with the $v$-factors. Accordingly, we define the phase $\phi_j$ corresponding to each equation in \eqref{equation hirota system} as 
        $$\phi_j(\xi_j) \coloneqq   \begin{cases}
            a\xi_j^3, & \text{if } j \in J_u, \\
            \xi_j^3, & \text{if } j \in J_v.
        \end{cases}$$
        
        Define the profiles $\tilde{g}_j = e^{-it\phi_j(\xi_j)} \F_x g_j$ and the total phase
        \begin{equation}\label{equation total phase}
            \Phi = -\phi(\xi) + \sum_{j=1}^{n} \phi_j(\xi_j).
        \end{equation}
        
        Furthermore, let 
        $$Z_j \coloneqq \begin{cases}
            X, & \text{if } j \in J_u, \\
            Y, & \text{if } j \in J_v.
        \end{cases} $$
        
        To establish LWP for \eqref{equation hirota system} via a fixed-point argument, the primary challenge lies in proving the multilinear estimates
        \begin{equation}\label{equation multilinear estimate}
            \|\Nl[g_1, \dots, g_n]\|_{Z^{s,b'}} \lesssim \prod_{j=1}^{n} \|g_j\|_{Z_j^{s_j,b}},
        \end{equation}
        where $s, s_j \in \R$, $b = (1/2)^+$, and $b' = (b-1)^+$. The following lemmas provide sufficient conditions for \eqref{equation multilinear estimate} to hold.
        
        \begin{lemma}[{\cite[Lemma 3.1]{correia2025sharplocalwellposednessschrodingerkortewegde}}]\label{lemma CHS}
            Let $n=2$. Suppose that for every $M > 1$ and $\alpha \in \R$,
            \begin{equation}
                 \sup_{\xi} \int_{\xi = \xi_1+ \xi_2} \frac{|m|^2 \langle \xi \rangle^{2s}}{ \langle \xi_1 \rangle^{2s_1}\langle \xi_2 \rangle^{2s_2}} \mathbbm{1}_{|\Phi - \alpha| < M} \, d \xi_1 \lesssim \langle \alpha \rangle^{1^-} M.
            \end{equation}
            Then the estimate \eqref{equation multilinear estimate} holds.
        \end{lemma}

        \begin{remark}
            This lemma remains valid if the supremum is taken over any $\xi_j$ and the integration is performed with respect to another frequency.
        \end{remark}
        
        Following an argument similar to the proof of \cite[Lemma 3.1]{correia2025sharplocalwellposednessschrodingerkortewegde}, we obtain the following result:\footnote{ Note that, in \cite{correia2025sharplocalwellposednessschrodingerkortewegde}, the estimate corresponding to \eqref{inequality CHS alpha xi_j} features $ \langle \alpha \rangle^{1^-} $ instead of $ \langle \alpha \rangle $. However, a closer inspection of the proof reveals that \eqref{inequality CHS alpha xi_j} may be adopted as a hypothesis, provided that $b>1/2$.} 
        
        \begin{lemma}\label{lemma CHS phase comparable to alpha}
            Suppose that, for all $\alpha, M \in \R$ with $|\alpha| \gtrsim M > 1$, we have
            \begin{equation}\label{inequality CHS alpha xi}
                \sup_{\xi} \int \frac{|m|^2 \langle \xi \rangle^{2s}}{ \prod_{j=1}^n \langle \xi_j \rangle^{2s_j}} \mathbbm{1}_{|\Phi - \alpha| < M} \, d \xi_1 \dots d \xi_{n-1} \lesssim \langle \alpha \rangle^{1^-} M,
            \end{equation}          
            and
            \begin{equation}\label{inequality CHS alpha xi_j}
                \sup_{l=1, \dots, n} \left( \sup_{\xi_l} \int \frac{|m|^2 \langle \xi \rangle^{2s}}{ \prod_{j=1}^n \langle \xi_j \rangle^{2s_j}} \mathbbm{1}_{|\Phi - \alpha| < M} \, d \eta_1 \dots d \eta_{n-1} \right) \lesssim \langle \alpha \rangle M,
            \end{equation}
            where $\{\eta_1, \dots, \eta_{n-1}\} = \{\xi_1, \dots, \xi_n\} \setminus \{\xi_l\}$. Then the estimate \eqref{equation multilinear estimate} holds.
        \end{lemma}
        
        In the subsequent lemma, we denote $\Xi \coloneqq \max_{1 \leq j \leq n} \langle \xi_j \rangle$.
        
        \begin{lemma}[{\cite[Lemma 3.5]{correia2025sharplocalwellposednessschrodingerkortewegde}}]\label{lemma Schur's test with weight}
            Let $A$ be a non-empty proper subset of $\{\emptyset, 1, \dots, n\}$ and let $\M_1, \M_2: \R^{n+1} \rightarrow \R^+$. Suppose that
            $$ \M_1 \M_2 = \frac{|m|^2 \langle \xi \rangle^{2s}}{\prod_{j=1}^n \langle \xi_j \rangle^{2s_j}} \quad \text{and} \quad |m| \lesssim \Xi^N,$$
            for some $N \in \mathbb{N}$. Then, if
            $$\sup_{\xi_{j \in A}, \alpha} \int \M_1 \Xi^{0^+} \mathbbm{1}_{|\Phi - \alpha| < M} \, d\xi_{j \notin A} + \sup_{\xi_{j \notin A}, \alpha} \int \M_2 \Xi^{0^+} \mathbbm{1}_{|\Phi - \alpha| < M} \, d\xi_{j \in A} \lesssim M,$$
            the estimate \eqref{equation multilinear estimate} holds.
        \end{lemma}

        For the following lemma, we define 
        $$ \M \coloneqq \frac{|m| \langle \xi \rangle^{s}}{\prod_{j=1}^n \langle \xi_j \rangle^{s_j}}. $$
        Furthermore, let $\eta_1, \eta_2, \eta_3, \eta_4 \in \{\xi, \xi_1, \xi_2, \xi_3\}$ denote the spatial frequency variables. 
       
        \begin{lemma}\label{leminha}
            Assume $n=3$, $|\eta_1| \gtrsim |\eta_2| \gtrsim |\eta_3|$, and $\M \lesssim \langle \eta_1 \rangle^{a_1} \langle \eta_2 \rangle^{a_2} \langle \eta_3 \rangle^{a_3}$. Furthermore, let $\{j_1, j_2, j_3, j_4\} = \{1, 2, 3, 4\}$ be such that
            $$\left| \frac{\partial \Phi}{\partial \eta_{j_1}} \right| \gtrsim \langle \eta_1 \rangle^2 \quad \text{for } \eta_{j_3} \text{ and } \eta_{j_4} \text{ fixed},$$
            $$\left| \frac{\partial \Phi}{\partial \eta_{j_3}} \right| \gtrsim \langle \eta_1 \rangle^2 \quad \text{for } \eta_{j_1} \text{ and } \eta_{j_2} \text{ fixed}.$$ 

            Then inequality \eqref{equation multilinear estimate} holds provided that any of the following conditions is satisfied:
    
            \begin{enumerate}[label=(\roman*)]
                \item\label{item 1} $a_1 < 2$, $a_1 + a_2 < 2$, and $a_1 + a_2 + a_3 < 2$.
                \item\label{item 2} $|\eta_1| \sim |\eta_2|$, $a_1 + a_2 < 2$, and $a_1 + a_2 + a_3 < 2$.
                \item\label{item 3} $|\eta_1| \sim |\eta_2|\sim |\eta_3|$ and $a_1 + a_2 + a_3 < 2$.
            \end{enumerate}
            
        \end{lemma}

        \begin{remark}
            For the applications  of this result, we adopt the notation $A=\{j_1, j_2\}$ and $A^c =\{j_3, j_4\}$.
        \end{remark}
        
        \begin{proof}
            We prove the statement assuming \ref{item 1}. By a change of variables,
            $$ I \coloneqq \sup_{\eta_{j_3}, \eta_{j_4}} \int_{|\Phi-\alpha| < M} \M \Xi^{0^+} \, d\eta_{j_1} \lesssim \sup_{\eta_{j_3}, \eta_{j_4}} \int_{|\Phi-\alpha| < M} \langle \eta_1 \rangle^{a_1-2+0^+} \langle \eta_2 \rangle^{a_2} \langle \eta_3 \rangle^{a_3} \, d\Phi. $$
            Since $|\eta_1| \gtrsim |\eta_2|$ and $a_1-2 < 0$, we have
            $$ I \lesssim \sup_{\eta_{j_3}, \eta_{j_4}} \int_{|\Phi-\alpha| < M} \langle \eta_2 \rangle^{a_1+a_2-2 + 0^+} \langle \eta_3 \rangle^{a_3} \, d\Phi \lesssim M, $$
            owing to the conditions $|\eta_2| \gtrsim |\eta_3|$, $a_1+a_2 < 2$, and $a_1+a_2+a_3 < 2$. An analogous argument applies when $\eta_{j_1}$ and $\eta_{j_2}$ are fixed. By Lemma \ref{lemma Schur's test with weight}, the estimate \eqref{equation multilinear estimate} holds.
        \end{proof}
        
        The following results are instrumental in proving the FREs:
        
        \begin{lemma}[{\cite[Lemma 5]{correia2023sharp}}]\label{lemma integral quadratic}
        Let $M > 0$ and $\alpha \in \R$. Then,
        $$ \int \indicatrix_{|q^2 - \alpha| < M} \, dq \lesssim M^{1/2}. $$
        \end{lemma}

        When the total phase $\Phi$ in \eqref{equation total phase} is stationary—or, more specifically, when its derivative vanishes in a given direction—one may invoke Morse's Lemma:
        
        \begin{lemma}[Morse's Lemma, {\cite[Lemma 2.2, part I]{milnor}}]\label{lemma morse}
        Let $p$ be a nondegenerate critical point of a $C^\infty$ function $f: \R^n \rightarrow \R$. Then, there exists a local coordinate system $(y_1, \dots, y_n)$ in a neighborhood $U$ of $p$ such that $y_j(p) = 0$ for all $j$ and
        $$ f = f(p) + y_1^2 + \dots + y_\lambda^2 - y_{\lambda+1}^2 - \dots - y_n^2 $$
        in $U$, for some $1 \leq \lambda \leq n$.
        \end{lemma}
        
        \begin{lemma}[Morse's Lemma with Parameters, {\cite[Lemma C.6.1]{Hrmander2007TheAO}} ]\label{lemma morse parameters}
        Let $f(x,y)$ ($x \in \R^n, y \in \R^N$) be a real-valued $C^\infty$ function in a neighborhood of $(0,0)$. Assume that $D_x f(0,0) = 0$ and that the Hessian $A \coloneqq D_x^2 f(0,0)$ is non-singular. Then, the equation $D_x f(x,y) = 0$ determines, in a neighborhood of $0$, a $C^\infty$ function $x(y)$ with $x(0) = 0$, such that in a neighborhood of $(0,0)$:
        $$ f(x,y) = f(x(y),y) + \langle Az, z \rangle / 2, $$
        where $z = x - x(y) + O(|x - x(y)|(|x| + |y|))$ is a $C^\infty$ function of $(x,y)$ at $(0,0)$.
        \end{lemma}
       
\section{Well-posedness in \texorpdfstring{$\A_a^0$}{A\_a}}\label{section LWP in A_a^0}

    In this section, we establish local well-posedness within the framework of classical strong solutions \eqref{equation classical u}--\eqref{equation classical v} in $X^{k,b} \times Y^{s,b}$. Once the necessary bilinear estimates for the coupling terms are derived, the local theory follows from a standard fixed-point argument. We conclude by proving the sharpness of these estimates.

    \subsection{Bilinear estimates}\label{section bilinear estimates}
        In this subsection, we prove the bilinear estimates for the nonlinear coupling terms in \eqref{equation classical u} and \eqref{equation classical v} by invoking Lemmas \ref{lemma CHS} and \ref{lemma CHS phase comparable to alpha}. The proof relies on partitioning the frequency domain into suitable subregions where a frequency-restricted estimate holds. We proceed as follows: when the total phase $\Phi$ in \eqref{equation total phase} has a large derivative, we perform a change of variables to integrate with respect to $\Phi$ itself. This procedure yields decay in the spatial frequencies via the Jacobian of the change of variables. We then bound the integrand either by a constant or by $|\Phi|$. In the former case, the integral is controlled by $M$ and Lemma \ref{lemma CHS} applies; in the latter, we invoke Lemma \ref{lemma CHS phase comparable to alpha}, exploiting the hypothesis $|\Phi| \lesssim |\alpha|$. Conversely, when $\Phi$ is stationary---or, more specifically, when its derivative vanishes in a given direction---one may invoke Morse's Lemma, provided the second derivative remains large.
    
        Recall that $b=(1/2)^+$ and $b' = (b-1)^+$. We begin by stating the classical bilinear estimate for the KdV equation:

        \begin{lemma}[\cite{Kenig1996ABE}]\label{lemma kdv estimate}
            For $k > -3/4$, the following estimate holds:
            $$\|\partial_x(u_1u_2)\|_{X^{k,b'}} \lesssim \|u_1\|_{X^{k,b}} \|u_2\|_{X^{k,b}} .$$
        \end{lemma}
        
        \begin{remark}
            This result is sharp with respect to the regularity $k$ (see \cite{Kenig1996ABE,Nakanishi2001CounterexamplesTB}).
        \end{remark}
        
        \begin{lemma}\label{lemma multilinear estimate delx(v1v2)} 
            Let $k \in \R$ and $s > k - 3/2$. Then the estimate
            \begin{equation}\label{equation multilinear estimate del_x(v1v2)}
                \|\partial_x(v_1v_2) \|_{X^{k,b'}} \lesssim \|v_1\|_{Y^{s,b}}\|v_2\|_{Y^{s,b}}
            \end{equation}
            holds in the following cases:
            \begin{enumerate}
                \item $a \in (-\infty, 1/4) \setminus \{0\}$ and $s > k/2 - 3/8$,
                \item $a = 1/4$ and $s \geq k/2 + 3/8$,
                \item $a \in (1/4, \infty) \setminus \{1\}$ and $s \geq k/2 $.
            \end{enumerate}
        \end{lemma}

        \begin{remark}
             The region of validity for this bilinear estimate is sharp (see Proposition \ref{proposition estimate del_x(v_1v_2) is sharp}).
        \end{remark}
        
        \begin{proof}
            According to \eqref{equation total phase}, the total phase for $\partial_x(v_1v_2)$ is $\Phi_1^u$ in \eqref{equation phi^u}. We define
            $$\mathscr{M} = \frac{|\xi|\langle \xi \rangle^k}{\langle \xi_1 \rangle^s \langle \xi_2 \rangle^s}.$$    
           
            By symmetry, we may assume without loss of generality that $|\xi_1| \geq |\xi_2|$, which implies $|\xi_1| \gtrsim |\xi|$. Moreover, if $|\xi_1| \lesssim 1$, all frequencies are bounded and \eqref{equation multilinear estimate del_x(v1v2)} follows directly from Lemma \ref{lemma CHS}. Therefore, throughout the remainder of the proof, we assume $|\xi_1| \gtrsim 1$.
    
            \vspace{2mm}
            \noindent \textbf{Case A:} $|\xi| \gtrsim |\xi_1|$. This implies $|\xi| \sim |\xi_1|$.
    
            \vspace{2mm}
            \noindent \textbf{Subcase A1:} $|\xi| \gg |\xi_2|$. Then $\xi \simeq \xi_1$, and since $a \neq 1$, it follows from \eqref{equation factorization phi^u} that $|\Phi_1^u| \sim |\xi|^3$. In this regime, Lemma \ref{lemma CHS phase comparable to alpha} is applicable. Assuming $M \lesssim |\alpha|$, the condition $|\Phi_1^u - \alpha| < M$ implies $|\Phi_1^u| \lesssim |\alpha|$. We bound
            $$I_1 \coloneqq \sup_{\xi_1} \int_{|\Phi_1^u-\alpha| <M} \mathscr{M}^2 d\xi \sim \sup_{\xi_1} \int_{|\Phi_1^u-\alpha| <M} \frac{|\xi|^{2 + 2k - 2s}}{\langle \xi_2 \rangle^{2s}} d\xi \lesssim \sup_{\xi_1} \int_{|\Phi_1^u-\alpha| <M} \frac{|\xi|^{2k - 2s}}{\langle \xi_2 \rangle^{2s}} d\Phi_1^u,$$
            by a change of variables, since $|\partial_\xi \Phi_1^u| \sim |a\xi^2 - \xi_2^2| \sim |\xi|^2$. If $s \geq 0$, the relation $|\xi|^3 \sim |\Phi_1^u| \lesssim |\alpha|$ yields $|\xi|^{2k - 2s} \langle \xi_2 \rangle^{-2s} \lesssim |\xi|^{2k - 2s} \lesssim |\xi|^{3^-} \lesssim|\alpha|^{1^-}$. Hence, $I_1 \lesssim \langle \alpha \rangle^{1^-} M$. If $s < 0$, a similar argument shows that $I_1 \lesssim \langle \alpha \rangle^{1^-} M$ provided $(2k-4s)/3 < 1$. The same reasoning applies when $\xi_2$ or $\xi$ is fixed, given that $a \neq 1$.
    
            \vspace{2mm}
            \noindent \textbf{Subcase A2:} $|\xi_2| \gtrsim |\xi|$.  In this setting, $|\xi|\sim |\xi_1|\sim |\xi_2|$. We distinguish the following scenarios regarding the parameter $a$:

            \vspace{2mm}
            \noindent $\bullet$ $a \in (-\infty, 1/4) \setminus \{0\}$: in this region, $|\Phi_1^u| \sim |\xi|^3$ whenever $a < 1/4$ by virtue of \eqref{inequality phase a<1/4}, and we invoke Lemma \ref{lemma CHS phase comparable to alpha}. Under the assumption $M \lesssim |\alpha|$, it follows that $|\xi_1|^3 \sim |\Phi_1^u| \lesssim |\alpha|$. Let
            \[
            I_1 \coloneqq \sup_{\xi_1} \int_{|\Phi_1^u-\alpha| <M} |\xi_1|^{2 + 2k - 4s} \, d \xi.
            \]
            We consider two possibilities:
            
            \vspace{2mm}
            \noindent \textbf{(i)} $|\partial_{\xi_1}\Phi_1^u| \gtrsim |\xi|^2$: a change of variables yields
            \begin{equation}\label{inequality calculations derivative large}
            I_1 \lesssim \sup_{\xi_1} \int_{|\Phi_1^u-\alpha| <M} |\xi_1|^{2k - 4s} \, d \Phi_1^u 
            \lesssim \int_{|\Phi_1^u-\alpha| <M} |\Phi_1^u|^{\frac{2k - 4s}{3}} \, d \Phi_1^u 
            \lesssim \langle \alpha \rangle^{1^-} M. 
            \end{equation}
            
            \vspace{2mm}
            \noindent \textbf{(ii)} $|\partial_{\xi_1} \Phi_1^u| \ll |\xi|^2$: setting $p_j = \xi_j/\xi_1$, we observe that
            \begin{equation}\label{equation P morse}
                P(p) \coloneqq \frac{\Phi_1^u}{|\xi_1|^3} = -ap^3 + 1 + p_2^3   
            \end{equation}
            possesses two nondegenerate stationary points at $p = \frac{1}{1 \pm \sqrt{a}}$. By Morse's Lemma (\ref{lemma morse}), there exists a local coordinate transformation $p \mapsto q$ in a neighborhood of each critical point such that $P = q^2 + c$ for some $c \in \mathbb{R}$. Choosing $\alpha' \in \R$ appropriately, it follows from Lemma \ref{lemma integral quadratic} that
            \begin{equation}\label{inequality calculations derivative small}
                I_1 \lesssim \sup_{\xi_1} \int_{|\Phi_1^u-\alpha| <M} |\xi_1|^{3 + 2k - 4s} \, dp 
                \sim \sup_{\xi_1} \int_{|q^2 - \alpha'| < \frac{M}{|\xi_1|^3}} |\xi_1|^{3 + 2k - 4s} \, dq
                \lesssim \sup_{\xi_1} |\xi_1|^{\frac{3}{2} + 2k - 4s} M^{\frac{1}{2}}. 
            \end{equation}
            Since $|\xi_1|^3 \lesssim |\alpha|$, it follows that $I_1 \lesssim \langle \alpha \rangle^{1^-} M^{1/2}$. An analogous estimate holds when $\xi$ or $\xi_2$ is fixed.
            
            \vspace{2mm}
            \noindent $\bullet$ $a = 1/4$: we apply Lemma \ref{lemma CHS}. Taking the supremum in $\xi$ and performing estimates analogous to those in the previous case, we obtain $I_1 \lesssim M$. Specifically, it suffices to observe that $|\xi|^{2k - 4s} \lesssim 1$ in \eqref{inequality calculations derivative large} and $|\xi|^{3/2 + 2k - 4s} \lesssim 1$ in \eqref{inequality calculations derivative small}.
            
            \vspace{2mm}
            \noindent $\bullet$ $a \in (1/4, \infty) \setminus \{1\}$: since $\xi_2 = \xi -\xi_1$, it follows that  $|\partial_\xi \Phi_1^u| \sim |-a\xi^2 + \xi_2^2|$ with $\xi_1$ fixed. Moreover, $|\partial_{\xi_1} \Phi_1^u| \sim |\xi_1^2 - \xi_2^2|$ for $\xi$ fixed. Since the conditions $|-a\xi^2 + \xi_2^2| \ll |\xi|^2$ and $|\xi_1^2 - \xi_2^2| \ll |\xi|^2$ cannot occur simultaneously, we may assume $|\xi_1^2 - \xi_2^2| \gtrsim |\xi|^2$ without loss of generality and use Lemma \ref{lemma CHS} with $\xi$ fixed. Given that $|\xi|^{2k - 4s} \lesssim 1$, the argument in \eqref{inequality calculations derivative large} yields $I_1 \lesssim M$.

            \vspace{2mm}
            \noindent \textbf{Case B:} $|\xi| \ll |\xi_1|$. In this regime, $\xi_2 \simeq -\xi_1$ and $\mathscr{M} \sim |\xi| \langle \xi \rangle^{k} |\xi_1|^{-2s}$. To verify the hypotheses of Lemma \ref{lemma CHS phase comparable to alpha}, we assume $M \lesssim |\alpha|$, so that \eqref{equation factorization phi^u} implies $|\xi\xi_1^2| \sim |\Phi_1^u| \lesssim |\alpha|$. Fixing $\xi$, we observe that $|\partial_{\xi_1}\Phi_1^u|  \sim |\xi\xi_1|$. Consequently, 
            \[
            \sup_{\xi} \int_{|\Phi_1^u-\alpha| < M} |\xi|^2 \langle \xi \rangle^{2k} |\xi_1|^{-4s} \, d\xi_1
            \lesssim
            \sup_{\xi} \int_{|\Phi_1^u-\alpha| < M} |\xi| \langle \xi \rangle^{2k} |\xi_1|^{-4s-1} \, d\Phi_1^u 
            \]
            \[
            \lesssim \sup_{\xi} \int_{|\Phi_1^u-\alpha| < M} (|\xi\xi_1^2|)^{1^-} \, d\Phi_1^u \lesssim \langle \alpha\rangle^{1^-} M.
            \]
            
            When $\xi_1$ (or $\xi_2$) is fixed, the derivative satisfies $|\partial_\xi \Phi_1^u| \sim |\xi_2|^2 \gtrsim |\xi\xi_1|$; hence, a similar reasoning provides the desired bound.
    
        \end{proof}

        \begin{lemma} \label{lemma multilinear estimate uv_x}
            Let $-k-3/2 < s < k+5/2$. Then the bilinear estimate
            \begin{equation}\label{inequality bilinear estimate uv_x}
                \|u v_x \|_{Y^{s,b'}} \lesssim \|u\|_{X^{k,b}}\|v\|_{Y^{s,b}}
            \end{equation}
            holds in the following cases:
            \begin{enumerate}
                \item $a < 0$ and $k \geq -3/4$,
                \item $0 < a < 1/4$ and $k > -3/4$,
                \item $a = 1/4$ and $k \geq 3/4$,
                \item $a \in (1/4, \infty) \setminus \{1\}$ and $k \geq 0$.
            \end{enumerate}
        \end{lemma}
            
        \begin{remark}
            The region of validity for this bilinear estimate is sharp (see Proposition \ref{proposition estimate uvx is sharp}).
        \end{remark}
            
        \begin{proof}
            The total phase $\Phi^v$ is defined in \eqref{equation phi^v}. Let 
            \[
            \mathscr{M} = \frac{|\xi_2|\langle \xi \rangle^{s}}{\langle \xi_1 \rangle^{k} \langle \xi_2 \rangle^{s}}.
            \]
            Since $\xi = \xi_1 + \xi_2$, it follows that
            \begin{equation}\label{equation phase factorization}
            \Phi^v = -\xi^3 + a\xi_1^3 + \xi_2^3 = \xi_1\bigl(a\xi_1^2-\xi^2 - \xi \xi_2 -\xi_2^2 \bigr).
            \end{equation}
            If all frequencies are bounded, the estimate \eqref{inequality bilinear estimate uv_x} follows from Lemma \ref{lemma CHS}. Hence, for the remainder of the proof, we assume that $|\xi_j| \gg 1$ for at least one $j \in \{\emptyset, 1, 2\}$.
            
            \vspace{2mm}
            \noindent \textbf{Case A:} $|\xi_2| \sim |\xi|$. In this regime, $\mathscr{M} \sim |\xi_2| \langle \xi_1 \rangle^{-k}$ and $|\xi_1| \lesssim |\xi|$.

            \vspace{2mm}
            \noindent \textbf{Subcase A1:} $\xi_2 \simeq -\xi$. This condition entails $\xi_1 \simeq 2\xi$ and $\mathscr{M} \sim |\xi|^{1-k}$. Observing that
            \begin{equation}\label{equation phase^v xi2 = -xi}
                \Phi^v = -\xi^3 + a\xi_1^3+\xi_2^3 \simeq -2\xi^3 + 8a\xi^3 = 2(4a-1)\xi^3,
            \end{equation}
            we distinguish between the scenarios $a=1/4$ and $a\neq 1/4$. If $a=1/4$, following the reasoning in \eqref{equation P morse}, there exist $\alpha' \in \R$ and a change of variables $\xi \mapsto q$ such that
            \begin{equation}\label{inequality I1}
                I_1 \coloneqq \sup_{\xi_1} \int_{|\Phi^v-\alpha| < M} |\xi_1|^{2-2k} \, d\xi
                \lesssim
                \sup_{\xi_1} |\xi_1|^{3-2k}
                \int_{|q^2-\alpha'| < \frac{M}{|\xi_1|^3}} dq
                \lesssim
                \sup_{\xi_1} |\xi_1|^{3/2-2k} M^{1/2} \lesssim M^{1/2},
            \end{equation}
            which allow us to apply Lemma \ref{lemma CHS}. If $a \neq 1/4$, we instead employ Lemma \ref{lemma CHS phase comparable to alpha}. Under the assumption $M \lesssim |\alpha|$, we have $|\xi_1|^3 \sim |\Phi^v| \lesssim |\alpha|$. Proceeding as in \eqref{inequality I1}, one finds that
            \[ I_1 \lesssim \sup_{\xi_1} |\xi_1|^{3/2-2k} M^{1/2} \lesssim |\xi_1|^3 \lesssim \langle \alpha \rangle M^{1/2}. \]              
            Furthermore, if $\xi_2$ is fixed, we obtain
            \[
            \sup_{\xi_2} \int_{|\Phi^v-\alpha| < M} |\xi|^{2-2k} \, d\xi
            \lesssim
            \sup_{\xi_2} \int_{|\Phi^v-\alpha| < M} |\xi|^{-2k} \, d\Phi^v
            \sim
            \sup_{\xi_2} \int_{|\Phi^v-\alpha| < M} |\Phi^v|^{-2k/3} \, d\Phi^v
            \lesssim
            \langle \alpha \rangle^{1^-} M.
            \]
            The analysis for fixed $\xi$ is analogous.
        
            \vspace{2mm}
            \noindent \textbf{Subcase A2:} $\xi_2 \simeq \xi$. In this case, $|\xi_1| \ll |\xi|$ and, by \eqref{equation phase factorization}, we have $|\Phi^v| \sim |\xi_1 \xi^2|$. We verify the hypotheses of Lemma \ref{lemma CHS phase comparable to alpha} under the assumption $M \lesssim |\alpha|$. In this regime, $|\xi_1 \xi^2| \sim |\Phi^v| \lesssim |\alpha|$, which implies
            \[
            \sup_{\xi_1} \int_{|\Phi^v-\alpha| < M} |\xi|^{2} \langle \xi_1 \rangle^{-2k} \, d\xi
            \lesssim
            \int_{|\Phi^v-\alpha| < M} |\xi| |\xi_1|^{-2k-1} \, d\Phi^v \lesssim
            \int_{|\Phi^v-\alpha| < M} |\xi| |\xi_1|^{2} \, d\Phi^v \lesssim
            \langle \alpha \rangle M.
            \]
            The cases where $\xi$ or $\xi_2$ are fixed are treated similarly, since $|\partial_{\xi_1} \Phi^v| \sim |\xi|^2 \gtrsim |\xi \xi_1|$.

            \vspace{2mm}
            \noindent \textbf{Subcase A3:} $|\xi_2| \not\simeq |\xi|$. This implies $|\xi_1| = |\xi-\xi_2| \gtrsim |\xi|$ and $|\xi| \sim |\xi_1| \sim |\xi_2|$; thus, $\M \sim |\xi|^{1-k}$. We then estimate:
            $$ I_1 \coloneqq \sup_{\xi_1} \int_{|\Phi^v-\alpha| <M} |\xi_1|^{2-2k} \, d \xi \lesssim \sup_{\xi_1} |\xi_1|^{-2k} \int_{|\Phi^v-\alpha| <M} d \Phi^v = \sup_{\xi_1} |\xi_1|^{-2k} M \lesssim \left\{ \begin{aligned}
                &M ,  &\text{if } a \geq 1/4, \\
                &\sup_{\xi_1} |\xi_1|^{3} M, &\text{if } a < 1/4.
            \end{aligned} \right.$$
          
            When $a \geq 1/4$, \eqref{inequality bilinear estimate uv_x} holds by Lemma \ref{lemma CHS}. For $a < 1/4$, we use Lemma \ref{lemma CHS phase comparable to alpha} as follows: in view of \eqref{equation relation phi^v and phi^u} and \eqref{inequality phase a<1/4}, we have $|\Phi^v|\sim|\xi_1|^3$. Under the assumption $M \lesssim |\alpha|$, the relation $|\xi_1|^{3} \sim |\Phi^v| \lesssim |\alpha|$ ensures $I_1 \lesssim \langle \alpha \rangle M$. Next, we fix $\xi$. If $|\partial_{\xi_1} \Phi^v| \gtrsim |\xi|^2$, the argument is identical to the scenario where $\xi_1$ is held constant. Conversely, if $|-a\xi_1^2 + \xi_2^2| \sim |\partial_\xi \Phi^v| \ll |\xi|^2$ (which necessitates $a > 0$), proceeding as in \eqref{equation P morse}, we obtain
            \[
            \sup_{\xi} \int_{|\Phi^v-\alpha| < M} |\xi|^{2-2k} \, d\xi_1 \lesssim \sup_{\xi} |\xi|^{\frac{3}{2}-2k} M^{1/2} \lesssim \langle \alpha \rangle^{1^-} M^{1/2}.
            \] 
            
            The case for fixed $\xi_2$ is analogous. We observe that for $a > 0$, the phase is stationary in the region $\xi_2 \simeq \pm \sqrt{a}\,\xi_1$, requiring $k > -3/4$. For $a < 0$, the phase is never stationary in this subcase, and the condition $k \geq -3/4$ suffices.
            
            \vspace{2mm}
            \noindent \textbf{Case B:} $|\xi_2| \gg |\xi|$. In this regime, $\xi_1 \simeq -\xi_2$ and $\mathscr{M} \sim |\xi_1|^{1-k-s}\langle \xi \rangle^s$. Note that $|\Phi^v| \sim |\xi_1|^3$. To verify the hypotheses of Lemma \ref{lemma CHS phase comparable to alpha}, we assume $M \lesssim |\alpha|$. Then $|\xi_1|^3 \lesssim |\alpha|$ implies
            $$\sup_{\xi} \int_{|\Phi^v-\alpha| < M} |\xi_1|^{2-2k-2s}\langle \xi \rangle^{2s}  d\xi_1 \lesssim \sup_{\xi} \int_{|\Phi^v-\alpha| < M} |\xi_1|^{-2k-2s}\langle \xi \rangle^{2s}  d\Phi^v \lesssim \sup_{\xi} \int_{|\Phi^v-\alpha| < M} |\xi_1|^{3^-}  d\Phi^v \lesssim\langle \alpha \rangle^{1^-} M.$$
            The same conclusion holds when $\xi_1$ (or $\xi_2$) is held constant.
            
            \vspace{2mm}
            \noindent \textbf{Case C:} $|\xi| \gg |\xi_2|$. Here, $\xi \simeq \xi_1$ and $\mathscr{M} \sim |\xi_2|\,|\xi|^{s-k}\langle \xi_2 \rangle^{-s} \lesssim |\xi|^{s-k}\langle \xi_2 \rangle^{1-s}$. Given that $a \neq 1$, we have $|\Phi^v| \sim |\xi_1|^3$. We apply Lemma \ref{lemma CHS phase comparable to alpha}. Assuming $M \lesssim |\alpha|$, we observe that $|\xi_1|^3 \lesssim |\alpha|$ and
            \[
            \sup_{\xi} \int_{|\Phi^v-\alpha| < M} \langle \xi_2 \rangle^{2-2s} |\xi_1|^{2s-2k}  d\xi_1 \lesssim \sup_{\xi} \int_{|\Phi^v-\alpha| < M} \langle \xi_2 \rangle^{2-2s} |\xi_1|^{-2+2s-2k}  d\Phi^v 
            \]
            \[ \lesssim \sup_{\xi} \int_{|\Phi^v-\alpha| < M} |\xi_1|^{3^-}  d\Phi^v\lesssim \langle \alpha \rangle^{1^-} M.
            \]
            This bound remains valid when fixing $\xi_1$ or $\xi_2$.

        \end{proof}    

       By virtue of these bilinear estimates, a standard fixed-point argument (see, for example, the book \cite[Section 7.4]{Linares2009IntroductionTN}) yields the following result.

        \begin{proposition}[Direct LWP]\label{proposition direct LWP}
            Let $a \in \R \setminus \{0,1\}$, $(k,s) \in \A_a^0$, and $(u_0,v_0) \in H^k(\R) \times H^s(\R)$. Thus, there exists a time $T = T(\|u_0\|_{H^k} + \|v_0\|_{H^s}, a) > 0$ and a unique solution
            \[
            (u,v) \in X^{k,b} \times Y^{s,b}
            \]
            to the system \eqref{equation hirota system} on the interval $[0,T]$ satisfying \eqref{equation classical u} and \eqref{equation classical v}. Moreover, the flow depends analytically on the initial data.
        \end{proposition}
    
    \subsection{Sharpness of the bilinear estimates}
        \begin{proposition}\label{proposition estimate del_x(v_1v_2) is sharp}
            The estimate
            \begin{equation}\label{equation multilinear estimate del_x(v1v2) 2}
                \|\partial_x(v_1v_2) \|_{X^{k,b'}} \lesssim \|v_1\|_{Y^{s,b}}\|v_2\|_{Y^{s,b}}
            \end{equation}
            fails if any of the following conditions is satisfied: 
            \begin{enumerate}
                \item $a \neq 1$ and $s \leq k-3/2$
                \item $a < 1/4$ and $s \leq k/2-3/8$
                \item $a = 1/4$ and $s < k/2+3/8$
                \item $a >1/4$ and $s < k/2$
            \end{enumerate}
        \end{proposition}

        \begin{proof}
            By duality, the inequality \eqref{equation multilinear estimate del_x(v1v2) 2} is equivalent to
            \begin{equation}\label{equation inequality duality delx(v1v2)}
                \left| \int_{\R^4} \frac{|\xi| \langle\xi\rangle^k \langle \tau-a\xi^3 \rangle^{b'}h(\xi,\tau)h_1(\xi_1,\tau_1)h_2(\xi_2,\tau_2) }{\langle\xi_1\rangle^s \langle \tau_1-\xi_1^3 \rangle^{b}\langle\xi_2\rangle^s \langle \tau_2-\xi_2^3 \rangle^{b}} \, d\tau_1 d\xi_1 d\tau_2 d\xi_2 \right| \lesssim \|h\|_{L^2} \|h_1\|_{L^2}\|h_2\|_{L^2}.
            \end{equation}
            For each case listed above, we construct a counterexample to show that \eqref{equation inequality duality delx(v1v2)} fails.
        
            \vspace{3 mm}

            \noindent\textbf{1)} $a \neq 1$ and $s \leq k-3/2$: 
            \vspace{1 mm}
            
            The obstruction $s > k-3/2$ arises in the region $\xi \simeq \xi_1$ (see Subcase A1 in the proof of Lemma \ref{lemma multilinear estimate delx(v1v2)}). Let $N \gg 1$ and define the following functions:
            $$h = \indicatrix_{[N+1,N+3](\xi)} \indicatrix_{[N^3-2,N^3](\tau)}, \quad h_1 = \indicatrix_{[N+1,N+2](\xi_1)} \indicatrix_{[N^3-2,N^3-1](\tau_1)}, \quad h_2 = \indicatrix_{[0,1](\xi_2)} \indicatrix_{[0,1](\tau_2)}. $$
            Then $\|h\|_{L^2} \sim \|h_1\|_{L^2} \sim \|h_2\|_{L^2} \sim 1$. Furthermore, in this region we have $|\tau_1 - \xi_1^3| \sim N^2$, $|\tau_2 - \xi_2^3| \sim 1$, and $|\tau - a\xi^3| \sim N^3$. Therefore, the inequality \eqref{equation inequality duality delx(v1v2)} implies
            $$N^{k+1+3b'-s-2b} \lesssim 1 \implies k-s \leq 2b-3b'-1 < \frac{3}{2} \implies s > k - \frac{3}{2}. $$

            \noindent\textbf{2)} $a<1/4$ and $s \leq k/2-3/8$:
            \vspace{1 mm}
            
            Here, the obstruction $s > k/2-3/8$ occurs in the region $\xi_1 \simeq \xi_2$ (see Subcase A2 in the proof of Lemma \ref{lemma multilinear estimate delx(v1v2)}). For $N \gg 1$, we define
            $$ h_1 = \indicatrix_{ \{ N \leq \xi_1 \leq N + N^{-1/2}, \, |\tau_1 - \xi_1^3| \leq 1 \} }, \quad h_2 = \indicatrix_{ \{ N \leq \xi_2 \leq N + N^{-1/2}, \, |\tau_2 - \xi_2^3| \leq 1 \} }. $$
            It follows that
            $$ \tau_1 + \tau_2 - \frac{\xi^3}{4} = \xi_1^3 + \xi_2^3 - \frac{\xi^3}{4} + O(1) = \frac{3}{4} \xi(\xi_1-\xi_2)^2 + O(1), $$
            where $|\xi(\xi_1-\xi_2)^2| \lesssim 1$. Thus, there exists a constant $C > 0$ such that $\left|\tau_1 + \tau_2 - \frac{\xi^3}{4}\right| \leq C$. Let
            $$h = \indicatrix_{ \left\{ 2N \leq \xi \leq 2N + 2N^{-1/2}, \, \left|\tau - \frac{\xi^3}{4}\right| \leq C \right\} }. $$
            The triangle inequality then yields
            $$|\tau - a\xi^3| \geq \left|\frac{1}{4}-a \right||\xi^3| - \left|\tau - \frac{\xi^3}{4} \right| \gtrsim N^3, $$
            which implies $|\tau - a\xi^3| \sim N^3$. Consequently, by \eqref{equation inequality duality delx(v1v2)}, we have
            $$N^{1+k+3b'-2s-1} \lesssim N^{-3/4} \implies k-2s \leq -3b' - \frac{3}{4} < \frac{3}{4} \implies s > \frac{k}{2} - \frac{3}{8}. $$

            \noindent\textbf{3)} $a=1/4$ and $s < k/2+3/8$: 
            \vspace{1 mm}
            
            The obstruction $s \geq k/2+3/8$ similarly emerges within the regime $\xi_1 \simeq \xi_2$ encountered in Case 2. Accordingly, the same reasoning applies, with the sole distinction that $|\tau - a\xi^3| \sim 1$ in view of the condition $a=1/4$.

            \vspace{3mm}
            \noindent\textbf{4)} $a > 1/4$ and $s < k/2$:
            \vspace{1 mm}
            
            The obstruction $s \geq k/2$ appears in the region $|\xi| \sim |\xi_1| \sim |\xi_2|$ (see Subcase A2 in the proof of Lemma \ref{lemma multilinear estimate delx(v1v2)}). Recalling \eqref{equation factorization phi^u a >=1/4} for $a > 1/4$, we have
            \begin{equation}\label{equation factorization}
                \xi_1^3 + \xi_2^3 - a\xi^3 = 3\xi(\xi_1 - \mu_a\xi)(\xi_1 - (1-\mu_a)\xi).
            \end{equation}
            For $N \gg 1$, let
            $$h_1 = \indicatrix_{ \{ N \leq \xi_1 \leq N + N^{-2}, \, |\tau_1 - \xi_1^3| \leq 1 \} }, \quad h_2 = \indicatrix_{ \left\{ -N + \frac{N}{\mu_a} \leq \xi_2 \leq -N + \frac{N}{\mu_a} + N^{-2}, \, |\tau_2 - \xi_2^3| \leq 1 \right\} }. $$
            Then there exists $C > 0$ such that $|\tau_1 + \tau_2 - a\xi^3| \leq C$, since 
            $$  \tau_1+\tau_2 -a\xi^3 = \xi_1^3+ \xi_2^3-a\xi^3  + O(1) = 3\xi(\xi_1-\mu_a\xi)(\xi_1-(1-\mu_a)\xi)+ O(1), $$
            where $|\xi(\xi_1-\mu_a\xi)(\xi_1-(1-\mu_a)\xi|\lesssim1$. . We define
            $$h = \indicatrix_{ \left\{ \frac{N}{\mu_a} \leq \xi \leq \frac{N}{\mu_a} + 2N^{-2}, \, |\tau - a\xi^3| \leq C \right\} } .$$
            Finally, the inequality \eqref{equation inequality duality delx(v1v2)} implies
            $$N^{1+k-2s-4} \lesssim N^{-3} \implies s \geq \frac{k}{2}. $$
        \end{proof}
        \begin{proposition}\label{proposition estimate uvx is sharp}
            The estimate
            \begin{equation}\label{equation multilinear estimate uv_x}
                \|u v_x \|_{Y^{s,b'}} \lesssim \|u\|_{X^{k,b}}\|v\|_{Y^{s,b}}
            \end{equation}
            fails if any of the following conditions is satisfied: 
            \begin{enumerate}
                \item $a \neq 1$ and $s \geq k+5/2$  
                \item $a<0$ and $k<-3/4$
                \item $0<a<1/4$ and $k\leq -3/4$
                \item $a=1/4$ and $k<3/4 $
                \item $a \in (1/4,\infty) \setminus \{1\} $ and $k<0$
            \end{enumerate}
        \end{proposition}

        \begin{proof} By duality, the estimate \eqref{equation multilinear estimate uv_x} is equivalent to
            \begin{equation}\label{equation inequality duality uvx}
                \left| \int \frac{|\xi_2| \langle\xi\rangle^s \langle \tau-\xi^3 \rangle^{b'}h(\xi,\tau)h_1(\xi_1,\tau_1)h_2(\xi_2,\tau_2) }{\langle\xi_1\rangle^k \langle \tau_1-a\xi_1^3 \rangle^{b}\langle\xi_2\rangle^s \langle \tau_2-\xi_2^3 \rangle^{b}}d\tau_1d\xi_1 d\tau_2 d\xi_2 \right| \lesssim \|h\|_{L^2} \|h_1\|_{L^2}\|h_2\|_{L^2}.
            \end{equation}
            
            \noindent\textbf{1)} $a\neq 1$ and $s \geq k+5/2$: 
            \vspace{1 mm}
            
            The obstruction $s < k+5/2$ arises in the region $\xi \simeq \xi_1$ (see Case C in the proof of Lemma \ref{lemma multilinear estimate uv_x}). Let $N\gg 1$ and define
            $$h = \indicatrix_{[N+2,N+4](\xi)} \indicatrix_{[N^3-2,N^3](\tau)}, \quad h_1 = \indicatrix_{[N+1,N+2](\xi_1)} \indicatrix_{[N^3-1,N^3](\tau_1)}, \quad h_2 = \indicatrix_{[1,2](\xi_2)} \indicatrix_{[-1,0](\tau_2)}. $$

            Since $a \neq 1$, we have $|\tau_1-a\xi_1^3| \sim N^3$, and thus the inequality \eqref{equation inequality duality uvx} implies
            $$N^{s+2b'-k-3b}\lesssim 1 \implies s-k \leq 2(b-b') + b = -2\epsilon+\delta + \frac{5}{2},$$
            where $b= \frac{1}{2}+ \delta$ and $b' = b-1+\epsilon$. By choosing $\delta < 2\epsilon$, we conclude that $s-k < 5/2$.

            \vspace{3mm}
            \noindent\textbf{2)} $a<0$ and $k<-3/4$: 
            \vspace{1 mm}
            
            The obstruction $k \geq -3/4$ occurs in the region $\xi_1 \simeq 2\xi$ (equivalently, $\xi_2 \simeq -\xi$; see Subcase A1 in the proof of Lemma \ref{lemma multilinear estimate uv_x}). For $N\gg 1$, we define
            $$h = \indicatrix_{ \left\{(\xi,\tau) \in \R^2; \hspace{2mm} N \leq \xi \leq N+ 2N^{-1/2}, \hspace{2mm}  |\tau -\xi^3| \leq 1\right\}}, \quad h_2 =\indicatrix_{ \left\{(\xi_2,\tau_2) \in \R^2; \hspace{2mm} -N \leq \xi_2 \leq -N+ N^{-1/2}, \hspace{2mm}  |\tau_2 -\xi_2^3| \leq 1\right\}}. $$
            
            Then, there exists $C>0$ such that $|\tau-\tau_2 - \xi_1^3/4| \leq C$, since
            $$  \tau-\tau_2 -\frac{\xi_1^3}{4} = \xi^3-\xi_2^3-\frac{\xi_1^3}{4}  + O(1) =\frac{3}{4} \xi_1(\xi_1+2\xi_2)^2 + O(1), $$
            where $|\xi_1(\xi_1+2\xi_2)^2| \lesssim 1$. Consequently, we define
            $$h_1 =\indicatrix_{ \left\{(\xi_1,\tau_1) \in \R^2; \hspace{2mm} 2N \leq \xi_1 \leq 2N+ N^{-1/2}, \hspace{2mm}  \left|\tau_1 -\frac{\xi_1^3}{4}\right| \leq C\right\}}. $$
            
            By the triangle inequality, $|\tau_1-a\xi_1^3| \sim N^3$, given that $a \neq 1/4$. Thus, \eqref{equation inequality duality uvx} implies
            $$N^{1+s-k-3b-s-1} \lesssim N^{-3/4} \implies k \geq -\frac{3}{4}.$$

            \vspace{3mm}
            \noindent\textbf{3)} $0<a<1/4$ and $k\leq -3/4$: 
            \vspace{1 mm}
            
            The obstruction $k > -3/4$ emerges in the region $\xi_2 \simeq \pm \sqrt{a}\xi_1$ (cf. Subcase A3 in the proof of Lemma \ref{lemma multilinear estimate uv_x}). Let $r = \sqrt{a}$ and 
            \[
            h_1 = \indicatrix_{ \left\{(\xi_1,\tau_1) \in \R^2; \hspace{2mm} N \leq \xi_1 \leq N+ N^{-1/2}, \hspace{2mm}  |\tau_1 -a\xi_1^3| \leq 1\right\}}, \quad h_2 =\indicatrix_{ \left\{(\xi_2,\tau_2) \in \R^2; \hspace{2mm} rN \leq \xi_2 \leq rN+ N^{-1/2}, \hspace{2mm}  |\tau_2 -\xi_2^3| \leq 1\right\}}.
            \]
            Setting $A = \frac{r^2}{(1+r)^2}$, we observe that
            \[ \tau -A\xi^3 = \tau_1+\tau_2-A\xi^3 = a \xi_1^3+\xi_2^3- A(\xi_1+\xi_2)^3+ O(1) = \frac{(\xi_2-r\xi_1)^2}{(1+r)^2}(r^2\xi_1+2r\xi_1+(2r+1)\xi_2)+ O(1).\]
            Hence, there exists $C>0$ such that $|\tau -A\xi^3|\leq C$. Accordingly, define
            \[ h = \indicatrix_{ \left\{(\xi,\tau) \in \R^2; \hspace{2mm} (1+r)N \leq \xi \leq (1+r)N+ 2N^{-1/2}, \hspace{2mm}  |\tau -\xi^3| \leq C\right\}}. \]
            Furthermore, $A \neq 1$ provided $a < 1/4$, which implies $|\tau-\xi^3| \sim N^3$. Consequently, inequality \eqref{equation inequality duality uvx} yields
            \[N^{1+3b'-k-1} \lesssim N^{-3/4} \implies k \geq 3b' + \frac{3}{4} > -\frac{3}{4}.\]

            \vspace{3mm}
            \noindent\textbf{4)} $a=1/4$ and $k < 3/4$: 
            \vspace{1 mm}
            
            As in Case 2, the constraint $k \geq 3/4$ appears within the regime $\xi_1 \simeq 2\xi$ (cf. Subcase A1 in the proof of Lemma \ref{lemma multilinear estimate uv_x}). Accordingly, the reasoning is identical, except that here $|\tau_1 - a\xi_1^3| \sim 1$ in view of the condition $a=1/4$.

            \vspace{3mm}
            \noindent\textbf{5)} $a>1/4$ and $k < 0$:
            \vspace{1 mm}
            
            The obstruction $k \geq 0$ occurs in Subcase A3 in the proof of Lemma \ref{lemma multilinear estimate uv_x}. Combining \eqref{equation relation phi^v and phi^u} with \eqref{equation factorization phi^u a >=1/4}, we obtain
            \begin{equation}\label{equation factorization Phi^v for a>1/4}
                a\xi_1^3+ \xi_2^3-\xi^3  = -3\xi_1 (\xi-\mu_a\xi_1)(\xi-(1-\mu_a)\xi_1). 
            \end{equation} 
            
            Therefore, following the reasoning in Case 4 of Proposition \ref{proposition estimate del_x(v_1v_2) is sharp}, the inequality \eqref{equation inequality duality uvx} implies
            $$N^{1-k-4} \lesssim N^{-3} \implies k \geq 0.$$
            
        \end{proof}
    
\section{Normal form reduction}\label{section normal form}

    In this section, we apply integration by parts to equations \eqref{equation classical u} and \eqref{equation classical v} to define the notion of integrated-by-parts strong solutions. Throughout the remainder of this paper, we assume $\beta = \gamma = \theta = 1$ without loss of generality. The profiles are defined as follows:

    \begin{equation}\label{equation profile}
        \tilde{u}(t) = e^{-ita\xi^3}\F_xu(t) \quad \text{ and } \quad \tilde{v}(t) =  e^{-it\xi^3}\F_xv(t).
    \end{equation}
    
    Consequently, we have
    \begin{equation}\label{equation profile u}
        \tilde{u}(t,\xi) = \tilde{u}(0,\xi) +  \int_0^t\int_{\xi= \xi_1+\xi_2} i\xi e^{it'\Phi_1^u }\tilde{v}_1\tilde{v}_2 d\xi_1dt' + \int_0^t\int_{\xi= \xi_1+\xi_2} i\xi e^{it'\Phi_2^u }\tilde{u}_1\tilde{u}_2 d\xi_1dt',    
    \end{equation}
    \begin{equation}\label{equation profile v}
        \tilde{v}(t,\xi) = \tilde{v}(0,\xi) +  \int_0^t\int_{\xi= \xi_1+\xi_2} i\xi_2 e^{it'\Phi^v }\tilde{u}_1\tilde{v}_2 d\xi_1dt',
    \end{equation}
    where we adopt the shorthand notation $\tilde{v}_j=\tilde{v}(t,\xi_j)$ and $\tilde{u}_j=\tilde{u}(t,\xi_j)$. The associated total phases are given by
    \begin{equation}\label{equation Phi_2^u}
        \Phi_1^u = -a\xi^3 + \xi_1^3+\xi_2^3, \quad \Phi_2^u = a(-\xi^3 + \xi_1^3+\xi_2^3), \quad \Phi^v = - \xi^3 + a \xi_1^3 + \xi_2^3.
    \end{equation}
    
    We perform integration by parts in time on the coupled terms over regions where problematic obstructions arise and the phase is large relative to the integrand's numerator. This procedure allows us to improve the local well-posedness results by addressing specific obstructions (see Table \ref{table problematic regions}). Assuming $u, v \in \mathscr{S} (\R\times\R )$, we first consider the equation for $u$ and decompose the coupled term as follows:
    \begin{equation}\label{equation N_0^u}
        \int_0^t\int_{\xi= \xi_1+\xi_2} i\xi e^{it'\Phi_1^u }\tilde{v}_1\tilde{v}_2 d\xi_1dt' =  \int_0^t\int_{U} i\xi e^{it'\Phi_1^u }\tilde{v}_1\tilde{v}_2 d\xi_1dt' + \int_0^t\underbrace{\int_{U^c} i\xi e^{it'\Phi_1^u }\tilde{v}_1\tilde{v}_2 d\xi_1}_{=:N^u_0[v,v]} dt',
    \end{equation}   
    where $U$ is referred to as the \textbf{problematic region}, as it contains the obstruction we aim to adress. The set $U$ depends on a parameter $\delta^u>0$, which will be chosen suitably in the proof of local well-posedness. Integrating the first term by parts in time yields:
    $$ \int_0^t\int_{U} i\xi e^{it'\Phi_1^u }\tilde{v}_1\tilde{v}_2 d\xi_1dt' = \left.\int_{U}  \frac{\xi e^{it'\Phi_1^u }}{\Phi_1^u}\tilde{v}_1\tilde{v}_2 d\xi_1\right|_{t'=0}^{t'=t} \underbrace{-\int_0^t\int_{U}  \frac{\xi e^{it'\Phi_1^u }}{\Phi_1^u}\partial_{t'}(\tilde{v}_1\tilde{v}_2) d\xi_1dt'}_{=:J}. $$
    
    We denote the boundary term by 
    \begin{equation}\label{equation B^u}
        B^u[v,v](t) \coloneqq \left.\int_{U}  \frac{\xi e^{it'\Phi_1^u }}{\Phi_1^u}\tilde{v}_1\tilde{v}_2 d\xi_1\right|_{t'=t}.
    \end{equation}
    Differentiating \eqref{equation profile v} with respect to time, we find:
    $$\partial_{t'}\tilde{v}(t') =  \int_{\xi= \xi_1+\xi_2} i\xi_2 e^{it'\Phi^v }\tilde{u}_1\tilde{v}_2 d\xi_1 \implies \partial_{t'}\tilde{v}_1 (t') =  \int_{\xi_1= \xi_{11}+\xi_{12}} i\xi_{12} e^{it'\Phi^{v_1} }\tilde{u}_{11}\tilde{v}_{12} d\xi_{11}. $$
    Substituting this into $J$ leads to: 
    $$J = -i\int_0^t\int_{U, \xi_1 = \xi_{11}+\xi_{12}}  \frac{\xi\xi_{12} e^{it'\Psi_1^u } }{\Phi_1^u}\tilde{u}_{11}\tilde{v}_{12} \tilde{v}_2d\xi_1d\xi_{11}dt' -i\int_0^t\int_{U, \xi_2=  \xi_{21}+\xi_{22}}  \frac{\xi\xi_{22} e^{it'\Psi_2^u } }{\Phi_1^u}\tilde{v}_{1}\tilde{u}_{21} \tilde{v}_{22}d\xi_1d\xi_{21}dt' $$
    \begin{equation}\label{equation N_1^u and N_2^u}
        =\int_0^t(N_1^u[u, v,v] + N_2^u[v, u,v])dt',
    \end{equation}
    where
    \begin{equation}\label{equation psi_1^u and psi_2^u}
        \Psi_1^u \coloneqq \Phi_1^u + \Phi^{v_1}  = - a\xi^3 + \xi_2^3+ a \xi_{11}^3 + \xi_{12}^3 \hspace{2mm} \text{ and } \hspace{2mm} \Psi_2^u \coloneqq \Phi_1^u + \Phi^{v_2} = - a\xi^3 + \xi_1^3+ a \xi_{21}^3 + \xi_{22}^3.
    \end{equation}

    In addition, we let $N_3^u[u,u](t) =  e^{-iat\xi^3}\F_x (\partial_x(u^2)) .$ We now apply an analogous procedure to $v$. Let $V$ be the problematic region for $v$ depending on a parameter $\delta^v>0$. It follows that:    
    $$\int_0^t\int_{\xi= \xi_1+\xi_2} i\xi_2 e^{it'\Phi^v }\tilde{u}_1\tilde{v}_2 d\xi_1dt' =   \int_0^t\int_{V^c} i\xi_2 e^{it'\Phi^v }\tilde{u}_1\tilde{v}_2 d\xi_1dt' + \left.\int_{V} \frac{\xi_2 e^{it'\Phi^v }}{ \Phi^v}\tilde{u}_1\tilde{v}_2 d\xi_1 \right|_{t'=0}^{t'=t}$$
   $$- i\int_0^t\int_{V,\xi_1 = \xi_{11}+ \xi_{12}} \frac{\xi_1\xi_2 e^{it'\Psi_1^v }}{ \Phi^v}\tilde{v}_{11}\tilde{v}_{12}\tilde{v}_{2} d\xi_{11}d\xi_1dt' -i\int_0^t\int_{V,\xi_1 = \xi_{11}+ \xi_{12}} \frac{\xi_1\xi_2 e^{it'\Psi_2^v }}{ \Phi^v}\tilde{u}_{11}\tilde{u}_{12}\tilde{v}_{2} d\xi_{11}d\xi_1dt' $$
    $$ -i\int_0^t\int_{V,\xi_2 = \xi_{21}+ \xi_{22}} \frac{\xi_2\xi_{22} e^{it'\Psi_3^v }}{ \Phi^v}\tilde{u}_1\tilde{u}_{21}\tilde{v}_{22}d\xi_{21}d\xi_1dt' $$
    \begin{equation}\label{equation N_j^v}
        =\int_0^tN_0^v[u,v]dt'+ \left.B^v[u,v](t')\right|_{t'=0}^{t'=t}+\int_0^t( N_1^v[v, v,v] + N_2^v[u, u,v]+N_3^v[u, u,v])dt',
    \end{equation}
    where
    \begin{equation}\label{equation psi^v}
        \Psi_1^v = -\xi^3+ \xi_2^3 + \xi_{11}^3 + \xi_{12}^3, \quad \Psi_2^v = -\xi^3+ \xi_2^3 + a\xi_{11}^3 + a\xi_{12}^3, \quad \Psi_3^v  = -\xi^3+ a\xi_1^3 +a\xi_{21}^3 + \xi_{22}^3.
    \end{equation}

    Finally, we obtain:
    \begin{equation}\label{equation IBP profile u}
        \tilde{u}(t,\xi) = \tilde{u}(0,\xi) +\left. B^u(t')\right|_{t'=0}^{t'=t}+\sum_{j=0}^{3} \int_0^tN_j^u(t')dt',    
    \end{equation}
    \begin{equation}\label{equation IBP profile v}
        \tilde{v}(t,\xi) = \tilde{v}(0,\xi) +  \left. B^v(t')\right|_{t'=0}^{t'=t}+\sum_{j=0}^{3} \int_0^tN_j^v(t')dt'.
    \end{equation}

    \begin{definition}[Integrated-by-parts strong solution]\label{definition IBPS} 
    Given initial data $(u_0,v_0) \in H^k(\R) \times H^s(\R)$, we say that $(u,v) \in C([0,T];H^k(\R) \times H^s(\R))$ is an \textbf{integrated-by-parts strong solution} (IBPS) to the system \eqref{equation hirota system} if there exist $\delta^u,\delta^v>0$ such that the profiles $(\tilde{u}, \tilde{v})$ satisfy either \eqref{equation profile u}--\eqref{equation IBP profile v} or \eqref{equation IBP profile u}--\eqref{equation profile v} for all $t \in [0,T]$.
    \end{definition}

    \begin{table}[ht]
        \centering
        \begin{tabular}{|c|c|c|c|}
        \hline
        Values of $a$ & Obstruction & Problematic region & Phase \\
        \hline
        $(-\infty, 1/4) \setminus \{0\}$&$s>k/2 -3/8$ & $U_1 = \{\xi_1 \in \R;\xi_1\simeq \xi_2 \text{ and } |\xi| >1/\delta^u  \}$ & $|\Phi^u_1|\sim |\xi|^3$ \\
        \hline
         $\R\setminus\{0,1\}$ & $s>k -3/2$ & $U_2 = \{\xi_1 \in \R;(\xi\simeq \xi_1 \text{ or } \xi\simeq \xi_2 ) \text{ and } |\xi| >1/\delta^u  \}$ & $|\Phi^u_1|\sim |\xi|^3 $ \\
        \hline
        $\R\setminus\{0,1\}$&$s<k+ 5/2$ & $V = \{\xi_1 \in \R;\xi\simeq \xi_1 \text{ and } |\xi|>1/\delta^v  \}$ & $|\Phi^v|\sim |\xi|^3$ \\
        \hline
        \end{tabular}
        \caption{Problematic regions.}
        \label{table problematic regions}
    \end{table}

    The problematic region $V$ is specified in Table \ref{table problematic regions}, while the set $U$ is defined as:
    \begin{equation}\label{equation definition U}
        U = 
        \begin{cases}
            U_1 \cup U_2, &  \text{ if } a \in (-\infty, 1/4) \setminus \{0\}, \\
            U_2, &  \text{ if } a \in [1/4, \infty) \setminus \{1\}, 
        \end{cases}
    \end{equation}
    where $U_1$ and $U_2$ are likewise detailed in the same table.

\section{Well-posedness in \texorpdfstring{$\A_a$}{A\_a}}\label{section well posedness in A_a}

    In this section, we extend the local theory to the region $\A_a$ within the framework of integrated-by-parts strong solutions. After establishing the necessary estimates, we are then able to complete the local well-posedness theory.

    \subsection{Multilinear estimates for the integral terms}\label{section multilinear estimates}
    
        In this subsection, we prove the multilinear estimates for the integral terms in equations \eqref{equation IBP profile u} and \eqref{equation IBP profile v}. For the quadratic integral terms, the procedure mirrors that of Subsection \ref{section bilinear estimates}. For the remaining higher-order integral terms, we invoke Lemma \ref{lemma Schur's test with weight}.
            
        Typically, the set $A$ in Lemma \ref{lemma Schur's test with weight} can be chosen such that the derivative of the total phase does not vanish simultaneously on both sides of the interpolation. Under this condition, we apply Lemma \ref{leminha}. Otherwise, we must utilize Lemma \ref{lemma Schur's test with weight}, either by performing a change of variables to the total phase when its derivative is large or by applying Morse's Lemma. However, note that Lemma \ref{lemma Schur's test with weight} does not exploit the potential magnitude of the total phase $|\Phi|$, whereas Lemma \ref{lemma CHS phase comparable to alpha} does. This distinction is crucial in the proofs of Lemmas \ref{Lemma N_1^u and N_2^u} and \ref{lemma N_3^v}, where a case analysis for $s > -k-1$ and $s \leq -k-1$ is required. Indeed, a direct application of Lemma \ref{lemma Schur's test with weight} imposes the restriction $s \leq -k-1$; however, the fact that the phase is large allows us to relax this condition via Lemma \ref{lemma CHS phase comparable to alpha}.
    
        We denote the nonlinear terms in physical space by 
        $$\Nl^u_j = \F_x^{-1}( e^{ita\xi^3}N_j^u), \quad \Nl^v_j  = \F_x^{-1}( e^{it\xi^3}N_j^v) \quad \text{for } 0\leq j \leq 3.$$

        Throughout this section, we assume $a \in \R \setminus \{0,1\}$. Recall the definition of $\A_a$ in \eqref{equation definition A_a} and the previously fixed values $b = (1/2)^+$ and $b' = (b-1)^+$.
        
        \begin{lemma}\label{lemma N_0^u}
            For $(k,s) \in \A_a$,
            $$\| \Nl_0^u[v_1,v_2]\|_{X^{k,b'}} \lesssim \|v_1\|_{Y^{s,b}}\|v_2\|_{Y^{s,b}}.$$
        \end{lemma} 

        \begin{proof}
            The definition of $N_0^u$ is in \eqref{equation N_0^u}. For $a \in (-\infty, 1/4) \setminus \{0\}$, the complement of the problematic region is
            $$(U_1 \cup U_2)^c = \{|\xi|\lesssim1 \text{ or }( \,  |\xi_1| \sim |\xi_2| ; \, \xi_1 \not \simeq \xi_2 )\}.$$

            For $a \in (1/4, \infty) \setminus \{1\}$, this region is
            $$U_2^c = \{|\xi|\lesssim1 \text{ or }|\xi_1| \sim |\xi_2| \}.$$
            
            The required bounds follow from the arguments presented in the proof of Lemma \ref{lemma multilinear estimate delx(v1v2)}, specifically in Subcase A2 and Case B.                 
        \end{proof}

        \begin{lemma}\label{lemma N_0^v}
            For $(k,s) \in \A_a$,
            \begin{equation}\label{estimate N_0^v}
                \| \Nl_0^v[u,v]\|_{Y^{s,b'}} \lesssim \|u\|_{X^{k,b}}\|v\|_{Y^{s,b}}
            \end{equation}
        \end{lemma} 

        \begin{proof}
            The definition of $N_0^v$ is provided in \eqref{equation N_j^v}. We have
            $$V^c = \{|\xi|\lesssim1 \text{ or }|\xi_2| \gtrsim |\xi|\},$$
            which corresponds to Cases A and B in the proof of Lemma \ref{lemma multilinear estimate uv_x}.
        \end{proof}

        \begin{lemma}\label{Lemma N_1^u and N_2^u}
            Suppose $(k,s) \in \A_a$, $\epsilon < \min\{4, s+3, 2s+3, s-k+3, 2s-k+3\}$ and 
            $$\epsilon < 
            \begin{cases}
                 \min\{3/2-k, s-k+3/2\}, & \text{if } s \leq -k-1, \\
                 \min\left\{s+\frac{3}{2}, 2s+\frac{3}{2}\right\},  & \text{if } a=4. \\
            \end{cases}$$
            
            Then, the following estimates hold:
            $$\| \Nl_1^u[u_{11}, v_{12}, v_2]\|_{X^{k+\epsilon,b'}} \lesssim \|u_{11}\|_{X^{k,b}}\|v_{12}\|_{Y^{s,b}}\|v_2\|_{Y^{s,b}},$$
            $$\| \Nl_2^u[v_1, u_{21}, v_{22}]\|_{X^{k+\epsilon,b'}} \lesssim \|v_1\|_{Y^{s,b}}\|u_{21}\|_{X^{k,b}}\|v_{22}\|_{Y^{s,b}}.$$
        \end{lemma}
        
        \begin{proof}       
            The terms $N_1^u$ and $N_2^u$ are defined in \eqref{equation N_1^u and N_2^u}. By symmetry, it suffices to prove the estimate for $\Nl_1^u$. We first consider the region $U_1$ (see Table \ref{table problematic regions} and \eqref{equation definition U}), recalling that here $a < 1/4$ and $a \neq 0$. In this regime, the frequencies satisfy $\xi_{11} + \xi_{12} = \xi_1 \simeq \xi_2 \simeq \xi/2$. From \eqref{equation Phi_2^u}, we have $ |\Phi_1^u|   \sim |(4a-1)\xi_2^3| \sim |\xi|^3, $ which implies
            $$ \mathscr{M} \coloneqq \frac{|\xi||\xi_{12}|\langle \xi\rangle ^{k+\epsilon}}{|\Phi_1^u|\langle \xi_{11}\rangle ^k\langle \xi_{12}\rangle ^s\langle \xi_{2}\rangle ^s} \sim \frac{|\xi_{12}||\xi |^{k+\epsilon-s-2}}{\langle \xi_{11}\rangle ^k\langle \xi_{12}\rangle ^s} \lesssim \frac{|\xi |^{k+\epsilon-s-2}}{\langle \xi_{11}\rangle ^k\langle \xi_{12}\rangle ^{s-1}}. $$
            Recall that $\Psi^u_1 = - a\xi^3 + \xi_2^3 + a\xi_{11}^3 + \xi_{12}^3$. From the definition of $N_1^u$, the convolution relation $\xi = \xi_{11} + \xi_{12} + \xi_2$ holds.
        
            \vspace{2mm}
            \noindent \textbf{Case A:} $|\xi_{11}| \gg |\xi_2|$. In this case, $\xi_{12} \simeq -\xi_{11}$ and $|\Psi^u_1| \sim |\xi_{11}|^3$. Consequently, $\mathscr{M} \sim |\xi|^{k+\epsilon-s-2}|\xi_{11}|^{1-k-s}$. If $s > -k - 1$, we employ Lemma  \ref{leminha} \ref{item 2} with $A=\{2, 12\}$. If $s \leq -k - 1$, we instead invoke Lemma \ref{lemma CHS phase comparable to alpha}:
            \begin{equation}\label{inequality application of CHS with M< alpha}
                \sup_{\xi} \int_{|\Psi_1^u-\alpha| < M} \frac{|\xi|^{2k-2s+2\epsilon-4}}{|\xi_{11}|^{2k+2s-2}} d\xi_{11} d\xi_2 \lesssim \sup_{\xi} \int_{|\Psi_1^u-\alpha| < M} \frac{|\xi|^{2k-2s+2\epsilon-3 + 0^+}}{|\xi_{11}|^{2k+2s}} d\Psi_1^u \frac{d\xi_2}{|\xi|^{1^+}}
            \end{equation}
            $$ \lesssim \sup_{\xi} \int_{|\Psi_1^u-\alpha| < M} |\xi_{11}|^{-2k-2s} d\Psi_1^u \frac{d\xi_2}{|\xi_2|^{1^+}} \lesssim \sup_{\xi} \int_{|\Psi_1^u-\alpha| < M} |\Psi_{1}^u|^{-(2k+2s)/3} d\Psi_1^u \frac{d\xi_2}{|\xi_2|^{1^+}} \lesssim \langle \alpha \rangle^{1^{-}} M.
            $$
            
            An analogous argument holds when $\xi_2$, $\xi_{11}$, or $\xi_{12}$ remains fixed. Thus, \eqref{inequality CHS alpha xi} and \eqref{inequality CHS alpha xi_j} are satisfied.
        
            \vspace{2mm}
            \noindent \textbf{Case B:} $|\xi_{11}| \ll |\xi_2|$. Here, $|\xi_{12} - \xi_{2}| \leq |\xi_{11}| + |\xi - 2\xi_2| \ll |\xi_2|$, which implies $\xi_{12} \simeq \xi_2$. It follows that
            $$ \mathscr{M} \sim \frac{|\xi|^{k+\epsilon-2s-1}}{\langle \xi_{11}\rangle^k}. $$
            We then apply Lemma \ref{leminha} \ref{item 2} with $A=\{\emptyset, 2\}$.

            \vspace{2mm}
            \noindent \textbf{Case C:} $|\xi_{11}| \sim |\xi_2|$. In this regime, $|\xi_{12}| \lesssim |\xi|$ and 
            $$\mathscr{M} \lesssim \frac{|\xi|^{\epsilon-s-2}}{\langle \xi_{12}\rangle^{s-1}} .$$ 
            We employ Lemma \ref{leminha} \ref{item 2}. For fixed $\xi$ and $\xi_2$, the relation $\xi_{12} = \xi - \xi_2 - \xi_{11}$ implies $|\partial_{\xi_{11}}\Psi_1^u| \sim |a\xi_{11}^2 - \xi_{12}^2|$. If $|\partial_{\xi_{11}}\Psi_1^u| \gtrsim |\xi|^2$, we set $A=\{\emptyset, 2\}$. Otherwise, we must have $\sqrt{a}|\xi_{11}| \simeq |\xi_{12}|$ (observing that this scenario requires $a > 0$). Assuming $\sqrt{a}\xi_{11} \simeq \xi_{12}$, it follows that
            $$\xi_2 \simeq \xi_{11} + \xi_{12} = (1+\sqrt{a})\xi_{11} \simeq \left(\frac{1}{\sqrt{a}}+1\right)\xi_{12}. $$
            Thus, $|\xi|, |\xi_2| > (1^+)|\xi_{11}|, (1^+)|\xi_{12}|$, and we choose $A = \{2, 12\}$. Alternatively, if $\sqrt{a}\xi_{11} \simeq -\xi_{12}$, then $\xi_2 \simeq (1-\sqrt{a})\xi_{11}$. Since $a < 1/4$, we can ensure
            $$ |\xi_2| > (1^+) \frac{|\xi_{11}|}{2} \quad \text{and} \quad |\xi| > (1^+)|\xi_{11}|, $$
            while noting that $\xi_2 \simeq -\frac{(1-\sqrt{a})}{\sqrt{a}}\xi_{12}$. For fixed $\xi$ and $\xi_{11}$, the derivative satisfies $|\partial_{\xi_2}\Psi_1^u | = 3|\xi_2^2 - \xi_{12}^2| \ll |\xi_2|^2$ if and only if
            $$ \left|\frac{1-\sqrt{a}}{\sqrt{a}}\right|^2 = 1 \Longleftrightarrow a = \frac{1}{4}, $$
            which is precluded by the assumption $a < 1/4$. Consequently, the choice $A=\{2, 12\}$ remains valid.
            
            \vspace{1mm}
            
            We now turn to the region $U_2$ (see the definition in Table \ref{table problematic regions}). In this case, $a \in \R\setminus\{0,1\}$. There are two possibilities: $\xi \simeq \xi_1$ or $\xi \simeq \xi_2$. Notice that in both cases $|\Phi_1^u| \sim |\xi|^3$, by \eqref{equation Phi_2^u}.
            
            \vspace{2mm}
            \noindent\underline{If $\xi \simeq \xi_1$:}
            
            \[
            \M = \frac{|\xi||\xi_{12}|\langle \xi\rangle^{k+\epsilon}}
            {|\Phi_1^u|\langle \xi_{11}\rangle^k\langle \xi_{12}\rangle^s\langle \xi_{2}\rangle^s}
            \sim \frac{|\xi_{12}||\xi|^{k+\epsilon-2}}
            {\langle \xi_{11}\rangle^k\langle \xi_{12}\rangle^s\langle \xi_{2}\rangle^s}
            \lesssim
            \frac{\langle \xi \rangle^{k+\epsilon-2}}
            {\langle \xi_{11}\rangle^k\langle \xi_{12}\rangle^{s-1}\langle \xi_{2}\rangle^s}.
            \]
            
            \vspace{2mm}
            \noindent \textbf{Case A:} $|\xi| \sim |\xi_{11}|$. In this case, $|\xi_{12}|\lesssim |\xi|$. Thus,
            \[
            \M  \lesssim \frac{|\xi|^{\epsilon-2}}
            {\langle \xi_{12}\rangle^{s-1}\langle \xi_{2}\rangle^s}.
            \]
            
            \vspace{2mm}
            \noindent \textbf{Subcase A1:} $|\xi| \not\simeq |\xi_{11}|$. Since $a \neq 0$, we have 
            $|a\xi^2-\xi_{12}^2| \gtrsim |\xi|^2$ or 
            $|a\xi_{11}^2-\xi_{12}^2| \gtrsim |\xi|^2$. 
            In both situations, we apply Lemma \ref{leminha} \ref{item 1}. 
            In the first case, we take $A = \{\emptyset,12\}$. 
            The case $|a\xi_{11}^2-\xi_{12}^2| \gtrsim |\xi|^2$ is treated analogously.
            
            \vspace{2mm}
            \noindent \textbf{Subcase A2:} $\xi \simeq \xi_{11}$. Then $|\xi_{12}|\ll|\xi|$. Hence, we set $A=\{\emptyset,2\}$ in Lemma \ref{leminha} \ref{item 1}.
            
            \vspace{2mm}
            \noindent \textbf{Subcase A3:} $\xi_{11} \simeq -\xi$. Thus, $\xi_{12}\simeq 2\xi$ and
            \[
            \M \lesssim \frac{|\xi|^{\epsilon-s-1}}{\langle \xi_2\rangle^{s}}.
            \]
            Notice that, for $\xi_2$ and $\xi_{11}$ fixed, 
            $|\partial_\xi \Psi_1^u| \sim |-a\xi^2+ \xi_{12}^2|
            \sim |(-a+4)\xi^2|$. 
            Thus, for $a \neq 4$, we choose $A=\{\emptyset,12\}$ in Lemma \ref{leminha} \ref{item 2}. 
            For $a=4$, we apply Lemma \ref{lemma Schur's test with weight}. Consider
            \[
            I \coloneqq 
            \sup_{\xi, \xi_2} 
            \int_{|\Psi_1^u-\alpha| <M}
            \frac{|\xi|^{-1-s+\epsilon + 0^+}}
            {\langle \xi_2\rangle^s } d\xi_{11} \quad \text{ and} \quad J \coloneqq \sup_{\xi_{11}, \xi_{12}} 
            \int_{|\Psi_1^u-\alpha| <M}
            \frac{|\xi|^{-1-s+\epsilon + 0^+}}
            {\langle \xi_2\rangle^s } d\xi.
            \]
            
            To estimate $I$, we use Morse's Lemma with parameters (Lemma \ref{lemma morse parameters}). 
            For $\xi$ and $\xi_2$ fixed, define $p_j = \xi_j/\xi$. Since $p_{12} = 1-p_{11}-p_2$, we set
            \begin{equation}\label{equation application morse parameters}
                P(p_{11}, p_2) \coloneqq \frac{\Psi_1^u}{\xi^3}
                = -4 + p_2^3 + 4p_{11}^3 + p_{12}^3.
            \end{equation}
            
            Considering the variable $p_2$ fixed, $P$ has a nondegenerate critical point at 
            $(p_{11},p_{12},p_2) = (-1,2,0)$. Hence, 
            it follows that, for every $p_2$ in a neighborhood of $0$, 
            there exists a unique critical point $z(p_2)$ and a diffeomorphism 
            $p_{11} \mapsto q$ such that
            \[
            P(q, p_2) = P(z(p_2)) + q^2.
            \]
            
            Hence, using a change of variables and 
            Lemma \ref{lemma integral quadratic}, we obtain
            \[
            I \lesssim \sup_{\xi, \xi_2} 
            \int_{|q^2 +P(z(p_2))-\alpha/\xi^3| <M/\xi^3}
            \frac{|\xi|^{-s+\epsilon + 0^+}}
            {\langle \xi_2\rangle^s } dq   \lesssim \sup_{\xi, \xi_2}
            \frac{|\xi|^{-s+\epsilon-3/2+ 0^+}}
            {\langle \xi_2\rangle^s } M^{1/2}
            \lesssim
            \sup_{\xi, \xi_2}
            \langle\xi_2 \rangle^{-2s+\epsilon -\frac{3}{2}+0^+}
            M^{1/2}
            \lesssim M^{1/2},
            \]
            because $|\xi_2| \lesssim|\xi|\sim |\xi_{11}|$.  
            To estimate $J$, we observe that $|\partial_\xi\Psi_1^u| \sim |\xi|^2$ and perform a change of variables:
            \[ J \lesssim \sup_{\xi_{11}, \xi_{12}} 
            \int_{|\Psi_1^u-\alpha| <M}
            \frac{|\xi|^{-3-s+\epsilon + 0^+}}
            {\langle \xi_2\rangle^s } d\Psi_1^u \lesssim \sup_{\xi_{11}, \xi_{12}} 
            \int_{|\Psi_1^u-\alpha| <M}
            \langle \xi\rangle ^{-3-2s+\epsilon + 0^+}
             d\Psi_1^u \lesssim M. \]
            
            \vspace{2mm}
            \noindent \textbf{Case B:} $|\xi| \gg |\xi_{11}|$. This implies $\xi \simeq \xi_{12}$. 
            We rely on Lemma \ref{leminha} \ref{item 1} with $A=\{11,12\}$, because
            \[
            \M  \lesssim \frac{|\xi|^{k+\epsilon-s-1}}
            {\langle \xi_{11}\rangle^k\langle \xi_{2}\rangle^s}.
            \]
            
            \vspace{2mm}
            \noindent \textbf{Case C:} $|\xi| \ll |\xi_{11}|$. In this case, 
            $\xi_{12} \simeq - \xi_{11}$ and $|\xi_{12}|\gg1$. Then,
            \[
            \M \sim 
            \frac{|\xi_{11}|^{1-k-s}|\xi|^{k-2+\epsilon}}
            {\langle \xi_2 \rangle^s }.
            \]
            
            If $s> -k-1$, we employ Lemma \ref{leminha} \ref{item 1} with $A=\{2,12\}$. 
            If $s\leq-k-1$, we instead apply Lemma \ref{lemma CHS phase comparable to alpha} 
            and argue as in \eqref{inequality application of CHS with M< alpha}, 
            since $|\Psi_1^u| \sim |\xi_{11}^3|$.
            
            \vspace{2mm}
            \noindent\underline{If $\xi\simeq \xi_2$:} 
            since $|\Phi_1^u| \sim |\xi|^3$, we have
            \[
            \M = \frac{|\xi||\xi_{12}|\langle \xi\rangle^{k+\epsilon}}
            {|\Phi_1^u|\langle \xi_{11}\rangle^k
            \langle \xi_{12}\rangle^s\langle \xi_{2}\rangle^s}
            \sim \frac{|\xi_{12}||\xi|^{k-s+\epsilon-2}}
            {\langle \xi_{11}\rangle^k\langle \xi_{12}\rangle^s}
            \lesssim
            \frac{|\xi|^{k-s+\epsilon-2}}
            {\langle \xi_{11}\rangle^k 
            \langle \xi_{12}\rangle^{s-1}}.
            \]
            
            \vspace{2mm}
            \noindent \textbf{Case A:} $|\xi_{11}| \ll |\xi|$. Then 
            $|\xi_{12}| \leq |\xi_{11}| + |\xi_1| \ll |\xi|$. 
            We use Lemma \ref{leminha} \ref{item 1} 
            with $A = \{\emptyset,11\}$.
            
            \vspace{2mm}
            \noindent \textbf{Case B:} $|\xi_{11}| \sim |\xi|$. 
            Since $\xi_{12} = \xi_{11}-\xi_1$, it follows that $|\xi_{12}|\sim |\xi|$, which yields $\M\sim |\xi|^{-2s+\epsilon-1}.$
            Furthermore, the condition $a \xi_{11}^2 \simeq \xi_{12}^2$ cannot hold. Indeed, if it did, then 
            $\xi_{12} \simeq \pm \sqrt{a}\xi_{11}$, implying $\xi_1 = \xi_{11}- \xi_{12} \simeq (1 \mp \sqrt{a}) \xi_{11}$. 
            As $a \neq 1$, this would entail $|\xi_1|\sim |\xi|$, contradicting the assumption $|\xi_1| \ll |\xi|$. 
            Thus, $a \xi_{11}^2 \not\simeq \xi_{12}^2$, and we apply Lemma \ref{leminha} \ref{item 3} with $A = \{\emptyset,2\}$.

            \vspace{2mm}
            \noindent \textbf{Case C:} $|\xi_{11}| \gg |\xi|$. 
            In this case, we have $\xi_{12} \simeq -\xi_{11}$, which implies that
            $\M \sim 
            |\xi|^{k-s+\epsilon-2}
            |\xi_{11}|^{-k-s+1}.$
            Consequently, if $s > -k-1$, we choose $A = \{\emptyset,11\}$ in Lemma \ref{leminha} \ref{item 2}. 
            If $s \leq -k-1$, we instead invoke Lemma 
            \ref{lemma CHS phase comparable to alpha} and follow the argument 
            used in \eqref{inequality application of CHS with M< alpha}, 
            noting that $|\Psi_1^u| \sim |\xi_{11}|^3$.
            
        \end{proof}

        \begin{lemma}\label{lemma N_1^v} 
           Let $(k,s)\in \A_a$ and $\epsilon < \min \{4, s+3, 2s+3\}$. Then,
              \[  \| \Nl_1^v[v_{11},v_{12},v_2]\|_{Y^{s+ \epsilon,b'}} \lesssim \|v_{11}\|_{Y^{s,b}}\|v_{12}\|_{Y^{s,b}}\|v_{2}\|_{Y^{s,b}}.\]
            
        \end{lemma}
        
        \begin{proof}
            Recall the definitions of $\Nl_1^v$ in \eqref{equation N_j^v} and $V$ in Table \ref{table problematic regions}. Since $\xi \simeq \xi_1$, \eqref{equation Phi_2^u} ensures that $|\Phi^v| \sim |\xi|^3$, which yields
            \[\M \coloneqq \frac{|\xi_1||\xi_{2}|\langle \xi\rangle ^{s+\epsilon}}{|\Phi^v|\langle \xi_{11}\rangle ^s\langle \xi_{12}\rangle ^s\langle \xi_{2}\rangle ^s} \sim \frac{|\xi_{2}||\xi |^{s+\epsilon-2}}{\langle \xi_{11}\rangle ^s\langle \xi_{12}\rangle ^s\langle \xi_{2}\rangle ^s} \lesssim \frac{|\xi |^{s+\epsilon-2}}{\langle \xi_{11}\rangle ^s\langle \xi_{12}\rangle ^s\langle \xi_{2}\rangle ^{s-1}}. \]
            From \eqref{equation psi^v}, the total phase is given by $\Psi^v_1 = -\xi^3 + \xi_2^3 + \xi_{11}^3 + \xi_{12}^3$. Without loss of generality, assume $|\xi_{11}| \geq |\xi_{12}|$, which implies $|\xi| \lesssim |\xi_{11}|$.
        
            \vspace{2mm}
            \noindent \textbf{Case A:} $|\xi_{11}| \not\simeq |\xi_{12}|$. Here, $|\xi| \simeq |\xi_{11}+\xi_{12}| \gtrsim |\xi_{11}|$ and $|\xi| \sim |\xi_{11}|$. Consequently, Lemma \ref{leminha} \ref{item 1} applies with $A=\{\emptyset,2\}$, provided that
            \[\M \lesssim \frac{|\xi|^{\epsilon-2}}{\langle \xi_{12}\rangle^s\langle \xi_2 \rangle ^{s-1}}.\]
        
            \vspace{2mm}
            \noindent \textbf{Case B:} $\xi_{11} \simeq \xi_{12}$. In this setting, $\xi_{11} \simeq \frac{\xi}{2}$. We then invoke Lemma \ref{leminha} \ref{item 2} with $A = \{\emptyset,11\}$, observing that
            \[\M \lesssim \frac{|\xi|^{\epsilon-s-2}}{\langle \xi_2 \rangle ^{s-1}}.\]
            
            \vspace{2mm}
            \noindent \textbf{Case C:} $\xi_{11} \simeq -\xi_{12}$. This implies $|\xi| \ll |\xi_{11}|$. Accordingly, we rely on Lemma \ref{leminha} \ref{item 1} with $A=\{\emptyset,12\}$, where
            \[\M \lesssim \frac{ |\xi|^{s+\epsilon-2}}{|\xi_{11}|^{2s}\langle \xi_2\rangle^{s-1}}. \]
        \end{proof}

        \begin{lemma}\label{lemma N_2^v}
            Let $(k,s)\in \A_a$ and $\epsilon < \min \{k-s+4, k+3, 2k+3, 2k-s+4\}$. For $a=4$, assume further that $\epsilon < \min \{2k-s+13/4, 2k+9/4\}$. Then,
            \[
            \|\Nl_2^v[u_{11}, u_{12}, v_2]\|_{Y^{s+\epsilon, b'}} \lesssim \|u_{11}\|_{X^{k,b}} \|u_{12}\|_{X^{k,b}} \|v_{2}\|_{Y^{s,b}}.
            \]
        \end{lemma}
        
        \begin{proof}
            Recall the definitions of $\Nl_2^v$ in \eqref{equation N_j^v} and $V$ in Table \ref{table problematic regions}. By \eqref{equation Phi_2^u}, we have $|\Phi^v| \sim |\xi|^3$. Consequently,
            \[
            \M = \frac{|\xi_1||\xi_{2}|\langle \xi\rangle ^{s+\epsilon}}{|\Phi^v|\langle \xi_{11}\rangle ^k\langle \xi_{12}\rangle ^k\langle \xi_{2}\rangle ^s} \lesssim \frac{|\xi |^{s+\epsilon-2}}{\langle \xi_{11}\rangle ^k\langle \xi_{12}\rangle ^k\langle \xi_{2}\rangle ^{s-1}}.
            \]
            From \eqref{equation psi^v}, note that $\Psi^v_2 = - \xi^3+ \xi_2^3+ a\xi_{11}^3 +a\xi_{12}^3$. By symmetry, assume $|\xi_{11}| \gtrsim |\xi_{12}|$, which implies $|\xi| \lesssim |\xi_{11}|$.
        
            \vspace{2mm}
            \noindent \textbf{Case A: } $|\xi_{11}| \not \simeq |\xi_{12}|$. 
            Since $|\xi| \simeq |\xi_{11}+\xi_{12}| \gtrsim |\xi_{11}|$, we have $|\xi| \sim |\xi_{11}|$. Thus, we apply Lemma \ref{leminha} \ref{item 1} with $A=\{\emptyset,2\}$, noting that
            \[
            \M \lesssim \frac{|\xi |^{s+\epsilon-k-2}}{\langle \xi_{12}\rangle ^k\langle \xi_{2}\rangle^{s-1}}.
            \]
        
            \vspace{2mm}
            \noindent \textbf{Case B: } $\xi_{11} \simeq \xi_{12}.$ 
            In this regime, $\xi_{11} \simeq \xi/2$ and  
            \[ 
            \M \lesssim \frac{|\xi |^{s+\epsilon-2k-2}}{\langle \xi_{2}\rangle^{s-1}}. 
            \]
            Moreover, for fixed $\xi_2$ and $\xi_{12}$, the derivative satisfies $|\partial_{\xi}\Psi_2^v| \sim |\xi^2 - a\xi_{11}^2| \sim |4-a||\xi_{11}|^2 \sim |\xi|^2$. Provided $a \neq 4$, Lemma \ref{leminha} applies with $A = \{2,12\}$. 
            If $a=4$, we instead invoke Morse's Lemma with parameters (cf. Lemma \ref{lemma morse parameters}) and Lemma \ref{lemma integral quadratic}, following the approach in \eqref{equation application morse parameters}. In this context, we use Lemma \ref{lemma Schur's test with weight} with $A=\{2,12\}$, $\M_1 \lesssim \langle \xi_{2} \rangle^{1-s} |\xi |^{s-2k-11/4+\epsilon}$ and $\M_2 \lesssim \langle \xi_{2} \rangle^{1-s} |\xi |^{s-2k-5/4+\epsilon}$. 
            The condition $|\xi_2| \lesssim |\xi_{12}|$ then yields
            \[  \sup_{\xi_2, \xi_{12}} \int_{|\Psi_2^v-\alpha| <M} \M_1 d\xi_{11} \lesssim
            \sup_{\xi_2, \xi_{12}} \int_{|\Psi_2^v-\alpha| <M} \langle \xi_{2}\rangle^{1-s}|\xi_{12} |^{s-2k-11/4+\epsilon+0^+} d\xi_{11} 
            \]
            \[\lesssim \sup_{\xi_2, \xi_{12}} \langle \xi_{2}\rangle^{1-s}|\xi_{12} |^{s-2k-13/4+\epsilon+0^+} M^{1/2} \lesssim M^{1/2}.
            \]
            For fixed $\xi$ and $\xi_{11}$, the relation $|\partial_{\xi_2}\Psi_2^v| \sim |\xi|^2$ holds, and the other bound follows via a change of variables.
        
            \vspace{2mm}
            \noindent \textbf{Case C: } $\xi_{11} \simeq -\xi_{12}$. 
            Here, $|\xi| \ll |\xi_{11}|$. We apply Lemma \ref{leminha} \ref{item 1} with $A=\{2,12\}$, utilizing the bound
            \[ 
            \M \lesssim \frac{|\xi |^{s+\epsilon-2}}{\langle \xi_{11}\rangle ^{2k}\langle \xi_{2} \rangle^{s-1}}.
            \]
        \end{proof}

        \begin{lemma}\label{lemma N_3^v}
            Let $(k,s) \in \A_a$ and $\epsilon < \min \{2k+3, k-s+4, 2k-s+4, k+3\}$. For $s \leq -k-1$, assume further that $\epsilon < k-s+3/2$. Then,
            \[
            \| \Nl_3^v[u_{1},u_{21},v_{22}]\|_{Y^{s+\epsilon,b'}} \lesssim \|u_{1}\|_{X^{k,b}}\|u_{21}\|_{X^{k,b}}\|v_{22}\|_{Y^{s,b}}.
            \]
        \end{lemma}
        
        \begin{proof}
            Recall the definitions of $\Nl_3^v$ in \eqref{equation N_j^v} and of $V$ in Table \ref{table problematic regions}. Since $|\Phi^v| \sim |\xi|^3$, we have
            \[
            \M = \frac{|\xi_2||\xi_{22}|\langle \xi\rangle ^{s+\epsilon}}{|\Phi^v|\langle \xi_{1}\rangle ^k\langle \xi_{21}\rangle ^k\langle \xi_{22}\rangle ^s} \sim \frac{|\xi_{2}\xi_{22}||\xi |^{s+\epsilon-k-3}}{\langle \xi_{21}\rangle ^k\langle \xi_{22}\rangle ^s} \lesssim \frac{|\xi_{2}||\xi |^{s+\epsilon-k-3}}{\langle \xi_{21}\rangle ^k\langle \xi_{22} \rangle^{s-1}}.
            \]
            Moreover, $\Psi^v_3 = - \xi^3 + a\xi_1^3 + a\xi_{21}^3 + \xi_{22}^3$.
        
            \vspace{2mm}
            \noindent \textbf{Case A: } $|\xi| \simeq |\xi_{22}|$. 
            Note that $|\xi_{21}+\xi_{22}| = |\xi-\xi_1| \ll |\xi| \simeq |\xi_{22}|$, which implies $\xi_{21} \simeq -\xi_{22}$. Hence, $\M \lesssim |\xi_2||\xi|^{\epsilon-2k-4}$. We then apply Lemma \ref{leminha} \ref{item 2} with $A=\{\emptyset,1\}$.
        
            \vspace{2mm}
            \noindent \textbf{Case B: } $|\xi| \not \simeq |\xi_{22}|$. 
            In this case, it follows that $|\xi| \not \simeq |\xi_{21}|$.
        
            \vspace{2mm}
            \noindent \textbf{Subcase B1: } $|\xi| \lesssim |\xi_{21}|$. 
            Here, $|\xi_{22}| \geq |\xi_{21}| - |\xi-\xi_1| \gtrsim |\xi|$ and $|\xi_{21}| \sim |\xi_{22}|$. Consequently, $\M \lesssim |\xi_{2}||\xi_{21}|^{1-s-k}|\xi|^{\epsilon+s-k-3}$. If $s > -k-1$, the result follows from Lemma \ref{leminha} \ref{item 1} with $A=\{1,21\}$. Otherwise, we use Lemma \ref{lemma CHS phase comparable to alpha}. First, observe that if $|\xi| \gtrsim |\xi_{21}|$, then $|\xi| \sim |\xi_{21}|$ and we set $A=\{1,21\}$ in Lemma \ref{leminha} \ref{item 2}. Conversely, if $|\xi_{21}| \gg |\xi|$, then $\xi_{22} \simeq -\xi_{21}$. Since $a \neq 1$, this implies $|\Psi_3^v| \sim |\xi_{21}|^3$ and $\M \sim |\xi_2||\xi_{21}|^{1-s-k}|\xi|^{s+\epsilon-k-3}$. Thus, we proceed as in \eqref{inequality application of CHS with M< alpha}.
            
            \vspace{2mm}
            \noindent \textbf{Subcase B2: } $|\xi| \gg |\xi_{21}|$. 
            In this regime, $|\xi_{22}| \leq |\xi_{21}| + |\xi-\xi_1| \ll |\xi|$, which implies $|\xi_{22}| \ll |\xi|$. Consequently, we employ Lemma \ref{leminha} \ref{item 1} with $A=\{1,22\}$, as
            \[
            \M \lesssim \frac{|\xi_{2}||\xi |^{s+\epsilon-k-3}}{\langle \xi_{21}\rangle ^k\langle \xi_{22} \rangle^{s-1}} \lesssim \frac{|\xi |^{s+\epsilon-k-2}}{\langle \xi_{21}\rangle ^k\langle \xi_{22} \rangle^{s-1}}.
            \]
        \end{proof}

    \subsection{Boundary terms}

        We now address the boundary terms defined in \eqref{equation B^u} and \eqref{equation N_j^v}. We denote their representations in physical space by 
        \[
        \B^u (t)=\F_x^{-1}( e^{ita\xi^3}B^u(t)), \quad \B^v(t) =\F_x^{-1}( e^{it\xi^3}B^v(t)),
        \]
        \[
        \B^u_0 (t)=\F_x^{-1}( e^{ita\xi^3}B^u(0)), \quad \B_0^v(t) = \F_x^{-1}( e^{it\xi^3}B^v(0)).
        \]

        Throughout the remainder of this section, $\eta \in C_c^\infty(\mathbb{R})$ denotes a smooth cutoff function such that $\eta \equiv 1$ on $[-1,1]$. In the following lemma, we establish that for $(k,s) \in \A_a^0$ (as defined in \eqref{equation definition A^0}), the boundary terms are bounded in Fourier restriction spaces. 

        \begin{lemma}\label{lemma B^u for a<1/4} 
            Let $(k,s) \in \A^0_{a}$ and $\rho < \min\{s-k+2,2s-k+3/2\}.$ Then,
            \begin{equation}\label{equation estimate B^u}
                \| \B^u[v_{1},v_{2}]\|_{X^{k+ \rho,b-\frac{\rho}{3}}} \lesssim (\delta^u)^{0^+} \|v_{1}\|_{Y^{s,b}}\|v_{2}\|_{Y^{s,b}}.
            \end{equation}
        \end{lemma}

        \begin{proof}
            The definition of $B^u$ is given in \eqref{equation B^u}. Let $\sigma = \tau-a\xi^3$ and $\sigma_j = \tau_j-\xi_j^3$ for $j \in \{1,2\}$. 
            By duality, the estimate \eqref{equation estimate B^u} reduces to proving that
            \begin{equation}\label{equation boundary estimate using duality}
                \left|\int_{U,\tau = \tau_1+\tau_2} \frac{|\xi|\langle \xi\rangle^{k+\rho} \langle \sigma\rangle^{b-\frac{\rho}{3}} h(\tau,\xi)h_1(\tau_1,\xi_1)h_2(\tau_2,\xi_2)}{|\Phi^u_1| \langle \xi_1\rangle^{s} \langle \xi_2\rangle^{s} \langle \sigma_1\rangle^{b}\langle \sigma_2\rangle^{b}}  d\xi_1 d\xi_2 d\tau_1 d\tau_2 \right| \lesssim (\delta^u)^{0^+}\|h\|_{L^2} \|h_1\|_{L^2} \|h_2\|_{L^2}, 
            \end{equation}
            where the region $U$ is defined in Table \ref{table problematic regions} and in \eqref{equation definition U}. We first consider the case $\rho=0$.

            \vspace{2mm}
            \noindent \textbf{Case A: }$|\sigma| \lesssim |\sigma_1|$ or $|\sigma| \lesssim |\sigma_2|$. Without loss of generality, assume $|\sigma| \lesssim |\sigma_1|$. By the Cauchy-Schwarz inequality, \eqref{equation boundary estimate using duality} reduces to verifying that
            \[ 
            I \coloneqq \sup_{\xi,\tau}\int_U \frac{|\xi|^2\langle \xi\rangle^{2k} \langle \sigma\rangle^{2b} }{|\Phi^u_1|^2 \langle \xi_1\rangle^{2s} \langle \xi_2\rangle^{2s} \langle \sigma_1\rangle^{2b}\langle \sigma_2\rangle^{2b}} d\xi_2 d\tau_2\lesssim (\delta^u)^{0^+}.
            \]
            To establish this, note that since $|\sigma| \lesssim |\sigma_1|$, we have
            \[
            I \lesssim \sup_{\xi,\tau}\int_U \frac{|\xi|^2\langle \xi\rangle^{2k}   }{|\Phi^u_1|^2 \langle \xi_1\rangle^{2s} \langle \xi_2\rangle^{2s}} \left( \int \frac{d\tau_2}{\langle \sigma_2 \rangle^{2b} } \right)  d\xi_2 \lesssim  \sup_{\xi}\int_U \frac{|\xi|^2\langle \xi\rangle^{2k}   }{|\Phi^u_1|^2 \langle \xi_1\rangle^{2s} \langle \xi_2\rangle^{2s}} d\xi_2. 
            \]
            In the region $U_1$ (for $a<1/4$), we have $|\Phi_1^u|\sim |\xi|^3$ and all frequencies are comparable. Thus,
            \[
            I \lesssim  \sup_{\xi}\int \frac{1}{ | \xi_1|^{4s-2k+4}}d\xi_2 \lesssim (\delta^u)^{0^+} \sup_{\xi}\int \frac{1}{ | \xi_1|^{4s-2k+4^-}}d\xi_2\lesssim (\delta^u)^{0^+}.
            \]
            Next, we consider $U_2$ for $a \in \R \setminus\{0,1\}$, where $|\Phi_1^u| \sim |\xi|^3$. By symmetry, we may assume $|\xi_1|\geq |\xi_2|$, which implies $\xi \simeq \xi_1$. Consequently,  
            \[ 
            I \lesssim \sup_{\xi}\int \frac{|\xi|^{2k-4 -2s}}{ \langle \xi_2\rangle^{2s}} d\xi_2 \lesssim (\delta^u)^{0^+}\sup_{\xi}\int \frac{|\xi|^{2k-4 -2s^+}}{ \langle \xi_2\rangle^{2s}}d\xi_2 \lesssim (\delta^u)^{0^+}.
            \]

            \vspace{2mm}
            \noindent \textbf{Case B: } $|\sigma| \gg|\sigma_1|, |\sigma_2|.$ Since $\sigma = \sigma_1+ \sigma_2+ \Phi_1^u$, it follows that $\sigma \simeq \Phi_1^u$. In this case, \eqref{equation boundary estimate using duality} becomes 
            \[ 
            \left|\int_{U,\tau = \tau_1+\tau_2} \frac{|\xi|\langle \xi\rangle^{k} \langle \sigma\rangle^{b'} }{|\Phi^u_1|^{0^+}(\delta^u)^{0+} \langle \xi_1\rangle^{s} \langle \xi_2\rangle^{s} \langle \sigma_1\rangle^{b}\langle \sigma_2\rangle^{b}} h h_1 h_2 d\xi_1 d\xi_2 d\tau_1 d\tau_2 \right| \lesssim \|h\|_{L^2} \|h_1\|_{L^2} \|h_2\|_{L^2}.
            \]

            Note that this inequality is equivalent to \eqref{equation multilinear estimate} via a duality argument and may be established using FREs. Consequently, the reasoning proceeds as in the proof of Lemma \ref{lemma multilinear estimate delx(v1v2)} for the regions $U_1$ (provided $a < 1/4$) and $U_2$ (if $a \in \R \setminus \{0,1\}$), since $|\Phi_1^u|\delta^u \gtrsim |\xi|^2 \gtrsim1$.

            \vspace{2mm}
            
            We now adress the case $\rho\neq 0$:

            \vspace{2mm}
            \noindent \textbf{Case A: } $|\sigma| \lesssim |\sigma_1|$ or $|\sigma| \lesssim |\sigma_2|$. Proceeding as in the case $\rho=0$, it suffices to show that
           \[  \sup_{\xi,\tau}\int_U \frac{|\xi|^2\langle \xi\rangle^{2k+2\rho} \langle \sigma\rangle^{2b} }{|\Phi^u_1|^2 \langle \xi_1\rangle^{2s} \langle \xi_2\rangle^{2s} \langle \sigma_1\rangle^{2b}\langle \sigma_2\rangle^{2b}} d\xi_2 d\tau_2\lesssim (\delta^u)^{0^+},
           \]
           which holds provided that $\rho < \min\{s-k+2, 2s-k+3/2\}. $
            
           \vspace{2mm}
           \noindent \textbf{Case B: } $|\sigma| \gg|\sigma_1|, |\sigma_2|.$ In this regime, $ \Phi_1^u \simeq \sigma$. Observe that in $U$ (see Table \ref{table problematic regions}), we have $|\sigma|\simeq |\Phi_1^u| \sim |\xi|^3 $. It follows that $|\xi|^{\rho}|\sigma |^{-\frac{\rho}{3}}\sim 1$, which reduces the problem to the case $\rho=0$.
        \end{proof}
        \begin{lemma}\label{lemma B_0^u for a<1/4} 
            Let $(k,s) \in \A_a$ and
            $\epsilon < \min\{s-k+2, 2s-k+3/2\}$. Then,
            \[
            \|\eta(t)\B_0^u[v_{1},v_{2}]\|_{X^{k+\epsilon,b}} \lesssim (\delta^u)^{0^+} \|v_{1}(0)\|_{H^s}\|v_{2}(0)\|_{H^s}.
            \]
        \end{lemma}

        \begin{proof}
            Observe that 
            \[ 
            \B_0^u(t) = \F_x^{-1}( e^{ita\xi^3}B^u(0)) = \F_x^{-1}( e^{ita\xi^3}\F_x\F_x^{-1} B^u(0)) = e^{-iat\partial_x^3}\F_x^{-1} B^u(0),
            \]
            which implies
            \[
            \|\eta(t) \B_0^u(t) \|_{X^{k,b}} \lesssim \| \F_x^{-1} B^u(0) \|_{H^k}.
            \]
            By duality and the Cauchy-Schwarz inequality, it suffices to show that
            \[ 
            \sup_{\xi}\int_U \frac{|\xi|^2\langle \xi\rangle^{2k+2\epsilon}   }{|\Phi^u_1|^2 \langle \xi_1\rangle^{2s} \langle \xi_2\rangle^{2s}} d\xi_2 \lesssim (\delta^u)^{0^+}.
            \]
            The result then follows by an argument similar to that used in Case A of Lemma \ref{lemma B^u for a<1/4}.
        \end{proof}
        \begin{lemma}\label{lemma B^u in H^k}
            Let $(k,s) \in \A_a$ and $\epsilon < \min\{s-k+2, 2s-k+3/2\}$. Then,
            \[
            \| \B^u[v_1,v_2]\|_{H^{k+\epsilon}}+\| \B_0^u[v_1,v_2]\|_{H^{k+\epsilon}} \lesssim (\delta^u)^{0^+} \|v_1\|_{H^s}\|v_2\|_{H^s}. 
            \]
        \end{lemma}

        \begin{proof}
            The proof is analogous to that of Lemma \ref{lemma B_0^u for a<1/4}.
        \end{proof}

        The following lemmas concern the boundary terms $\B^v$ and $\B^v_0$. Given that their proofs follow from the same arguments as those for $\B^u$ and $\B_0^u$, the details are omitted.

        \begin{lemma}\label{lemma B^v for a<1/4}
            Let $(k,s) \in \A^0_a$. Then,
            \[
            \| \B^v[u,v]\|_{Y^{s,b}} \lesssim (\delta^v)^{0^+} \|u\|_{X^{k,b}}\|v\|_{Y^{s,b}}.
            \]
        \end{lemma}

        \begin{lemma}\label{lemma B_0^v} 
            Let $(k,s) \in \A_a$ and $\epsilon < \min\{k-s+3, k+3/2\}$. Then,
            \[
            \| \eta(t)\B_0^v[u,v]\|_{Y^{s+\epsilon,b}} \lesssim (\delta^v)^{0^+} \|u(0)\|_{H^k}\|v(0)\|_{H^s}. 
            \]
        \end{lemma}

        \begin{lemma}\label{lemma B^v in H^s}
            Let $(k,s) \in \A_a$ and $\epsilon < \min\{k-s+3, k+3/2\}$. Then,
            \[
            \| \B^v[u,v]\|_{H^{s+\epsilon}}+\| \B_0^v[u,v]\|_{H^{s+\epsilon}} \lesssim (\delta^v)^{0^+} \|u\|_{H^k}\|v\|_{H^s}. 
            \]
        \end{lemma}

    \subsection{Additional multilinear estimates}

        We also require smoothing estimates for the term $\partial_x(u^2)$ in the equation for $u$. As noted in \cite{correia2025sharplocalwellposednessschrodingerkortewegde}, this would ideally correspond to a global smoothing estimate (i.e., a gain in $H^s$ regularity), which is known to be unattainable \cite{ISAZA}. Nevertheless, one can derive a gain in temporal regularity. Recall that $\Nl_3^u[u,u] (t)= \partial_x(v^2)$ and $N^u_3[u,u](t) = e^{-ita\xi^3}\F_x( \partial_x(v^2))$. We decompose $N_3^u$ into two components: 
        \[
        N_{3,\ll}^u[u_1,u_2] (t)\coloneqq \int_{\xi = \xi_1+\xi_2, |\xi_2 |\ll |\xi|} \xi e^{it\Phi_2^u }\tilde{u}_1\tilde{u}_2 d\xi_1,
        \]
        \[
        N_{3,\gtrsim}^u[u_1,u_2](t) \coloneqq \int_{\xi = \xi_1+\xi_2, |\xi_1 |\gtrsim |\xi_2|\gtrsim |\xi|} \xi e^{it\Phi_2^u }\tilde{u}_1\tilde{u}_2 d\xi_1,
        \]
        where $\Phi_2^u$ is defined in \eqref{equation Phi_2^u}. We define $\Nl^u_{3,\ll}$ and $\Nl^u_{3,\gtrsim}$ accordingly. Note that $\R$ can be partitioned into the subsets 
        $\{\xi_1 \in \R : |\xi_1|\ll |\xi|\}$, $\{\xi_1 \in \R : |\xi_2|\ll |\xi|\}$, and $\{\xi_1 \in \R : |\xi_1|\gtrsim |\xi| \text{ and } |\xi_2|\gtrsim |\xi|\}$. Furthermore, the set $\{\xi_1 \in \R : |\xi_1|\gtrsim |\xi| \text{ and } |\xi_2|\gtrsim |\xi|\}$ is contained in the union of $\{\xi_1 \in \R : |\xi_1|\gtrsim |\xi_2| \gtrsim |\xi|\}$ and $\{\xi_1 \in \R : |\xi_2|\gtrsim |\xi_1| \gtrsim |\xi|\}$. Consequently, the symmetry of the arguments of $\Nl^u_3$ implies
        \begin{equation}\label{inequality bound N_3 by N_3,<< and N_3,>}
            \| \Nl^u_3[u,u]\|_{X^{k,b}} \lesssim  \| \Nl^u_{3,\ll}[u,u]\|_{X^{k,b}} + \| \Nl^u_{3,\gtrsim}[u,u]\|_{X^{k,b}}.
        \end{equation}
        
        By adapting the arguments in \cite[Lemma 4.17]{correia2025sharplocalwellposednessschrodingerkortewegde}, we have 
        \begin{lemma}[Time smoothing for the KdV nonlinearity]\label{lemma time smoothing kdv} 
            Let $k > -3/4$ and $\rho > 0$ be such that
            \[ \rho <
            \begin{cases}
                \displaystyle k+\frac{3}{4}, &\text{ if } -\frac{3}{4} < k \leq 0,\\
               \displaystyle \frac{3}{4}, &\text{ if }  k > 0. \\
            \end{cases} \]
            Then, the following estimates hold:
            \begin{equation}\label{inequality N_3,>}
                \|\Nl^u_{3,\gtrsim} \|_{X^{k+\rho,b'}} \lesssim \|u_1\|_{X^{k,b}}\|u_2\|_{X^{k,b}}.
            \end{equation}
            \begin{equation}\label{inequality N_3,<< with c}
                \|\Nl^u_{3,\ll} \|_{X^{k+\rho,c}} \lesssim \|u_1\|_{X^{k+\rho,b-\rho/3}}\|u_2\|_{X^{k,b}}, 
            \end{equation}
            where $-b<c < b-1$.
        \end{lemma}   

        For the proof, we need the following elementary calculus result:
        \begin{proposition}\label{proposition calculus}
            Let $\gamma, \theta \in \R$. If $l>1$ and $l' \geq 0$, then
            \[
            \int_\R \frac{dx}{\langle x - \gamma \rangle^{l'} \langle x -\theta \rangle^{l}} \underset{l,l'}{\lesssim} \frac{1}{ \langle \gamma-\theta \rangle^{\min\{l,l'\}}}. 
            \]
        \end{proposition}

        \begin{proof}[Proof of Lemma \ref{lemma time smoothing kdv}]
            We first establish the inequality \eqref{inequality N_3,<< with c}. Let  $\sigma_j = \tau_j-a\xi_j^3$ for $j \in \{\emptyset,1,2\}$. By duality, the estimate is equivalent to
            \begin{equation}\label{inequality dual of N_3,<<}
                \left| \int_{\substack{\xi =\xi_1+\xi_2  \\ \sigma = \sigma_1 + \sigma_2 + \Phi } }  \frac{|\xi|\langle\xi\rangle^{k+\rho}\indicatrix_{\{|\xi_2|\ll |\xi|\}}h(\tau,\xi)h_1(\tau_1,\xi_1)h_2(\tau_2,\xi_2)}{\langle \sigma \rangle^{-c}\langle \sigma_1 \rangle^{b-\rho/3}\langle \sigma_2 \rangle^{b}\langle\xi_1\rangle^{k+\rho}\langle\xi_2\rangle^{k} } d\tau d\tau_1 d\xi d\xi_1 \right| \lesssim \|h\|_{L^2}\|h_1\|_{L^2}\|h_2\|_{L^2}.
            \end{equation}

            Without loss of generality, assume $a=1$. In cases A and B, we suppose $k \leq 0$ and $|\xi_2| \gg 1$; the regimes where $k > 0$ or $|\xi_2| \lesssim 1$ are treated analogously.

            \vspace{2mm}
            \noindent \textbf{Case A: }$|\sigma_1| \gtrsim \max\{|\sigma|, |\sigma_2|\}$. In this case, the relation $\sigma = \sigma_1+\sigma_2+\Phi$ implies $|\sigma_1| \gtrsim |\Phi|$. By the Cauchy-Schwarz inequality, \eqref{inequality dual of N_3,<<} reduces to showing that
            \[ I \coloneqq \sup_{\tau_1,\xi_1} \int \frac{|\xi|^2}{\langle \sigma \rangle^{-2c}\langle \sigma_1 \rangle^{2b-2\rho/3}\langle \sigma_2 \rangle^{2b}|\xi_2|^{2k} } d\tau d\xi < \infty. \]

            By a change of variables (noting that $|\partial_\xi \Phi| \sim |\xi|^2$) and Proposition \ref{proposition calculus}, we obtain
            \[ I \lesssim \sup_{\tau_1,\xi_1} \int \frac{|\xi_2|^{-2k}}{\langle \sigma \rangle^{-2c}\langle \sigma_1 \rangle^{2b-2\rho/3}\langle \sigma_2 \rangle^{2b} } d\sigma d\Phi \lesssim   \sup_{\tau_1,\xi_1}\int \frac{|\Phi|^{-\frac{2k}{3}}}{\langle \sigma_1+\Phi \rangle^{-2c}\langle \Phi \rangle^{2b-2\rho/3}}d\Phi, \]
            since $|\Phi| = 3 |\xi \xi_1 \xi_2| \sim |\xi^2\xi_2| \gtrsim |\xi_2|^3$. This implies that $I$ is finite, provided that
            \[ -2c+\frac{2k}{3}+2b-\frac{2\rho}{3} > 1 \iff c < b-\frac{1}{2}+\frac{(k-\rho)}{3}, \]
            which is satisfied due to the hypothesis $c < b-1$.

            \vspace{2mm}
            \noindent \textbf{Case B: }$|\sigma| \gtrsim \max\{|\sigma_1|, |\sigma_2|\}$. In this setting, 
            \[ \sup_{\tau,\xi} \int \frac{|\xi|^2}{\langle \sigma \rangle^{-2c}\langle \sigma_1 \rangle^{2b-2\rho/3}\langle \sigma_2 \rangle^{2b}|\xi_2|^{2k} } d\tau_1 d\xi_1 \lesssim \sup_{\tau,\xi} \int \frac{d\sigma_1 d\Phi}{\langle \sigma \rangle^{-2c}\langle \sigma_1 \rangle^{2b-2\rho/3}\langle \sigma_2 \rangle^{2b}|\Phi|^{2k/3} } \]
            \[ \lesssim \sup_{\tau,\xi} \int \frac{d\sigma_1 d\Phi}{\langle \Phi \rangle^{-2c}\langle \sigma_1 \rangle^{2b-2\rho/3}\langle \sigma_2 \rangle^{2b}|\Phi|^{2k/3} } \lesssim \sup_{\tau,\xi} \int \frac{ d\Phi}{\langle \Phi \rangle^{-2c}\langle \sigma - \Phi \rangle^{2b-2\rho/3}|\Phi|^{2k/3} } < \infty, \]
            since $-2c+2b-2\rho/3 + 2k/3 > 1$. Indeed, this is equivalent to $c < b-\frac{1}{2}+\frac{(k-\rho)}{3}$, which follows from $c < b-1$. For the cases $k \geq 0$ and $|\xi_2| \lesssim 1$, the reasoning is analogous. This concludes the proof of \eqref{inequality N_3,<< with c}. 
            
            \vspace{2mm}
            \noindent \textbf{Case C: }$|\sigma_2| \gtrsim \max\{|\sigma|, |\sigma_1|\}$. Considering first $|\xi_2| \gg 1$, we estimate
            \[ I \coloneqq \sup_{\tau_2,\xi_2} \int \frac{|\xi|^2}{\langle \sigma \rangle^{-2c}\langle \sigma_1 \rangle^{2b-2\rho/3}\langle \sigma_2 \rangle^{2b}|\xi_2|^{2k} } d\tau d\xi \lesssim \sup_{\tau_2,\xi_2} \int \frac{|\xi||\xi_2|^{-2k-1}}{\langle \sigma \rangle^{-2c}\langle \sigma_1 \rangle^{2b-2\rho/3}\langle \sigma_2 \rangle^{2b}} d\sigma d\Phi, \]
            since $|\partial_\xi \Phi| = |\xi^2-\xi_1^2| \sim |\xi_2\xi|$. Moreover, we observe that $|\xi||\xi_2|^{-2k-1} \leq |\xi||\xi_2|^{1/2} \sim |\Phi|^{1/2}$ and $2b-2\rho/3 > 2b-1/2$. Consequently, 
            \[ I \lesssim \sup_{\tau_2,\xi_2} \int \frac{|\Phi|^{\frac{1}{2}}}{\langle \sigma \rangle^{-2c}\langle \sigma_1 \rangle^{2b-1/2}\langle \sigma_2 \rangle^{2b}}d\sigma d\Phi  \lesssim \sup_{\tau_2,\xi_2} \int \frac{|\Phi|^{\frac{1}{2}}}{\langle \sigma \rangle^{2\mu -2c}\langle \sigma_1 \rangle^{2b-1/2}\langle \sigma_2 \rangle^{2b-2\mu}} d\sigma d\Phi, \]
            where $\mu > 0$ is chosen such that $2\mu-2c > 1$. Proposition \ref{proposition calculus} then yields
            \[ I \lesssim \sup_{\tau_2,\xi_2} \int \frac{|\Phi|^{\frac{1}{2}}}{\langle \sigma_2+\Phi \rangle^{2b-1/2}\langle \sigma_2 \rangle^{2b-2\mu}} d\Phi \lesssim \sup_{\tau_2,\xi_2} \int \frac{d\Phi}{\langle \sigma_2+\Phi \rangle^{2b-1/2}\langle \Phi \rangle^{2b-2\mu-1/2}} < \infty, \]
            provided $2b-\mu > 1$. Thus, we require the existence of $\mu > 0$ such that $c+1/2 < \mu < 2b-1$, which holds if and only if $c < 2b-3/2$. 

            In the regime $|\xi_2| \lesssim 1$, it suffices to observe that $|\sigma_2| \gtrsim |\sigma|, |\Phi|$  ensure
            \[  \sup_{\tau_1,\xi_1} \int \frac{|\xi|^2}{\langle \sigma \rangle^{-2c}\langle \sigma_2 \rangle^{2b}} d\tau d\xi \lesssim \sup_{\tau_1,\xi_1} \int \frac{1}{\langle \sigma \rangle^{-2c+2b^--1} \langle \Phi \rangle ^{1^+}} d\sigma d\Phi < \infty, \]
            provided $c<b-1$. The proof of \eqref{inequality N_3,>} reduces to employing frequency-restricted estimates, and the argument follows the same strategy as in the proofs of Lemmas \ref{lemma multilinear estimate delx(v1v2)} and \ref{lemma multilinear estimate uv_x}.

        \end{proof} 

        The following result aims to circumvent the absence of a smoothing effect in \eqref{estimate N_0^v}. Indeed, Lemma \ref{lemma N_0^v com termo de bordo} will be employed to establish that $v-\eta(t)\B^v$ possesses sufficient gain in regularity to yield local well-posedness.

        \begin{lemma}\label{lemma N_0^v com termo de bordo}
            Let $(k,s) \in \A_a$ and $\epsilon < \min\{k+4-s, 2k+3, k+3, 2k+4-s\}$. In the case $a=4$, assume further that $\epsilon < \min \{2k-s+13/4, 2k+9/4\}$. Then,
            \[ 
            \| \Nl_0^v[u_1, \B^v[u_{21},v_{22}]] \|_{Y^{s+\epsilon,b'}} \lesssim \|u_1\|_{X^{k,b}}\|u_{21}\|_{X^{k,b}}\|v_{22}\|_{Y^{s,b}}. 
            \]
        \end{lemma}
        
        \begin{proof}
            Observe that 
            
            $$e^{-it\xi^3}\F_x \Nl_0^v[u_1,\B^v](\xi,t) = N_0^v[u_1,\B^v] = \int_{\xi_1 \in V^c} i\xi_2 e^{it\Phi^v}\tilde{u}_1 \underbrace{e^{-it\xi^3 }\F_x\B^v}_{B^v} d\xi_1 
              $$
            $$=i\int_{\xi_1 \in V^c} \xi_2 e^{it\Phi^v}\tilde{u}_1 \int_{\xi_{21}\in V}\frac{\xi_{22}e^{it\Phi^{v_2}}}{\Phi^{v_2}}\tilde{u}_{21}\tilde{v}_{22} d\xi_{21}d\xi_1 =i\int_{\xi_1 \in V^c} \int_{\xi_{21}\in V}\frac{\xi_2\xi_{22}e^{it\Psi_4^{v}}}{\Phi^{v_2}}\tilde{u}_1 \tilde{u}_{21}\tilde{v}_{22} d\xi_{21} d\xi_1,$$
            where
            \[
            \Phi^{v_2} =  - \xi^3_2+ a\xi_{21}^3+\xi^3_{22} \quad \text{and} \quad \Psi_4^v \coloneqq \Phi^v+ \Phi^{v_2} = -\xi^3 + a\xi_1^3+ a\xi_{21}^3+\xi^3_{22}.
            \]
        
            We proceed by applying FREs. 
            Note that $\xi_{21} \in V$ implies $ \xi_2 \simeq \xi_{21} \text{ and } |\xi_2|> \frac{1}{\delta^v} $. Consequently, $|\Phi^{v_2}|\sim |\xi_2|^3$ since $a \neq 1$. Furthermore,
            $\xi_1 \in V^c$ implies either $|\xi|\lesssim 1$ or $|\xi_2| \gtrsim|\xi|.$
        
            \vspace{2mm}
            \noindent \textbf{Case A: }$|\xi| \lesssim 1 $. This implies $ \xi_1 \simeq -\xi_2$ and
            $$\M \sim \frac{|\xi_{22}|\langle\xi \rangle^{s+\epsilon}}{| \xi_1 |^{2k+2}  \langle\xi_{22} \rangle^s} \lesssim\frac{\langle \xi_{22}\rangle^{1-s}}{\langle \xi_1\rangle^{2k+2}}. $$
            Moreover, as $|\xi_2| \simeq |\xi_1| \gg |\xi_{22}|$, we apply Lemma \ref{leminha} \ref{item 2} with $A=\{ \emptyset,1\}$.
            
            \vspace{2mm}
            \noindent \textbf{Case B: } $|\xi_2| \gtrsim|\xi|$. In this regime, $ |\xi_1| \lesssim |\xi_2|$ and
            $$ \M \lesssim \frac{\langle\xi \rangle^{s+\epsilon} \langle \xi_{22}\rangle^{1-s}}{| \xi_2 |^{k+2}  \langle\xi_{1} \rangle^k} .  $$
        
            \vspace{2mm}
            \noindent \textbf{Subcase B1: } $a\xi_1^2 \simeq \xi^2.$ This condition requires $ a>0$ and $\sqrt{a}\xi_1 \simeq \pm\xi$. Thus, $\xi_1+\xi_2 \simeq \pm \sqrt{a}\xi_1$, which yields 
            \begin{equation}\label{equation relation frequencies}
                \xi_{21} \simeq\xi_2  \simeq (\pm \sqrt{a}-1)\xi_1.
            \end{equation}
            Thus, $|\xi_{21}|\sim |\xi_2| \sim |\xi_1| \sim |\xi|$, and we obtain
            \[ \M \lesssim \frac{ \langle \xi_{22}\rangle^{1-s}}{| \xi |^{2k-s+2-\epsilon}  } . \]
        
            If $a\neq 4$, we note that for fixed $\xi$ and $\xi_{22}$, we have $|\partial_{\xi}\Psi_4^v| \sim |\xi_1^2-\xi_{21}^2| \gtrsim |\xi_1|^2$; otherwise,  $a = 4$ according to \eqref{equation relation frequencies}. We then apply Lemma \ref{leminha} \ref{item 2} with $A=\{1,21\}$. If $a=4$, we employ Lemma \ref{lemma Schur's test with weight} with $A=\{1,21\}$, where
            \[ \M_1 \lesssim \langle \xi_{22} \rangle^{1-s} |\xi |^{s-2k-11/4+\epsilon} \quad \text{and} \quad \M_2 \lesssim \langle \xi_{22} \rangle^{1-s} |\xi |^{s-2k-5/4+\epsilon}.\]
            Invoking Morse's Lemma with parameters (cf. Lemma \ref{lemma morse parameters}) and Lemma \ref{lemma integral quadratic}, following the approach in \eqref{equation application morse parameters}, we obtain
            $$ \sup_{\xi,\xi_{22}} \int_{|\Psi_4^v-\alpha|<M}  \frac{ \langle \xi_{22}\rangle^{1-s}}{| \xi |^{2k-s+11/4-\epsilon + 0^-}  }d\xi_1 \lesssim \sup_{\xi,\xi_{22}}\frac{ \langle \xi_{22}\rangle^{1-s}}{| \xi |^{2k-s+13/4-\epsilon + 0^-}  }M^{1/2} \lesssim \sup_{\xi_{22}} \langle \xi_{22}\rangle^{\epsilon -2k-9/4+0^+} M^{1/2}  \lesssim M^{1/2}.$$
            For fixed $\xi_1$ and $\xi_{21}$, the result follows by changing the variable to the phase $\Psi_4^v$.
        
            \vspace{2mm}
            \noindent \textbf{Subcase B2:} $a\xi_1^2 \not \simeq \xi^2$. For fixed $\xi_{21}$ and $\xi_{22}$, we have
            $$|\partial_\xi \Psi_4^v| = |-\xi^2+ a \xi_1^2| \gtrsim \max \{ |\xi_1|^2,|\xi|^2\} \gtrsim |\xi_2|^2.  $$
            The bound is obtained from Lemma \ref{leminha} \ref{item 1} with $A=\{0,1\}$ by considering all possible orderings of $|\xi|, |\xi_1|, |\xi_{22}|$, since $|\xi_2| \gtrsim \max\{|\xi|, |\xi_1|, |\xi_{22}|\}$.
        \end{proof}

    \subsection{Local well-posedness in \texorpdfstring{$\A_a$}{A\_a}}

        In this section, we provide a formal statement and proof for Theorem \ref{theorem general LWP}. Let $a \in \R \setminus \{0,1\}$. The regimes in which local well-posedness remains to be proved are: 
        
        \[\A^+_{a} \coloneqq 
        \begin{cases}
             \displaystyle\left \{(k,s)\in \R^2; \text{ }  k>-\frac{3}{4}, \text{ } k+\frac{5}{2} \leq s< k+3  \right \},  &\text{if } a < 1/4, \\
            \displaystyle \left \{(k,s)\in \R^2; \text{ } k \geq \frac{3}{4}, \text{ } k+\frac{5}{2} \leq s< k+3  \right \},  &\text{if } a= 1/4, \\
            \displaystyle \left \{(k,s)\in \R^2; \text{ }  k\geq0, \text{ } k+\frac{5}{2} \leq s< k+3 \right \},   &\text{if } a >1/4.
        \end{cases} \]
        
        \[\A^-_{a} \coloneqq 
        \begin{cases}
             \displaystyle\left \{(k,s)\in \R^2; \text{ } \max \left\{-\frac{3}{4},\frac{k}{2}-\frac{3}{4}, k-2\right\}< s\leq \max  \left\{\frac{k}{2}-\frac{3}{8}, k-\frac{3}{2}\right\}  \right \},  &\text{if } a < 1/4, \\
            \displaystyle \left \{(k,s)\in \R^2; \text{ } s \geq \frac{k}{2}+\frac{3}{8}, \text{ }  k-2< s\leq k-\frac{3}{2}  \right \},  &\text{if } a= 1/4, \\
            \displaystyle \left \{(k,s)\in \R^2; \text{ } s \geq \frac{k}{2}, \text{ }  k-2< s \leq k-\frac{3}{2} \right \},   &\text{if } a >1/4.
        \end{cases} \]
        
        \begin{theorem}[LWP]\label{theorem improved LWP}
            Let $a \in \R \setminus \{0,1\}$, $(k,s) \in \A_{a}$, and $(u_0,v_0) \in H^k(\R)\times H^s(\R)$. Then, there exist $\delta^u, \delta^v, T>0$, depending only on $\|u_0\|_{H^k}+\|v_0\|_{H^s}$ and on the parameter $a$, and
            $$(u,v) \in C([0,T],H^k(\R)) \times C([0,T],H^s(\R)),$$
            such that $(u,v)$ is an integrated-by-parts strong solution to \eqref{equation hirota system} on $[0,T]$. Moreover, the flow map depends analytically on the initial data and is independent of the choices of $\delta^u$ and $\delta^v$. 
        \end{theorem}

        \begin{proof}
            
            To establish existence for $(k,s) \in \A_a$, we first construct functions $u$ and $v$ in an auxiliary space with lower regularity. We then show that they gain regularity due to the smoothing effects established in the multilinear estimates. This allows us to conclude that the solution possesses the same spatial regularity as the initial data. 

            \vspace{2mm}
            \noindent \textbf{Case 1: }$(k,s)\in \A_{a}^- 
            \cup \A_a^0$. 

            \vspace{2mm}
            \noindent \textit{Step 1: Fixed point at low regularity.} Define 
            $$k'= \sup \{k_0 \in (-\infty,k] ; (k_0,s) \in \A^0_{a}\}^-.$$

            Hence, $(k',s)\in \A^0_{a}$. Let $\rho = k-k'$, $R>0$, and define
            $$B_R = \{(u,v)\in (X^{k,b-\rho/3} \cap X^{k',b}) \times Y^{s,b} ; \| u\|_{X^{k,b-\rho/3}} + \| u\|_{X^{k',b}} + \| v\|_{Y^{s,b}} \leq R\}. $$

            Let $\eta \in C_c^\infty(\R)$ with
            $\eta \equiv 1$ in $[-1,1]$ and $\eta_T(t) =\eta(t/T)$. Define the map $\Theta(u,v) \coloneqq (w,z)$, where 
            $$w(t) = \eta(t)\left(e^{-at\partial_x^3}u_0+ \B^u- \B^u_0 + \sum_{j=0}^3 \int_0^t \eta_T(t') e^{-a(t-t')\partial_x^3} \Nl_j^u(t') dt'\right), $$
            $$z(t) = \eta(t)\left(e^{-t\partial_x^3}v_0+  \int_0^t \eta_T(t') e^{-(t-t')\partial_x^3} (u v_x) dt'\right).$$

            By Lemma  \ref{lemma multilinear estimate uv_x}, 
            $$\| z\|_ {Y^{s,b}} \lesssim \|v_0\|_{H^s}+ T^{0^+}\|u\|_{X^{k',b}} \|v\|_{Y^{s,b}} \lesssim  \|v_0\|_{H^s}+ T^{0^+}R^2. $$

            Regarding $w$, first observe that
            \begin{equation}\label{inequality argument in X^k,b^-}
                \left \|\int_0^t \eta_T(t') e^{-a(t-t')\partial_x^3} \Nl_{3,\ll}^u(t') dt' \right \|_{X^{k,b-\rho/3}} \lesssim\left \|\int_0^t \eta_T(t') e^{-a(t-t')\partial_x^3} \Nl_{3,\ll}^u(t') dt' \right \|_{X^{k,b^-}} 
            \end{equation}
            $$  \lesssim \|\eta_T\ \Nl_{3,\ll}^u  \|_{X^{k,b^--1}} \lesssim T^{0^+}\|\Nl_{3,\ll}^u  \|_{X^{k,c}} ,$$
            where $b^--1< c <b-1$. Hence, by Lemma \ref{lemma time smoothing kdv}, we obtain
            $$ T^{0^+}\|\Nl_{3,\ll}^u  \|_{X^{k,c}}  \lesssim T^{0^+} \|u\|_{X^{k,b-\rho/3}}\|u \|_{X^{k-\rho,b}} = T^{0^+} \|u \|_{X^{k,b-\rho/3}}\|u \|_{X^{k',b}} \lesssim T^{0^+} R^2,$$
            
            Similarly, we have 
            $$\left \|\int_0^t \eta_T(t') e^{-a(t-t')\partial_x^3} \Nl_{3,\gtrsim}^u(t') dt' \right \|_{X^{k,b-\rho/3}} \lesssim T^{0^+}\|\Nl_{3,\gtrsim}^u  \|_{X^{k,b'}} \lesssim T^{0^+} \|u \|_{X^{k-\rho,b}}^2 =  T^{0^+} \|u \|_{X^{k',b}}^2 \lesssim T^{0^+} R^2.$$ 

            By Lemma \ref{lemma kdv estimate}, it follows that
            $$\left \|\int_0^t \eta_T(t') e^{-a(t-t')\partial_x^3} \Nl_{3}^u(t') dt' \right \|_{X^{k',b}} \lesssim T^{0^+} R^2.  $$

            By Lemmas \ref{lemma N_0^u} and \ref{Lemma N_1^u and N_2^u}, there exists $C_{\delta^u}>0$ such that
            \begin{equation}\label{equation N_j^u in X^k,b}
                \left \|\sum_{j=0}^2\int_0^t \eta_T(t') e^{-a(t-t')\partial_x^3} \Nl_{j}^u(t') dt' \right \|_{X^{k,b}} \leq C_{\delta^u} T^{0^+}R^2(1+R).
            \end{equation}

            Consequently, using \eqref{inequality bound N_3 by N_3,<< and N_3,>} and Lemmas \ref{lemma B^u for a<1/4} and \ref{lemma B_0^u for a<1/4}, there exists $C>0$ such that
            \begin{equation}\label{inequality contraction}
                \|w \|_{X^{k,b-\rho/3}} + \|w \|_{X^{k',b}} \leq C(\|u_0\|_{H^k} + (\delta^u)^{0^+}\|v_0\|^2_{H^s} + (\delta^u)^{0^+}R^2) +C_{\delta^u} T^{0^+}R^2(1+R).
            \end{equation}

            Assume $C>1$. We take $R>2C(\|u_0\|_{H^k}+ \|v_0\|_{H^s})$ and choose $(\delta^u)^{0^+} < (8CR)^{-1}. $
            Finally, taking $T$ sufficiently small, we have $\Theta: B_R\rightarrow B_R$. A similar argument shows that $\Theta$ is a strict contraction. Thus, $\Theta$ has a unique fixed point
            $$(u,v)\in (X^{k,b-\rho/3} \cap X^{k',b}) \times Y^{s,b}.$$
            
            \noindent\textit{Step 2: Gain of regularity.} We now show that $(u,v) \in C([0,T], H^k_x) \times C([0,T], H^s_x)$. Indeed, the fact that $v \in Y^{s,b}$ with $b>1/2$ implies $v \in C([0,T], H^s_x)$. By Lemma \ref{lemma B^u in H^k}, we have $\B^u, \B^u_0 \in C([0,T], H^k_x)$. Since $u \in X^{k',b}$ and $v \in Y^{s,b}$, the spatial smoothing established in Lemmas \ref{lemma N_0^u} and \ref{Lemma N_1^u and N_2^u} ensures that
            $$\sum_{j=0}^2\int_0^t \eta_T(t') e^{-a(t-t')\partial_x^3} \Nl_{j}^u(t') dt' \in X^{k,b}, $$
            given our choice of $\rho = k-k'$. Similarly, following the argument in \eqref{inequality argument in X^k,b^-}, Lemma \ref{lemma time smoothing kdv} ensures that
            $$\int_0^t \eta_T(t') e^{-a(t-t')\partial_x^3} \Nl_{3}^u(t') dt' \in X^{k,b^-}, $$
            where we choose $b^->1/2$. Hence, we obtain $u \in C([0,T], H^k_x)$. The analytic dependence of the flow map on the initial data follows from the multilinear estimates by a reasoning analogous to the fixed-point argument.

            \vspace{2mm}
            \noindent \textbf{Case 2: }$(k,s)\in \A_{a}^+ \cup \A_a^0$.
            
            \vspace{2mm}
            \noindent
            \textit{Step 1: Fixed point at low regularity.} Define 
            $$s'= \sup \{s_0 \in (-\infty,s] ; (k,s_0) \in \A_a^0\}^-.$$
            
            Thus, $(k,s')\in \A^0_{a}$. According to the lemmas in Section \ref{section multilinear estimates}, we have sufficient smoothing to achieve the necessary gain in regularity to pass from $s'$ to $s$ for all nonlinear terms in the equation for $v$ in \eqref{equation z formula}, with the exception of the term $\Nl_0^v$. This justifies the requirement for Lemma \ref{lemma N_0^v com termo de bordo}, as will become clear throughout the argument. Let $R>0$ and define the ball
            $$B_R = \{(u,v)\in X^{k,b} \times Y^{s',b} ;  \| u\|_{X^{k,b}} + \| v\|_{Y^{s',b}} \leq R\}. $$
            
            Define the map $\Theta(u,v) \coloneqq (w,z)$, where 
            $$w(t) = \eta(t)\left(e^{-at\partial_x^3}u_0 +  \int_0^t \eta_T(t') e^{-a(t-t')\partial_x^3} \partial_x(u^2+ v^2) dt'\right), $$
            \begin{equation}\label{equation z formula}
                z(t) = \eta(t)\left(e^{-t\partial_x^3}v_0+ \B^v- \B^v_0+ \sum_{j=0}^3 \int_0^t \eta_T(t') e^{-(t-t')\partial_x^3} \Nl_j^v(t') dt'\right).
            \end{equation}
            
            By Lemmas \ref{lemma kdv estimate} and \ref{lemma multilinear estimate delx(v1v2)}, we have
            $$\| w\|_ {X^{k,b}} \lesssim \|u_0\|_{H^k}+ T^{0^+}(\|u\|_{X^{k,b}}^2+ \|v\|_{Y^{s',b}}^2) \lesssim  \|u_0\|_{H^k}+ T^{0^+}R^2. $$
            
            For $z$, by Lemmas \ref{lemma N_1^v}, \ref{lemma N_2^v}, and \ref{lemma N_3^v}, there exists $C_{\delta^v}>0$ such that
            \begin{equation}
                \left \|\sum_{j=1}^3\int_0^t \eta_T(t') e^{-(t-t')\partial_x^3} \Nl_{j}^v(t') dt' \right \|_{Y^{s,b}} \leq C_{\delta^v} T^{0^+}R^3.
            \end{equation}
            
            Furthermore, Lemma \ref{lemma N_0^v} implies
            $$\left \|\int_0^t \eta_T(t') e^{-(t-t')\partial_x^3} \Nl_{0}^v(t') dt' \right \|_{Y^{s',b}} \leq C_{\delta^v} T^{0^+}R^2. $$
            Consequently, by Lemmas \ref{lemma B^v for a<1/4} and \ref{lemma B_0^v}, there exists $C>0$ such that       
            $$ \|z \|_{X^{s',b}} \leq C (\|v_0\|_{H^{s'}} + (\delta^v)^{0^+}\|u_0\|_{H^{k}} \|v_0\|_{H^{s'}} + (\delta^v)^{0^+}R^2) +C_{\delta^v} T^{0^+}R^2(1+R)   . $$
            
            Arguing as in \eqref{inequality contraction}, we conclude that $\Theta$ is a strict contraction. Thus, $\Theta$ has a unique fixed point
            $$(u,v)\in X^{k,b} \times Y^{s',b}.$$
            
            \noindent\textit{Step 2: Regularity gain.} We now show that $(u,v) \in C([0,T], H_x^k) \times C([0,T], H_x^s).$ Indeed, the fact that $u \in X^{k,b}$ with $b>1/2$ implies $u \in C([0,T], H_x^k)$. By Lemma \ref{lemma B_0^v}, we obtain $\B^v_0 \in Y^{s,b}$. From Lemmas \ref{lemma N_1^v}, \ref{lemma N_2^v}, and \ref{lemma N_3^v}, it follows that
            \begin{equation}\label{equation v - terms in Y^s,b}
                \sum_{j=1}^3\int_0^t \eta_T(t') e^{-(t-t')\partial_x^3} \Nl_{j}^v(t') dt'  \in Y^{s,b}
                \implies v- \eta(t)\left(\B^v+  \int_0^t \eta_T(t') e^{-(t-t')\partial_x^3} \Nl_0^v(t') dt'\right)\in Y^{s,b}.
            \end{equation}
            
            Define the operator $L_u: Y^{s,b} \rightarrow Y^{s,b}$ by
            \begin{equation}\label{equation L_u}
                L_u(f) = f - \eta(t)\int_0^t \eta_T(t') e^{(t-t')\partial_x^3} \Nl_0^v[u,f](t') dt',
            \end{equation}
            which is well-defined, since Lemma \ref{lemma N_0^v} implies
            \begin{equation}\label{equation L_u in Y^s,b}
                \|L_u(f)- f \|_{Y^{s,b}} \leq CT^{0^+}\|u\|_{X^{k,b}}\|f\|_{Y^{s,b}} < \frac{1}{2}\|f\|_{Y^{s,b}},
            \end{equation}
            for sufficiently small $T$. \footnote{Note that changing $T$ also changes $u$, which in turn modifies $L_u$. Thus, $T$ must be chosen small enough before defining $u$. This is achieved by ensuring $T^{0^+}< \frac{R}{2C}$ prior to applying the fixed point theorem.} 
            
            This inequality also implies that $L_u$ is invertible. Let $r \coloneqq v-\eta(t)\B^v$. Then, Lemma \ref{lemma N_0^v com termo de bordo} and \eqref{equation v - terms in Y^s,b} imply
            $$L_u(r)= r- \eta(t)  \int_0^t \eta_T(t') e^{-(t-t')\partial_x^3} \Nl_0^v[u,r](t') dt'\in Y^{s,b}.$$
            
            Since $L_u$ is invertible, we conclude that $r \in Y^{s,b}$. Thus, $v \in C([0,T], H_x^s)$ by Lemma \ref{lemma B^v in H^s}. The analytic dependence on the initial data follows from the multilinear estimates by a reasoning analogous to the fixed point argument.

            Moreover, for smooth solutions, the IBPS coincides with the classical strong solution, as integration by parts is justified in this setting. Consequently, the solution established in Proposition \ref{proposition direct LWP} agrees with the IBPS by density and the multilinear estimates provided in Subsection \ref{section bilinear estimates} and Section \ref{section well posedness in A_a}. Furthermore, a density argument combined with these estimates ensures that any continuous extension of the flow to initial data with regularity in $\A_a$ is unique; thus, it remains independent of the particular choice of $\delta^u$ and $\delta^v$.

        \end{proof}

    \subsection{Global well-posedness}
        We now proceed to the proof of Theorem \ref{theorem GWP}.

        \begin{proposition}[Persistence of regularity]\label{proposition persistence} 
            Suppose $a \in \R \setminus\{0,1\}$, and let $(k,s), (k',s') \in \A_a$ satisfy $k'\geq k$ and $s'\geq s$. For initial data $(u_0,v_0) \in H^{k'}(\R) \times H^{s'}(\R)$, let $(u,v)$ and $(u',v')$ denote the maximal solutions provided by Theorem \ref{theorem improved LWP} at the regularity levels $(k,s)$ and $(k',s')$, respectively. Then, their maximal times of existence coincide and $(u,v) = (u',v')$.
        \end{proposition}

        \begin{proof}
            \textit{Step 1. Local persistence in $\A^0_a$.} For $(k,s), (k',s') \in \A^0_a$, the solutions are given by Proposition \ref{proposition direct LWP} (by the proof of Theorem \ref{theorem general LWP}). Without loss of generality, assume $(k,s')$ and $(k',s) \in \A_a^0$. Suppose, for the sake of contradiction, that there exists initial data $(u_0,v_0) \in H^{k'}(\R) \times H^{s'}(\R)$ such that $T'_{max} < T_{max}$, where $T'_{max}$ and $T_{max}$ are the maximal existence times of $(u',v')$ and $(u,v)$, respectively. Consequently, there exists $R>0$ such that
            $$ \sup_{t \in [0,T'_{max}]} \| (u(t),v(t)) \|_{H^k \times H^s} \leq R. $$

            Let $T>0$ denote the local existence time for initial data whose $H^k \times H^s$ norm is bounded by $R$, and assume $T'_{max} < T$ (after a possible time translation of $(u,v)$ and $(u',v')$). By uniqueness, $(u',v')$ coincides on $[0,T'_{max}]$ with the $H^k \times H^s$ local solution, which satisfies
            $$\|u\|_{X^{k,b}} + \|v\|_{Y^{s,b}} \lesssim R. $$

            Lemmas \ref{lemma kdv estimate}, \ref{lemma multilinear estimate delx(v1v2)} and  \ref{lemma multilinear estimate uv_x} then yield
            $$ \left\| \int_0^t \eta_T(t') e^{-a(t-t')\partial_x^3} \partial_x(u^2+ v^2) \,dt' \right\|_{X^{k',b}} \lesssim T^{0^+}\left(  \| u  \|_{X^{k,b}}\| u  \|_{X^{k',b}}+ \| v  \|^2_{Y^{s,b}}\right) \lesssim T^{0^+}R\|u\|_{X^{k',b}} + T^{0^+}R^2$$
            and
            $$ \left\| \int_0^t \eta_T(t') e^{-(t-t')\partial_x^3} uv_x \,dt' \right\|_{Y^{s',b}} \lesssim T^{0^+}  \| u  \|_{X^{k,b}}\| v  \|_{Y^{s',b}} \lesssim T^{0^+}R\|v\|_{Y^{s',b}}. $$

            It follows that
            $$\| u\|_{X^{k',b}} + \| v\|_{Y^{s',b}} \lesssim \|u_0\|_{H^{k'}} + \|v_0\|_{H^{s'}} + T^{0^+}R^2 + T^{0^+}R(\| u\|_{X^{k',b}} + \| v\|_{Y^{s',b}}).$$

            Choosing $T$ sufficiently small depending on $R$, we obtain
            $$\| u\|_{H^{k'}} + \| v\|_{H^{s'}} \lesssim \|u_0\|_{H^{k'}} + \|v_0\|_{H^{s'}} + T^{0^+}R^2 < \infty.$$

            This contradicts the blowup alternative, thereby establishing the persistence property.

            \textit{Step 2. Global persistence in $\A_a^0$.} Suppose $(k,s), (k',s') \in \A^0_a$. By the argument of Step 1, there exists a finite sequence $\{(k_j,s_j)\}_{j=1}^n \subset \A^0_a$ with $(k_1,s_1) = (k,s)$ and $(k_n,s_n) = (k',s')$ such that local persistence holds between $(k_j, s_j)$ and $(k_{j+1}, s_{j+1})$ for each $j$. This implies persistence between $(k,s)$ and $ (k',s')$.

            \textit{Step 3. Global persistence in $\A_a$.} Suppose $(k,s), (k',s') \in \A_a$. If $(k,s) \notin \A_a^0$, the proof of Theorem \ref{theorem improved LWP} ensures the existence of $(k_0,s_0) \in \A_a^0$ with $k_0 \leq k$ and $s_0 \leq s$ such that persistence is true between $(k_0,s_0)$ and $(k,s)$. If $(k,s) \in \A_a^0$, we simply set $(k_0,s_0) = (k,s)$.

            Similarly, we define $(k_0', s_0')$ such that persistence is valid between $(k_0',s_0')$ and $(k',s')$ with $k'_0 \geq k_0$ and $s'_0 \geq s_0$. By Step 2, persistence holds between the regularity levels $(k_0, s_0)$ and $(k_0',s_0')$. Therefore, the maximal times of existence for all four regularity levels must coincide.
        \end{proof}

        \begin{proof}[Proof of Theorem \ref{theorem GWP}]
            For $a \neq 1/4$, the conservation law \eqref{equation mass} ensures global well-posedness in $L^2 \times L^2$. The conclusion of the theorem then follows from Proposition \ref{proposition persistence}. For $a=1/4$, the argument is analogous, leveraging the conservation laws \eqref{equation mass} and \eqref{equation energy}.
        \end{proof}

\section{Ill-posedness}\label{section ill-posedness}

    In this section, we establish the sharpness of our results by proving Theorem \ref{theorem ill-posedness}. The optimality is deduced in all cases except for $a=-1/8$ (as noted in Remark \ref{remark ill-posedness for a=-1/8}). Following \cite{Bourgain1997PeriodicKD} (Section 6), it is known that \eqref{equation hirota system} is $C^3$-ill-posed for $k<-3/4$ by considering initial data of the form $(u_0,0)$. The sharpness of the remaining obstructions is addressed in the subsequent lemmas.

    The underlying heuristic is as follows. The equations for the integrated-by-parts strong solution (see \eqref{equation IBP profile u} and \eqref{equation IBP profile v}) contain boundary and integral terms in addition to the initial condition. If an obstruction arises within a boundary term or an integral term with a small phase, it likely constitutes a genuine obstruction to the flow. The counterexamples we construct below are built to exploit precisely these obstructions. Furthermore, to prove the failure of $C^2$ (or $C^3$) regularity, we employ the standard arguments of \cite{ BEJENARU2006228, Bourgain1997PeriodicKD,  TZVETKOV19991043}.

    In the subsequent proofs, we adopt the notation:
    \[\hat {f} \coloneqq \F_xf.\]

    \begin{lemma}\label{lemma s-3 <= k}
        Let $a \in \R \setminus \{0,1\}$. Suppose that the flow of the system \eqref{equation hirota system} is $C^2$ at the origin. Then, $s \leq k+3$.
    \end{lemma}

    \begin{proof}
        For $N \gg 1$, define the initial data in the frequency domain by 
        $$\hat{u}_0 = \indicatrix_{[N,N+1]}, \quad \hat{v}_0 = \indicatrix_{[1,3]}.$$

        Consider the initial data $(\epsilon u_0, \epsilon v_0)$. Since the flow is assumed to be $C^2$ at the origin, then for any $t \in \R$, we may define a continuous bilinear operator $I: H^k \times H^s \rightarrow H^s$ by
        $$I[u_0, v_0] = \frac{\partial^2 v}{\partial \epsilon^2} \bigg|_{\epsilon=0} = \int_0^t e^{-(t-t')\partial^3_x}\big( (e^{-at'\partial^3_x} u_0) \partial_x(e^{-t'\partial^3_x} v_0) \big) \, dt',$$
        which, in frequency variables, takes the form
        $$ie^{it\xi^3}\int_{\xi = \xi_1+ \xi_2} \int_0^t \xi_2 e^{it'\Phi^v} \hat{u}_0(\xi_1)\hat{v}_0(\xi_2) \, dt' \, d\xi_1 = ie^{it\xi^3}\int_{\xi =\xi_1+ \xi_2} \xi_2 \frac{e^{it\Phi^v}-1}{i\Phi^v} \hat{u}_0(\xi_1)\hat{v}_0(\xi_2) \, d\xi_1.$$

        Moreover, $|\xi_2| \sim 1$ and the assumption $a \neq 1$ implies that $|\Phi^v| = |-\xi^3 + a\xi_1^3 + \xi_2^3| \sim N^3$. Let $c > 0$ be a sufficiently small constant and set $t = cN^{-3}$. By the Mean Value Theorem, it follows that
        \begin{equation}\label{equation lower bound phase}
            \left| \real \left( \frac{e^{it\Phi^v}-1}{i} \right) \right| = |\sin (t\Phi^v)| \gtrsim  |t\Phi^v | \implies \real \left( \frac{e^{it\Phi^v}-1}{i\Phi^v} \right) \gtrsim N^{-3}.
        \end{equation}

        Moreover, as $\xi_2$, $\hat{u}_0$, and $\hat{v}_0$ do not change sign, we have
        $$\|I\|_{H^s} \gtrsim N^{s-3} \|\hat{u}_0 \ast \hat{v}_0 \|_{L^2} \gtrsim N^{s-3}.$$

        To see this, observe that
        $$ \|\hat{u}_0 \ast \hat{v}_0 \|_{L^2}^2 =\int \left( \int \indicatrix_{[N, N+1]}(\xi_1) \indicatrix_{[1, 3]}(\xi_2) d\xi_1 \right)^2 d\xi \gtrsim \int_{N+2}^{N+3} \left( \int_{N}^{N+1} \indicatrix_{[1, 3]}(\xi_2) d\xi_1 \right)^2 d\xi \gtrsim 1,$$  
        since $\xi \in [N+2, N+3]$ and $\xi_1 \in [N,N+1]$ imply $\xi_2 = \xi-\xi_1 \in [1,3]$. On the other hand, the continuity of $I$ requires
        $$ \|I\|_{H^s} \lesssim \|u_0\|_{H^k}\|v_0\|_{H^s} \lesssim N^k. $$
        Therefore, we must have $s \leq k+3$, which concludes the proof.
    \end{proof}

    \begin{remark}
        The preceding argument can be generalized to systems of the form
        $$\left\{ 
        \begin{aligned} 
            u_t - iL_1(D)u &= \Nl_1(u,v),\\
            v_t - iL_2(D)v &= \Nl_2(u,v) + \Nl_3(u,v),
        \end{aligned} \right.$$
        defined for $(t,x) \in \R \times \R$, where $L_1$ and $L_2$ are linear differential operators with Fourier symbols $L_1(\xi)$ and $L_2(\xi)$, respectively. Assume that:
        \begin{enumerate}
            \item $\F_x(\Nl_2(u,v))(\xi) = \int_{\xi=\xi_1+\xi_2} m(\xi_1,\xi_2)\F_x(u)\F_x(v) \, d\xi_1$ such that there exist $a < b$ with $|m(\xi_1,\xi_2)| \gtrsim |\xi_1|^p$ for all $\xi_2 \in [a,b]$.
            \item $\Nl_1$ is an $n$-linear form satisfying \eqref{equation form of the nonlinearity}.
            \item $\Nl_3$ is an $n'$-linear form with $n' > 2$ satisfying \eqref{equation form of the nonlinearity}.
            \item $|\Phi^v| = |-L_2(\xi) + L_1(\xi_1) + L_2(\xi_2)| \sim |\xi_1|^l$ for $|\xi_1| \gg |\xi_2|$.
        \end{enumerate}

        Under these assumptions, the $C^2$-regularity of the solution map at the origin entails $s - l + p \leq k$. For instance, this criterion is satisfied by the nonlinearity $\Nl_2(u,v) = \partial^p_x (u)  v$ subjetcted to the dispersion relations $L_1(\xi) = \xi^3$ and $L_2(\xi) = -\xi^2$. Consequently, the same constraint $s \leq k+3$ arises in the NLS-KdV system investigated in \cite{correia2025sharplocalwellposednessschrodingerkortewegde}:    \begin{equation}\label{equation NLS-KdV}
            \left\{ 
            \begin{aligned} 
                u_t + u_{xxx} &= \alpha\partial_x(u^2) + \beta \partial_x(|v|^2), \\
                iv_t + v_{xx} &= \gamma uv + \theta|v|^2v.
            \end{aligned} \right.
        \end{equation}
        The reader is referred to \cite{correia2025sharplocalwellposednessschrodingerkortewegde} for further justification, which is essentially identical to the argument used in Lemma \ref{lemma s-3 <= k}.
    \end{remark}

    \begin{lemma}
        Let $a  \in \R \setminus \{0,1\}$. Suppose the flow of the system \eqref{equation hirota system} is $C^2$ at the origin. Then, $ s \geq k-2$.
    \end{lemma}
    \begin{proof}
        Assume, for the sake of contradiction, that there exists $\rho \in \R$ such that
        \begin{equation}\label{equation contradiction hyotheses}
            s+\frac{1}{2} < \rho < k-2+ \frac{1}{2}.
        \end{equation}
        
        For $N \gg 1$, define 
        $$\hat{u}_0 = 0, \hspace{2mm} \hat{v}_0 = \frac{1}{\langle \xi \rangle^{\rho}}(\indicatrix_{[-(1-b)N,(1-b)N]}+ \indicatrix_{[bN,N]}), $$
        where $b < 1$ is chosen sufficiently close to $1$ (its precise value will be specified later). Then, $\|v_0\|_{H^s} \lesssim 1$.
        Let the initial data be $(0, \epsilon v_0)$, and define
        $$I[v_0] = \frac{\partial^2 u}{\partial \epsilon^2} \bigg|_{\epsilon=0} = \int_0^t e^{-a(t-t')\partial^3_x}\partial_x((e^{-t'\partial^3_x} v_0) (e^{-t'\partial^3_x} v_0))dt',$$
        which, in frequency variables, takes the form
        $$ie^{iat\xi^3}\int_{\xi = \xi_1+ \xi_2} \int_0^t \xi e^{it'\Phi^u_1} \hat{v}_0(\xi_1)\hat{v}_0(\xi_2) dt' d\xi_1 = ie^{iat\xi^3}\int_{\xi =\xi_1+ \xi_2}  \xi \frac{e^{it\Phi_1^u}-1}{i\Phi_1^u} \hat{v}_0(\xi_1)\hat{v
        }_0(\xi_2)\, d\xi_1 . $$

        Fix $\xi \in [bN, N]$. If 
        $\xi_1 \in [bN, N]$, it follows that
        $$\xi_2 = \xi-\xi_1 \in [-(1-b)N,(1-b)N]. $$

        Moreover, for $b$ sufficiently close to 1, we have $|\Phi^u_1| \sim N^3$. Indeed, by the triangle inequality,
        $$|\Phi^u_1| = |-a\xi^3 +\xi_1^3 + \xi_2^3| \geq |1-a||\xi|^3 - |\xi^3_1-\xi^3|- |\xi_2|^3 $$
        $$\implies |\Phi^u_1| \geq (|1-a|b^3- (1-b^3)-(1-b)^3)N^3.$$
        
        Let $c > 0$ be a sufficiently small constant and set $t = cN^{-3}$. By the Mean Value Theorem,
        $$ \left| \real  \frac{e^{it\Phi_1^u}-1}{i} \right| = |\sin (t\Phi_1^u)| \gtrsim  |t\Phi_1^u | \implies \real  \frac{e^{it\Phi_1^u}-1}{i\Phi^u_1} \gtrsim N^{-3}. $$
        
        Hence,
        $$\left|\int_{\xi =\xi_1+ \xi_2}  \xi \frac{e^{it\Phi_1^u}-1}{i\Phi_1^u} \hat{v_0}(\xi_1)\hat{v
        }_0(\xi_2) d\xi_1 \right| \gtrsim  N^{-2}\int_{|\xi_2|<(1-b)N} \frac{1}{\langle \xi_1 \rangle^{\rho}\langle \xi_2 \rangle^{\rho}}\, d\xi_1, $$
        since the real part of the integrand does not change sign. It follows that
        $$\left|\int_{\xi =\xi_1+ \xi_2}  \xi \frac{e^{it\Phi_1^u}-1}{i\Phi_1^u} \hat{v_0}(\xi_1)\hat{v
        }_0(\xi_2) d\xi_1 \right|  \gtrsim  N^{-2-\rho}\int_{|\xi_2|<(1-b)N} \frac{1}{\langle \xi_2 \rangle^{\rho}} d\xi_2  \gtrsim N^{-2-\rho}\int_{|\xi_2|<1} \frac{1}{\langle \xi_2 \rangle^{\rho}} d\xi_2 \gtrsim N^{-2-\rho}.$$  
       
        Consequently,
        $$\|I\|_{H^k([b N, N])} \gtrsim N^{k-2-\rho} \left( \int_{bN}^N d\xi \right)^{1/2} \gtrsim N^{k-2-\rho+1/2}.$$

        On the other hand, the $C^2$ regularity of the solution map implies $\|I\|_{H^k([b N,  N])} \lesssim  \|v_0\|_{H^s}^2 \lesssim 1$. Thus, 
        $$0\geq k-2-\rho+\frac{1}{2} \implies \rho \geq k-2+\frac{1}{2}.$$

        This contradicts \eqref{equation contradiction hyotheses}, and we conclude that $s \geq k-2$.
    \end{proof}
    \begin{remark}
        By a similar argument (noting that the estimate for the phase $|\Phi_1^u| \sim N^3$ is slightly more involved in the present case), the obstruction $s \geq k-2$ was also obtained for the NLS-KdV system \eqref{equation NLS-KdV} in \cite{correia2025sharplocalwellposednessschrodingerkortewegde}.
    \end{remark}

    \begin{lemma}\label{lemma ill-posedness s<k/2-3/4}
        Let $a \in \R \setminus\{0\}$. Suppose the flow of the system \eqref{equation hirota system} is $C^2$ at the origin. Then, $s \geq k/2-3/4$.
    \end{lemma}
    \begin{proof}
        Let $\delta > 0$ be sufficiently small. For $N \gg 1$, define 
        $$\hat{u}_0 = 0, \hspace{2mm} \hat{v}_0 = \indicatrix_{[-N-\delta N,-N+\delta N]} + \indicatrix_{[N+ \delta N,N + 2\delta N]}. $$

        Considering the initial data $(0, \epsilon v_0)$, we define the operator
        \begin{equation}\label{equation operator I of u}
            I[v_1, v_2] = \frac{\partial^2 u}{\partial \epsilon^2} \bigg|_{\epsilon=0} = \int_0^t e^{-a(t-t')\partial^3_x}\partial_x((e^{-t'\partial^3_x} v_0) (e^{-t'\partial^3_x} v_0))\,dt',
        \end{equation}
        which, in frequency variables, is given by
        $$ie^{iat\xi^3}\int_{\xi = \xi_1+ \xi_2} \int_0^t \xi e^{it'\Phi_1^u} \hat{v}_0(\xi_1)\hat{v}_0(\xi_2)\, dt'\, d\xi_1 = ie^{iat\xi^3}\int_{\xi =\xi_1+ \xi_2}  \xi \frac{e^{it\Phi_1^u}-1}{i\Phi_1^u} \hat{v}_0(\xi_1)\hat{v}_0(\xi_2)\, d\xi_1 . $$

        For $\xi \in [\delta N, 2\delta N]$, the factorization in \eqref{equation factorization phi^u} implies $|\Phi_1^u| \sim |N|^3$, provided $\delta$ is sufficiently small. Let $c > 0$ be a small constant and set $t = cN^{-3}$. It then follows from the Mean Value Theorem that
        $$ \real \frac{e^{it\Phi_1^u}-1}{i\Phi_1^u} \gtrsim N^{-3}.$$

        Furthermore, since $\xi$ and $\hat{v}_0$ do not change sign, we obtain
        $$\|I\|_{H^k} \gtrsim N^{k-2} \|\hat{v}_0 \ast \hat{v}_0 \|_{L^2([\delta N,2\delta N])}.$$

        Additionally, 
        $$\|\hat{v}_0 \ast \hat{v}_0 \|_{L^2([\delta N,2\delta N])}^2 \gtrsim \int_{\delta N}^{2\delta N}\left(\int_{N+\delta N}^{N+2\delta N} \hat{v}_0(\xi_1)\hat{v}_0(\xi_2)d\xi_1\right)^2d\xi\gtrsim N^{3},$$ which implies $\|I\|_{H^k} \gtrsim N^{k-1/2}$. On the other hand, the assumption that the solution map is $C^2$ at the origin yields $ \|I\|_{H^k} \lesssim \|v_0\|_{H^s}^2 \lesssim N^{2s+1}. $ Therefore, we must have $s \geq k/2 - 3/4$, which concludes the proof.
        
    \end{proof}

    In the case $a=1/4$, we have the following lemmas:

    \begin{lemma}\label{lemmma ill-posedness s>=k/2+3/8 a=1/4}
        Let $a = 1/4$. Suppose the flow of the system \eqref{equation hirota system} is $C^2$ at the origin. Then, $s \geq k/2+3/8$.
    \end{lemma}

    \begin{proof}
        For $a=1/4$ and $\xi=\xi_1+\xi_2$, we observe that
        \begin{equation}\label{equation phase}
            \Phi^u_1 = -\frac{\xi^3}{4} + \xi_1^3 + \xi_2^3 = \frac{3}{4} \xi (\xi_1-\xi_2)^2,  
        \end{equation}
        
        For $N \gg 1$, let 
        $$\hat{u}_0 = 0, \hspace{2mm} \hat{v}_0 = \indicatrix_{\left[N,N+N^{-\frac{1}{2}}\right]}. $$

        Considering the initial data as $(0,\epsilon v_0)$, we define
        
        $$I[v_0] = \frac{\partial^2 u}{\partial \epsilon^2} \bigg|_{\epsilon=0} = \int_0^t e^{-a(t-t')\partial^3_x}\partial_x((e^{-t'\partial^3_x} v_0) (e^{-t'\partial^3_x} v_0))\,dt',$$
        which, in frequency variables, is given by
        $$ie^{iat\xi^3}\int_{\xi = \xi_1+ \xi_2} \int_0^t \xi e^{it'\Phi_1^u} \hat{v}_0(\xi_1)\hat{v}_0(\xi_2) \,dt' \,d\xi_1.$$

        Note that $|\xi| \sim N$ and $|\xi_1 - \xi_2| \lesssim N^{-1/2}$; thus, $|\Phi_1^u| \lesssim 1$ by \eqref{equation phase}. For $t > 0$ sufficiently small such that $|t\Phi_1^u| \ll 1$, we have
        $$\|I\|_{H^k}^2 \gtrsim \int \langle \xi \rangle^{2k} \left(t\xi\int\hat{v_0}(\xi_1)\hat{v}_0(\xi_2)d\xi_1\right)^2d\xi \gtrsim N^{2k+2}\int   \left(\int\hat{v_0}(\xi_1)\hat{v}_0(\xi_2)d\xi_1\right)^2d\xi. $$

        Furthermore, observe that 
        $$N + \frac{1}{2\sqrt{N}} \leq \xi_1\leq N + \frac{1}{\sqrt{N}} \hspace{2mm }\text{ and } \hspace{2mm } 2N + \frac{1}{\sqrt{N}} \leq \xi\leq 2N + \frac{3}{2\sqrt{N}}$$
        $$ \implies \xi_2 = \xi-\xi_1 \in \left[N,N+\frac{1}{\sqrt{N}}\right].$$

        Therefore,
        $$ \int   \left(\int\hat{v_0}(\xi_1)\hat{v}_0(\xi_2)d\xi_1\right)^2d\xi \gtrsim  \int_{2N+\frac{1}{\sqrt{N}}}^{2N + \frac{3}{2\sqrt{N}}}   \left(\int_{N+\frac{1}{2\sqrt{N}}}^{N + \frac{1}{\sqrt{N}}}  d\xi_1\right)^2d\xi \gtrsim N^{-\frac{1}{2}-1}.$$

        This yields
        $\|I\|_{H^k} \gtrsim N^{k+1/4}$. On the other hand, the $C^2-$regularity of the flow at the origin entails $ \|I\|_{H^k} \lesssim  \|v_0\|_{H^s}^2 \lesssim N^{2s-1/2}. $ Consequently, we must have $s \geq k/2+3/8$.
        
    \end{proof}

    \begin{lemma}\label{lemma a=1/4 flow C^2 implies k >=3/4}
        Let $a = 1/4$. Suppose the flow of the system \eqref{equation hirota system} is $C^2$ at the origin. Then, $k \geq 3/4$.
    \end{lemma}

    \begin{proof}
        Since $a=1/4$ and $\xi=\xi_1+\xi_2$, we observe that
        \begin{equation}\label{equation phase v}
            \Phi^v = -\xi^3+ \frac{\xi_1^3}{4} + \xi_2^3 = -\frac{3}{4}\xi_1(\xi_1+2\xi_2)^2.  
        \end{equation}

        Note that if $|\xi_1|(\xi_1 + 2\xi_2)^2 \lesssim 1$, the phase $\Phi^v$ remains bounded. Specifically, for $N \gg 1$, define 
        
        $$\hat{u}_0 =\indicatrix_{\left[2N,2N+2N^{-\frac{1}{2}}\right]}, \hspace{2mm} \hat{v}_0 = \indicatrix_{\left[-N-N^{-\frac{1}{2}},-N\right]}.  $$

        Considering the initial data $(\epsilon u_0, \epsilon v_0)$, and since the solution map is assumed to be $C^2$ at the origin, for any $t \in \R$ we may define a continuous operator $I: H^k \times H^s \rightarrow H^s$ by
        $$I[u_0, v_0] = \frac{\partial^2 v}{\partial \epsilon^2} \bigg|_{\epsilon=0} = \int_0^t e^{-(t-t')\partial^3_x}((e^{-at'\partial^3_x} u_0) \partial_x(e^{-t'\partial^3_x} v_0))\,dt',$$
        which, in frequency variables, is given by
        $$ie^{it\xi^3}\int_{\xi = \xi_1+ \xi_2} \int_0^t \xi_2 e^{it'\Phi^v} \hat{u}_0(\xi_1)\hat{v}_0(\xi_2) \,dt' \,d\xi_1.$$

        Observe that $|\xi_1+ 2\xi_2| \lesssim  N^{-1/2}.$ Consequently, $|\Phi^v|\lesssim 1$ by \eqref{equation phase v}. For $t > 0$ sufficiently small such that $|t\Phi^v| \ll 1$, we have
        $$\|I\|_{H^s}^2 \gtrsim \int \langle \xi \rangle^{2s} \left(t\xi_2\int\hat{u_0}(\xi_1)\hat{v}_0(\xi_2)d\xi_1\right)^2d\xi \gtrsim N^{2s+2}\int   \left(\int\hat{u_0}(\xi_1)\hat{v}_0(\xi_2)d\xi_1\right)^2d\xi. $$

        Furthermore, note that 
        $$2N + \frac{1}{2\sqrt{N}} \leq \xi_1\leq 2N + \frac{1}{\sqrt{N}} \hspace{2mm }\text{ and } \hspace{2mm } N  \leq \xi\leq N + \frac{1}{2\sqrt{N}}$$
        $$ \implies \xi_2 = \xi-\xi_1 \in \left[-N-\frac{1}{\sqrt{N}},-N\right].$$

        Therefore,
        $$ \int   \left(\int\hat{u_0}(\xi_1)\hat{v}_0(\xi_2)d\xi_1\right)^2d\xi \gtrsim  \int_{N}^{N + \frac{1}{2\sqrt{N}}}   \left(\int_{2N+\frac{1}{2\sqrt{N}}}^{2N + \frac{1}{\sqrt{N}}}  d\xi_1\right)^2d\xi \gtrsim N^{-3/2}.$$

       This implies $\|I\|_{H^s} \gtrsim N^{s+\frac{1}{4}}$. On the other hand, the continuity of $I$ entails $\|I\|_{H^s} \lesssim \|u_0\|_{H^k} \|v_0\|_{H^s} \lesssim N^{k+s-\frac{1}{2}}$. Consequently, we must have $k \geq 3/4$, which concludes the proof.

    \end{proof}

    In the case $a \in (1/4,\infty) \setminus \{1\}$, we establish the following results:

    \begin{lemma}\label{lemma illposedness s>=k/2 a>1/4}
        Let $a \in (1/4,\infty) \setminus \{1\}$. Suppose the flow of the system \eqref{equation hirota system} is $C^2$ at the origin. Then $s \geq k/2$.
    \end{lemma}

    \begin{proof}
        In view of the factorization \eqref{equation factorization phi^u a >=1/4}, we have
        \begin{equation}
            \Phi^u_1 =  3\xi(\xi_1-\mu_a\xi)(\xi_1-(1-\mu_a)\xi),  
        \end{equation}
        where $\mu_a = \frac{1}{2} + \frac{\sqrt{3(4a-1)}}{6}$. 
        
        The obstruction $s \geq k/2$ stems from the regime $\xi_1 \simeq \mu_a\xi$, where $|\Phi^u_1|$ is small (cf. Subcase A2 in the proof of Lemma \ref{lemma multilinear estimate delx(v1v2)}). Since $\mu_a \notin \{0,1\}$, it follows that $|\xi| \sim |\xi_1| \sim |\xi_2|$ and $|\Phi^u_1| \sim |\xi^2_1(\xi_1 - \mu_a\xi)|$. Specifically, by choosing frequencies such that $|\xi - \mu_a\xi_1| \lesssim 1/|\xi_1|^2$, we ensure that $|\Phi^u_1| \lesssim 1$. For $N \gg 1$, we may set
        $$\xi_1 \in \left[N,N +\frac{1}{N^2}\right], \hspace{2mm} \xi \in \left[ \frac{N}{\mu_a},\frac{N}{\mu_a} +\frac{1}{N^2}\right]$$
        $$\implies \xi_2 = \xi-\xi_1 \in \left[\left(\frac{1}{\mu_a}-1\right)N-\frac{1}{N^2},\left(\frac{1}{\mu_a}-1\right)N +\frac{1}{N^2}\right].$$

        By defining
        $$\hat{u}_0 =0, \quad \hat{v}_0 = \indicatrix_{\left[N,N +\frac{1}{N^2}\right]} + \indicatrix_{\left[\left(\frac{1}{\mu_a}-1\right)N-\frac{1}{N^2},\left(\frac{1}{\mu_a}-1\right)N +\frac{1}{N^2}\right]} $$ 
        and proceeding as in the proof of Lemma \ref{lemmma ill-posedness s>=k/2+3/8 a=1/4}, we conclude that $s \geq k/2$. 
    \end{proof}

    \begin{lemma}\label{lemma a>1/4 k>=0}
        Let $a \in (1/4,\infty) \setminus \{1\}$. If the flow of the system \eqref{equation hirota system} is $C^2$ at the origin, then $k \geq 0$.
    \end{lemma}
    
    \begin{proof}
        In view of \eqref{equation factorization Phi^v for a>1/4}, we have
        \begin{equation}
            \Phi^v = -3\xi_1 (\xi-\mu_a\xi_1)(\xi-(1-\mu_a)\xi_1),  
        \end{equation}
        where $\mu_a = \frac{1}{2} + \frac{\sqrt{3(4a-1)}}{6}$. 
        
        The obstruction $k \geq 0$ arises in the region $\xi \simeq \mu_a\xi_1$, where $|
        \Phi^v|$ is small (cf. Subcase A3 in the proof of Lemma \ref{lemma multilinear estimate uv_x}). Since $\mu_a \notin \{0,1\}$, it follows that $|\xi| \sim |\xi_1| \sim |\xi_2|$ and $|\Phi^v| \sim |\xi^2_1(\xi_1 - \mu_a\xi)|$. Consequently, provided $|\xi - \mu_a\xi_1| \lesssim 1/|\xi_1|^2$, we obtain $|\Phi^v| \lesssim 1$. In fact, for $N \gg 1$, we may set
        $$\xi_1 \in \left[N,N +\frac{1}{N^2}\right], \hspace{2mm} \xi \in \left[ \mu_aN, \mu_aN +\frac{1}{N^2}\right]$$
        $$\implies \xi_2 = \xi-\xi_1 \in \left[\left(\mu_a-1\right)N-\frac{1}{N^2},\left(\mu_a-1\right)N +\frac{1}{N^2}\right].$$
    
        By defining
        $$\hat{u}_0 =\indicatrix_{\left[N,N +\frac{1}{N^2}\right]}, \quad \hat{v}_0 =   \indicatrix_{\left[\left(\mu_a-1\right)N-\frac{1}{N^2},\left(\mu_a-1\right)N +\frac{1}{N^2}\right]} $$ 
        and proceeding as in the proof of Lemma \ref{lemma a=1/4 flow C^2 implies k >=3/4}, we conclude that $k \geq 0$.
    \end{proof}

    The frequency region used to construct the counterexample in the following lemma does not appear in the proof of the multilinear estimates in Sections \ref{section bilinear estimates} and \ref{section multilinear estimates}. Instead, this regime emerges in the term $\partial_x(v^2)v_x$, which arises when substituting equation \eqref{equation profile u} into \eqref{equation profile v}. Its total phase is 
    \begin{equation}\label{phase mkdv}
        \Theta \coloneqq -\xi^3+\xi_2^3+\xi_{11}^3+\xi_{12}^3 = -3(\xi_2+\xi_{11})(\xi_2+\xi_{12})(\xi_{11}+\xi_{12}),
    \end{equation}
    which is small when one of the factors is small. Using this alongside Bourgain's argument in Section 6 of \cite{Bourgain1997PeriodicKD}, we prove the following:

    \begin{lemma}\label{lemma ill-posedness s<-3/4}
        Let $a \in(-\infty, 1/4)\setminus\left\{-\frac{1}{8},0\right\}$ and suppose the flow of the system \eqref{equation hirota system} is $C^3$ at the origin. Then, $s \geq -3/4$.
    \end{lemma}
    \begin{remark}\label{remark ill-posedness for s<-1}
         Let $a \in \R \setminus\{0\}$. By reasoning similar to that for Lemma \ref{lemma ill-posedness s<k/2-3/4}, one can show that if $s < -1$, there exists no flow that is $C^2$ at the origin. 
    \end{remark}

    \begin{proof} For $N \gg 1$, define
        \begin{equation}\label{equation v0}
            \hat{v}_0 = \indicatrix_{\left[N,N+\frac{1}{\sqrt{N}}\right]} + \indicatrix_{\left[-N+\frac{5}{4\sqrt{N}},-N+\frac{3}{2\sqrt{N}}\right]}.
        \end{equation}

        Set the initial data to $(0, \epsilon v_0)$. Since the flow is $C^3$ at the origin, for any $t \in \R$, we can define an operator $I: H^s \rightarrow H^s$ by
        $$I[v_0] \coloneqq \frac{\partial^3 v}{\partial \epsilon^3} = \int_0^t e^{-(t-t')\partial^3_x}\int_0^{t'} e^{-a(t'-q)\partial^3_x}\partial_x\left[(e^{-q\partial^3_x} v_0)^2\right] dq \hspace{0.5mm}\partial_x(e^{-t'\partial^3_x} v_0)\,dt',$$
        satisfying
        \begin{equation}\label{inequality continous operator I}
            \| I\|_{H^s} \lesssim \|v_0\|^3_{H^s}.
        \end{equation}
        
        In frequency variables, we have
        $$\hat{I}=-e^{it\xi^3}\int_{\xi = \xi_{11}+\xi_{12}+ \xi_2}\xi_1\xi_2\hat{v}_0(\xi_{11})\hat{v}_0(\xi_{12}) \hat{v}_0(\xi_2) \int_0^t  e^{it'\Phi^v} \int_0^{t'}  e^{iq\Phi^{u_1}_1} \,dq\,dt'\, d\xi_1 \, d\xi_{11},  $$
        where 
        $ \Phi^v = - \xi^3 + a \xi_1^3 + \xi_2^3 $ and $\Phi_1^{u_1} = -a\xi_1^3 + \xi_{11}^3+\xi_{12}^3$.
        Moreover,
        $$ \int_0^t  e^{it'\Phi^v} \int_0^{t'}  e^{iq\Phi^{u_1}_1} dqdt'  = \int_0^t  e^{it'\Phi^v} \left(\frac{e^{it'\Phi_1^{u_1}}-1}{i\Phi_1^{u_1}} \right) dt'  = \int_0^t  \frac{e^{it'\Theta}-e^{it'\Phi^v}}{i\Phi_1^{u_1}}  dt' =  \frac{1}{i\Phi_1^{u_1}}\int_0^t  e^{it'\Theta}  dt'+ \frac{e^{it\Phi^v}-1}{\Phi_1^{u_1}\Phi^v}, $$
        where $\Theta$ is defined in equation \eqref{phase mkdv}. Suppose $\xi = N + o(N)$. By \eqref{equation v0}, three cases arise:
        \begin{equation}\label{equation situtation 1}
            \xi_{11},\xi_{12} \in \left[N,N+\frac{1}{\sqrt{N}}\right] \text{ and } \xi_2 \in \left[-N+\frac{5}{4\sqrt{N}},-N+\frac{3}{2\sqrt{N}}\right],
        \end{equation}
        \begin{equation}\label{equation situtation 2}
            \xi_{11},\xi_{2} \in \left[N,N+\frac{1}{\sqrt{N}}\right] \text{ and } \xi_{12} \in \left[-N+\frac{5}{4\sqrt{N}},-N+\frac{3}{2\sqrt{N}}\right],
        \end{equation}

        \begin{equation}\label{equation situtation 3}
            \xi_{12},\xi_{2} \in \left[N,N+\frac{1}{\sqrt{N}}\right] \text{ and } \xi_{11} \in \left[-N+\frac{5}{4\sqrt{N}},-N+\frac{3}{2\sqrt{N}}\right].
        \end{equation}
        
        In this setting, we have $|\Phi_1^{u_1}|\sim|\Phi^v|\sim |N|^3$, which follows from  \eqref{equation factorization phi^u} and the fact that under the transformation $(\xi,\xi_1) \mapsto -(\xi_1, \xi)$, the total phase $\Phi^v$ coincides with $\Phi_1^u$. Moreover, $|\Theta|\lesssim 1$ (uniformly in $N$). Thus, for sufficiently small $t$, we have
       \begin{equation}\label{equation estimates integral of exponetials}
            \left|\frac{1}{i\Phi_1^{u_1}}\int_0^t  e^{it'\Theta}  dt' \right| \sim \frac{|t|}{N^3} \quad \text{and} \quad \left| \frac{e^{it\Phi^v}-1}{\Phi_1^{u_1}\Phi^v}\right| \lesssim \frac{1}{N^6}.
       \end{equation}
       
       Consequently, the dominant contribution to $\hat{I}$ comes from $\frac{1}{i\Phi_1^{u_1}}\int_0^t  e^{it'\Theta}  dt' $, which corresponds to the term
       $$\frac{-e^{it\xi^3}}{i}\int_{\xi = \xi_{11}+\xi_{12}+ \xi_2}\frac{\xi_1\xi_2}{\Phi_{1}^{u_1}}\hat{v}_0(\xi_{11})\hat{v}_0(\xi_{12}) \hat{v}_0(\xi_2) \int_0^t e^{it'\Theta}  \,dt'\, d\xi_1\, d\xi_{11}. $$

        Note that this expression remains unchanged across the scenarios in \eqref{equation situtation 1}, \eqref{equation situtation 2}, and \eqref{equation situtation 3}, with the exception of:
        $$\frac{\xi_1\xi_2}{\Phi_1^{u_1}} = \frac{\xi_2}{\xi_{11}^2+\xi_{12}^2-\xi_{11}\xi_{12} -a\xi_1^2 }.$$

        In the case of \eqref{equation situtation 1}, we find
        \begin{equation}\label{equation contribution 1}
            \frac{\xi_1\xi_2}{\Phi_1^{u_1}} = \frac{-N(1+o(1))}{(N^2-4aN^2)(1+o(1))} = -\frac{1}{(1-4a)N}+o\left(\frac{1}{N}\right).
        \end{equation}

        For \eqref{equation situtation 2} and \eqref{equation situtation 3}, the expression becomes
        \begin{equation}\label{equation contribution 2}
            \frac{N(1+o(1))}{3N^2(1+o(1))} = \frac{1}{3N}+ o\left(\frac{1}{N}\right).
        \end{equation}

        To obtain $s\geq -3/4$ in the conclusion, the leading term must satisfy
        $$-\frac{1}{(1-4a)N} + \frac{2}{3N} \neq 0 \Longleftrightarrow a \neq -\frac{1}{8}.$$

        In view of \eqref{equation estimates integral of exponetials}, this implies 
        $$|\hat{I}| \sim \left | \int\frac{\xi_1\xi_2}{\Phi_{1}^{u_1}}\hat{v}_0(\xi_{11})\hat{v}_0(\xi_{12}) \hat{v}_0(\xi_2) \int_0^t e^{it'\Theta}  \,dt' \,d\xi_1\, d\xi_{11} \right | \sim \frac{|t|}{N}  \left | \int\hat{v}_0(\xi_{11})\hat{v}_0(\xi_{12}) \hat{v}_0(\xi_2) \, d\xi_2\, d\xi_{11} \right |, $$
        where the integration is restricted, without loss of generality, to the region defined in \eqref{equation situtation 1}. Since the integrand is of constant sign on this domain, we obtain
        $$|\hat{I}| \gtrsim \frac{1}{N}  \left | \int_{\xi_{11}\in \left[N,N+\frac{1}{2\sqrt{N}}\right]}\hat{v}_0(\xi_{11})\hat{v}_0(\xi_{12}) \hat{v}_0(\xi_2)\,  d\xi_2 \,d\xi_{11} \right |. $$

        The identity $\xi_{12}= \xi-\xi_{11}-\xi_{2}$ imply $ N\leq  \xi_{12} \leq N+\frac{1}{\sqrt{N}}$ assuming $\xi_2 \in \left[-N+\frac{5}{4\sqrt{N}},-N+\frac{3}{2\sqrt{N}}\right]$, $\xi_{11}\in \left[N,N+\frac{1}{2\sqrt{N}}\right]$, and $\xi \in \left [N+ \frac{2}{\sqrt{N}},  N+\frac{9}{4\sqrt{N}}\right]$ . Given that the domain of integration is restricted to that in \eqref{equation situtation 1}, we obtain
        $$\|I\|_{H^s} \gtrsim \left( \int_{N+ \frac{2}{\sqrt{N}}}^{N+ \frac{9}{4\sqrt{N}}}\langle \xi \rangle^{2s} |\hat{I}(\xi)|^2 d\xi \right)^{1/2} \sim N^{s-9/4}.$$

        Furthermore, $\|v_0\| \sim N^{s-1/4}$ by \eqref{equation v0}. Thus, \eqref{inequality continous operator I} yields
        $$N^{s-9/4} \leq N^{3s-3/4} \implies s\geq-\frac{3}{4}. $$
    \end{proof}

    \begin{remark}
        In the case $a=-1/8$, the sum of the contributions from \eqref{equation contribution 1} and \eqref{equation contribution 2} results in a leading term of order $N^{-3/2}$, which only implies $s \geq -1$. However, this conclusion is already established by the $C^2$ regularity of the flow at the origin (see Remark \ref{remark ill-posedness for s<-1}).

    \end{remark}

    By collecting the results of the lemmas in this section, we obtain Theorem \ref{theorem ill-posedness}.

\bibliographystyle{IEEEtranS}

\bibliography{sources.bib}
    
    \end{document}